\documentclass{article}

\usepackage{geometry}                
\usepackage{graphicx}

\usepackage{epstopdf}

\usepackage{enumerate}

\usepackage{amsmath}
\usepackage{amssymb}
\usepackage{amsthm}
\usepackage{url}
\usepackage{color}

\usepackage{tikz-cd}

\usepackage{pifont}
\usepackage{stmaryrd}
\usepackage{dsfont}
\usepackage{textcomp}
\usepackage{verbatim}

\newtheorem{thm}{Theorem}[section]
\newtheorem{prop}[thm]{Proposition}
\newtheorem{lemma}[thm]{Lemma}
\newtheorem{cor}[thm]{Corollary}

\theoremstyle{definition}

\newtheorem{dfn}[thm]{Definition}

\newtheorem{constr}[thm]{Construction}
\newtheorem{ax}[thm]{Axioms}
\newtheorem{ass}[thm]{Assumption}

\newcommand{\xs}{x_1, \dots, x_{n}}

\newcommand{\as}{a_1, \dots, a_{n}}

\newcommand{\rnk}{\mathrm{rank}}

\newcommand{\id}{\mathrm{id}}

\newcommand{\dom}{\mathrm{dom}}

\newcommand{\df}{\mathrm{df}}

\newcommand\blfootnote[1]{%
  \begingroup
  \renewcommand\thefootnote{}\footnote{#1}%
  \addtocounter{footnote}{-1}%
  \endgroup
}

\title{Rank-initial embeddings of non-standard models of set theory}
\author{
Paul K. Gorbow \\
\\
	\small University of Gothenburg \\
	\small Department of Philosophy, Linguistics, and Theory of Science \\
	\small Box 200, 405 30 G\"oteborg, Sweden \\
}
\date{}                                           %

\begin{document}

\maketitle

\begin{abstract}
A theoretical development is carried to establish fundamental results about rank-initial embeddings and automorphisms of countable non-standard models of set theory, with a keen eye for their sets of fixed points. These results are then combined into a ``geometric technique'' used to prove several results about countable non-standard models of set theory.

In particular, back-and-forth constructions are carried out to establish various generalizations and refinements of Friedman's theorem on the existence of rank-initial embeddings between countable non-standard models of the fragment $\mathrm{KP}^\mathcal{P}$ + $\Sigma_1^\mathcal{P}$-Separation of $\mathrm{ZF}$; and Gaifman's technique of iterated ultrapowers is employed to show that any countable model of $\mathrm{GBC} + \textnormal{``$\mathrm{Ord}$ is weakly compact''}$ can be elementarily rank-end-extended to models with well-behaved automorphisms whose sets of fixed points equal the original model. These theoretical developments are then utilized to prove various results relating self-embeddings, automorphisms, their sets of fixed points, strong rank-cuts, and set theories of different strengths. Two examples: The notion of ``strong rank-cut'' is characterized (i) in terms of the theory $\mathrm{GBC} + \textnormal{``$\mathrm{Ord}$ is weakly compact''}$, and (ii) in terms of fixed-point sets of self-embeddings. 
\end{abstract}

\section{Introduction}

\blfootnote{{\bf Acknowlegdement:} This paper presents a portion of the results in the author's Ph.D. Thesis \cite{Gor18}, which was carried out at The University of Gothenburg, under close supervision of Ali Enayat.}
\blfootnote{{\bf MSC 2010:} 03E30, 03C62, 03H99 (primary), 03C15, 03C20, 03E55 (secondary)}
\blfootnote{{\bf Keywords:} Set theory, KP, ZF, GBC, Weakly compact, Nonstandard model, Embedding, Self-embedding, Automorphism, Fixed point, Strong cut, Iterated ultrapower, Recursively saturated}

In \cite{Fri73} Friedman famously invented an ingenious back-and-forth construction to show that {\em every} non-standard countable model of a certain fragment of $\mathrm{ZF}$ (or $\mathrm{PA}$) has a proper self-embedding. He actually proved a more general result, and his technique was later used to prove sharper results on non-standard models of arithmetic under additional assumptions. On a parallel track, Ramsey's theorem was used in \cite{EM56} to show that any first-order theory with an infinite model has a model with a non-trivial automorphism. In particular, there are models of $\mathrm{PA}$, $\mathrm{ZFC}$, etc. with non-trivial automorphisms. Later on, Gaifman refined this technique considerably in the domain of models of arithmetic, showing that any countable model $\mathcal{M}$ of $\mathrm{PA}$ can be elementarily end-extended to a model $\mathcal{N}$ with an automorphism $j : \mathcal{N} \rightarrow \mathcal{N}$ whose set of fixed points is precisely $\mathcal{M}$ \cite{Gai76}. This was facilitated by the technical break-through of iterated ultrapowers, introduced by Gaifman and later adapted by Kunen to a set theoretical setting.

This paper generalizes and refines these two theorems in the domain of non-standard models of set theory, and combines them into a geometric machinery which is used to prove several new results. In the course of this project we also take the opportunity to generalize some more related results from arithmetic to set theory.

These results require a few definitions.

The {\em Takahashi hierarchy}, presented e.g. in \cite{Tak72} (and in Section \ref{Basic logic and model theory} of the present paper), is similar to the well-known L\'evy hierarchy, but any quantifiers of the forms $\exists x \in y, \exists x \subseteq y, \forall x \in y, \forall x \subseteq y$ are considered bounded. $\Delta_0^\mathcal{P}$ is the set of set-theoretic formulae with only bounded quantifiers in that sense, and $\Sigma_n^\mathcal{P}$ and $\Pi_n^\mathcal{P}$ are then defined recursively in the usual way for all $n \in \mathbb{N}$. $\mathrm{KP}^\mathcal{P}$ (presented in Section \ref{tour KP}) is the set theory axiomatized by Extensionality, Pair, Union, Powerset, Infinity, $\Delta_0^\mathcal{P} \textnormal{-Separation}$, $\Delta_0^\mathcal{P} \textnormal{-Collection}$ and $\Pi_1^\mathcal{P} \textnormal{-Foundation}$.

Let us now go through some notions of substructure relevant to set theory (these are explained in more detail in Section \ref{Models of set theory}). Let $\mathcal{M} \models \mathrm{KP}^\mathcal{P}$. A {\em rank-initial} substructure $\mathcal{S}$ of $\mathcal{M}$ is a submodel that is downwards closed in ranks (so if $s \in \mathcal{S}$ and $\mathcal{M} \models \mathrm{rank}(m) \leq \mathrm{rank}(s)$, then $m \in \mathcal{S}$). It is a {\em rank-cut} if, moreover, there is an infinite strictly descending downwards cofinal sequence of ordinals in $\mathcal{M} \setminus \mathcal{S}$. It is a {\em strong} rank-cut if, moreover, for every function $f : \mathrm{Ord}^\mathcal{S} \rightarrow \mathrm{Ord}^\mathcal{M}$ coded in $\mathcal{M}$ (in the sense that $\mathcal{M}$ believes there is a function $\hat f$ whose externalization restricted to $\mathrm{Ord}^\mathcal{S}$ equals $f$), there is an ordinal $\mu \in \mathcal{M} \setminus \mathcal{S}$ such that $f\hspace{2pt}(\xi) \not\in \mathcal{S} \Leftrightarrow f\hspace{2pt}(\xi) > \mu$. Note that these notions for substructures also make sense for embeddings. 

If (the interpretation of the element-relation in) $\mathcal{M}$ is well-founded, then we say that $\mathcal{M}$ is a {\em standard} model, and otherwise we say that it is non-standard. The largest well-founded rank-initial substructure of $\mathcal{M}$ exists. It is called the {\em well-founded part} of $\mathcal{M}$ and is denoted $\mathrm{WFP}(\mathcal{M})$. It turns out that $\mathrm{WFP}(\mathcal{M})$ is a rank-cut of $\mathcal{M}$. 

Suppose that $\mathcal{M}$ is a model of $\mathrm{KP}^\mathcal{P}$ with a proper rank-cut $\mathcal{S}$. If $A \subseteq \mathcal{S}$, then $A$ is {\em coded} in $\mathcal{M}$ if there is $a \in \mathcal{M}$ such that $\{x \in \mathcal{S} \mid \mathcal{M} \models x \in a\} = A$. The {\em standard system} of $\mathcal{M}$ over $\mathcal{S}$, denoted $\mathrm{SSy}_\mathcal{S}(\mathcal{M})$ is the second-order structure obtained by expanding $\mathcal{S}$ with all the subsets of $\mathcal{S}$ coded in $\mathcal{M}$. We define $\mathrm{SSy}(\mathcal{M}) = \mathrm{SSy}_{\mathrm{WFP}(\mathcal{M})}(\mathcal{M})$.

Section \ref{Existence of embeddings between models of set theory} is concerned with proving various existence results for embeddings between countable non-standard models of fragments of $\mathrm{ZFC}$. Friedman showed that for any countable non-standard models $\mathcal{M}$ and $\mathcal{N}$ of 
$$\mathrm{KP}^\mathcal{P} + \Sigma_1 \textnormal{-Separation} + \textnormal{Foundation}\footnote{Here Foundation is the axiom schema $\exists x . \phi(x) \rightarrow \exists y . (\phi(y) \wedge \forall v \in y . \neg \phi(v))$, where $\phi$ ranges over formulae in the language of set theory.},$$ 
and $\mathcal{S}$ such that 
$$\mathcal{S} = \mathrm{WFP}(\mathcal{M}) = \mathrm{WFP}(\mathcal{N}),$$ 
there is a proper rank-initial embedding of $\mathcal{M}$ into $\mathcal{N}$ iff the $\Sigma_1^\mathcal{P}$-theory of $\mathcal{M}$ with parameters in $\mathcal{S}$ is included in the corresponding theory of $\mathcal{N}$ and $\mathrm{SSy}_\mathcal{S}(\mathcal{M}) = \mathrm{SSy}_\mathcal{S}(\mathcal{N})$. 

Theorem \ref{Friedman thm} and Corollary \ref{Friedman cor}, refine Friedman's result in multiple ways. Firstly, we show that it holds for any common rank-cut $\mathcal{S}$ of $\mathcal{M}$ and $\mathcal{N}$ (not just for the standard cut); secondly, we show that it holds for all countable non-standard models of $\mathrm{KP}^\mathcal{P} + \Sigma_1 \textnormal{-Separation}$; thirdly we show that continuum many such embeddings can be obtained (generalizing a result from \cite{Wil73} for models of $\mathrm{PA}$); and fourthly, we show that the embedding can be constructed so as to yield a rank-cut of the co-domain. Moreover, Theorem \ref{Friedman selfembedding} establishes that for every model $\mathcal{M}$ of $\mathrm{KP}^\mathcal{P} + \Sigma_1 \textnormal{-Separation}$ and every rank-cut $\mathcal{S}$ of $\mathcal{M}$, there is a rank-initial topless self-embedding of $\mathcal{M}$ which fixes $\mathcal{S}$ pointwise but moves some element on every rank above $\mathcal{S}$.

Friedman's insight lead to further developments in this direction in the model-theory of arithmetic. In particular, it was established for countable non-standard models of $\mathrm{I}\Sigma_1$. Ressayre showed, conversely, that if $\mathcal{M} \models \mathrm{I}\Sigma_0 + \mathrm{exp}$, and for every $a \in \mathcal{M}$ there is a proper initial self-embedding of $\mathcal{M}$ which fixes every element $\mathcal{M}$-below $a$, then $\mathcal{M} \models \mathrm{I}\Sigma_1$ \cite{Res87b}. Theorem \ref{Ressayre characterization Sigma_1-Separation} is a set theoretic version of this optimality result, to the effect that if $\mathcal{M} \models \mathrm{KP}^\mathcal{P}$, and for every $a \in \mathcal{M}$ there is a proper rank-initial self-embedding of $\mathcal{M}$ which fixes every element that is an $\mathcal{M}$-member of $a$, then $\mathcal{M} \models \mathrm{KP}^\mathcal{P} + \Sigma_1 \textnormal{-Separation}$.

Wilkie showed that for every countable non-standard model $\mathcal{M} \models\mathrm{PA}$ and for every element $a$ of $\mathcal{M}$, there is a proper initial self-embedding whose image includes $a$ \cite{Wil77}. Theorem \ref{Wilkie theorem} and Corollary \ref{Wilkie selfembedding} generalize this result to set theory in refined form, e.g.: For every countable non-standard model $\mathcal{M} \models \mathrm{KP}^\mathcal{P} + \Sigma_2^\mathcal{P}\textnormal{-Separation} + \Pi_2^\mathcal{P}\textnormal{-Foundation}$ and for every element $a$ of $\mathcal{M}$, there is a proper initial self-embedding whose image includes $a$.

Yet another result in this vein is that the isomorphism types of countable recursively saturated models of $\mathrm{PA}$ only depends on the theory and standard system. A generalization of this result for $\mathrm{PA}$ (allowing common $\omega$-topless\footnote{$\omega$-toplessness is a strong notion of toplessness; see Section \ref{Models of set theory} for details.} initial segments to be fixed) was proved in \cite{KK88}. Theorem \ref{rec sat iso thm} is an analogous result for $\mathrm{ZF}$. 

Theorem \ref{Ressayre thm} states that every countable recursively saturated model $\mathcal{M}$ of $\mathrm{ZF}$, with a rank-cut $\mathcal{S}$, has an arbitrarily high rank-initial self-embedding $j$, fixing $\mathcal{S}$ pointwise, such that $j(\mathcal{M}) \prec \mathcal{M}$. This was originally proved in \cite{Res87a}, but we provide a new proof that is conceptually simpler.

The gist of our generalization of Gaifman's result is that any countable model $(\mathcal{M}, \mathcal{A})$ of the theory $\mathrm{GBC} + \text{``$\mathrm{Ord}$ is weakly compact''}$\footnote{Here $\mathrm{GBC}$ is the G\"odel-Bernays theory of classes with the axiom of choice, and ``$\mathrm{Ord}$ is weakly compact'' denotes a formal statement to the effect that every binary class-tree as high as the class of ordinals $\mathrm{Ord}$ has a branch going all the way to the top.} has an elementary rank-end-extension $\mathcal{N}$, such that $\mathrm{SSy}_\mathcal{M}(\mathcal{N}) = (\mathcal{M}, \mathcal{A})$. This result was actually obtained in \cite{Ena04}, but in this paper a more detailed result is proved, see Theorem \ref{Gaifman thm}. 

Once these refined and generalized Friedman- and Gaifman-style results have been established, we combine them in numerous ways to prove a number of new results about non-standard models of set theory. This is carried out in Section \ref{Characterizations}. 

Kirby and Paris essentially showed in \cite{KP77} that any cut $\mathcal{S}$ of a model $\mathcal{M} \models \mathrm{I}\Delta_0$ is strong iff $\mathrm{SSy}_\mathcal{S}(\mathcal{M}) \models \mathrm{ACA}_0$. Theorem \ref{Kirby Paris thm} generalizes this to set theory. It turns out that any rank-cut $\mathcal{S}$ including $\omega^\mathcal{M}$ of an ambient model $\mathcal{M} \models \mathrm{KP}^\mathcal{P} + \mathrm{Choice}$ is strong iff $\mathrm{SSy}_\mathcal{S}(\mathcal{M}) \models \mathrm{GBC} + \textnormal{``$\mathrm{Ord}$ is weakly compact''}$. This result is given a new proof relying on our refined and generalized versions of the Friedman and Gaifman theorems. A similar technique was used in \cite{Ena07} to reprove the result of Kirby and Paris in the context of arithmetic.

Using the above characterization of strong rank-cuts we show (Theorem \ref{characterize strongly topless substructure}) that for any countable model $\mathcal{M} \models \mathrm{KP}^\mathcal{P}$, and any rank-cut $\mathcal{S}$ of $\mathcal{M}$: there is a self-embedding of $\mathcal{M}$ whose set of fixed points is precisely $\mathcal{S}$ iff $\mathcal{S}$ is a $\Sigma_1^\mathcal{P}$-elementary  strong rank-cut of $\mathcal{M}$. In \cite{BE18} the analogous result is shown for models of the fragment $\mathrm{I}\Sigma_1$ of $\mathrm{PA}$. 

The latter result of Bahrami and Enayat was inspired by an analogous result in the context of countable recursively saturated models of $\mathrm{PA}$ \cite{KKK91}, namely that any cut of such a model is the fixed point set of an automorphism iff it is an elementary strong cut. Theorem \ref{characterize strongly topless substructure of rec sat} generalizes this result to set theory by means of a new proof, again relying on a combination of our Friedman- and Gaifman-style theorems. It is shown that for any rank-cut $\mathcal{S}$ of a countable recursively saturated model of $\mathrm{ZFC} + V = \mathrm{HOD}$:\footnote{$V = \mathrm{HOD}$ is the statement that every set is hereditarily ordinal definable; see Section \ref{ZFC and GBC} for details.} $\mathcal{S}$ is the fixed point set of an automorphism of $\mathcal{M}$ iff it is an elementary strong rank-cut.

Finally, we combine the Friedman- and Gaifman-style theorems to show (Theorem \ref{strongly topless self-embedding iff GBC weakly compact}) that for any countable non-standard $\mathcal{M} \models \mathrm{KP}^\mathcal{P} + \Sigma_1^\mathcal{P} \textnormal{-Separation} + \textnormal{Choice}$: $\mathcal{M}$ has a strong rank-cut isomorphic to $\mathcal{M}$ iff $\mathcal{M}$ expands to a model of $\mathrm{GBC} + \textnormal{``}\mathrm{Ord}$ is weakly compact''.

\section{Motivation}

It is a common theme throughout mathematics to study structures and how these structures relate to each other. Usually structures are related to each other by functions from one structure (the domain) to another (the co-domain), which preserve some of the structure involved. Since the study of such functions has turned out to be very fruitful in many branches of mathematics, it makes sense to apply this methodology to models of set theory as well.

When we consider models of such expressive theories as set theories, it is natural to compare structures by means of embeddings. Any embedding exhibits the domain as a substructure of the co-domain, but we can ask various questions about ``how nicely'' the domain can be embedded in the co-domain: Firstly, for any first-order structure we can ask if the embedding is elementary, i.e. whether the truth of every first-order sentence with parameters in the domain is preserved by the embedding. Secondly, for structures of set theory we can ask whether the domain is embedded ``initially'' in the co-domain. For set theory, the intuition of ``initiality'' may be captured by various different formal notions, of different strengths (see Section \ref{Models of set theory}). The weakest notion of this form is called {\em initiality} and requires simply that the image of the embedding is downwards closed under $\in$, i.e. if $b$ is in the image, and the co-domain satisfies that $c \in b$, then $c$ is also in the image. This paper is concerned with {\em rank-initial} embeddings, defined by the stronger property that for every value of the embedding, every element of the co-domain of rank less than or equal to the rank of that value is also a value of the embedding. As noted in Section \ref{tour KP}, $\mathrm{KP}^\mathcal{P}$ proves that the function $(\alpha \mapsto V_\alpha)$ is total on the ordinals, so it makes sense to consider a notion of embedding which preserves this structure $(\alpha \mapsto V_\alpha)$. And indeed, an embedding $i : \mathcal{M} \rightarrow \mathcal{N}$ is rank-initial iff $i$ is initial and $i(V_\alpha^{\hspace{2pt}\mathcal{M}}) = V_{i(\alpha)}^{\hspace{2pt}\mathcal{N}}$, for every ordinal $\alpha$ in $\mathcal{M}$ (see Corollary \ref{rank-init equivalences in KPP}), so the choice to study the notion of rank-initial embedding is quite a natural in the setting of $\mathrm{KP}^\mathcal{P}$.

Between well-founded structures, all initial embeddings are trivial: This follows from the Mostowski collapse theorem (see Theorem \ref{Mostowski collapse} and Proposition \ref{emb pres}(\ref{emb pres fixed})). In particular, if $i : \mathcal{M} \rightarrow \mathcal{N}$ is an initial embedding between well-founded extensional structures, then $\mathcal{M}$ and $\mathcal{N}$ are isomorphic to transitive sets (with the inherited $\in$-structure) $\mathcal{M}'$ and $\mathcal{N}'$, respectively, and $i$ is induced by the inclusion function of $\mathcal{M}'$ into $\mathcal{N}'$. So for a study of initial embeddings of models of set theory to yield any insight, we must turn our attention to non-standard models. As explained in the introduction, for non-standard models of arithmetic, several interesting results have been obtained that are either directly about initial embeddings between such models, or are proved by means of considering such embeddings. Thus a motivation for this work is to determine whether these generalize to the set theoretic setting, and if so, for which particular set theory. We give positive answers for suitable extensions of $\mathrm{KP}^\mathcal{P}$.

A deeper motivation lies in combining the techniques of Friedman and Gaifman into a geometric machinery. Several results and proofs of Section \ref{Characterizations} testify to the versatility of this approach, where relationships are established between the theory $\mathrm{GBC} + \textnormal{``$\mathrm{Ord}$ is weakly compact''}$, strong cuts, rank-initial embeddings, and fixed point sets of rank-initial embeddings.

\section{Basic logic and model theory}\label{Basic logic and model theory}

This section contains basic material on logic and model theory typically found in introductory textbooks such as \cite{CK90}. An expanded version of this section with proofs of the results is found as \S 4.1 of \cite{Gor18}.

We work with the usual first-order logic. A {\em signature} is a set of constant, function and relation symbols. The {\em language} of a signature is the set of well-formed formulas in the signature. The arity of function symbols, $f$, and relation symbols, $R$, are denoted $\mathrm{arity}(f\hspace{2pt})$ and $\mathrm{arity}(R)$, respectively. Models in a language are written as $\mathcal{M}$, $\mathcal{N}$, etc. They consist of interpretations of the symbols in the signature; for each symbol $S$ in the signature, its interpretation in $\mathcal{M}$ is denoted $S^\mathcal{M}$. If $X$ is a term, relation or function definable in the language over some theory under consideration, then $X^\mathcal{M}$ denotes its interpretation in $\mathcal{M}$. 

The domain of $\mathcal{M}$ is also denoted $\mathcal{M}$, so $a \in \mathcal{M}$ means that $a$ is an element of the domain of $\mathcal{M}$. Finite tuples are written as $\vec{a}$, and the tuple $\vec{a}$ considered as a set (forgetting the ordering of the coordinates) is also denoted $\vec{a}$. Moreover, $\vec{a} \in \mathcal{M}$ means that each coordinate of $\vec{a}$ is an element of the domain of $\mathcal{M}$. $\mathrm{length}(\vec{a})$ denotes the number of coordinates in $\vec{a}$. For each natural number $k \in \{ 1, \dots, \mathrm{length}(\vec{a}) \}$, $\pi_k(\vec{a})$ is the $k$-th coordinate of $\vec{a}$. When a function $f : A \rightarrow B$ is applied as $f\hspace{2pt}(\vec{a})$ to a tuple $\vec{a} \in A^n$, where $n \in \mathbb{N}$, then it is evaluated coordinate-wise, so $f\hspace{2pt}(\as) = (f\hspace{2pt}(a_1), \dots, f\hspace{2pt}(a_n))$.
If $\Gamma$ is a set of formulae in a language and $n \in \mathbb{N}$, then $\Gamma[x_1, \dots, x_n]$ denotes the subset of $\Gamma$ of formulae all of whose free variables are in $\{x_1, \dots, x_n\}$.

The {\em theory} of a model $\mathcal{M}$, denoted $\mathrm{Th}(\mathcal{M})$, is the set of formulae in the language satisfied by $\mathcal{M}$. If $\Gamma$ is a subset of the language and $S \subseteq \mathcal{M}$, then 
\[\mathrm{Th}_{\Gamma, S}(\mathcal{M}) =_\df \{ \phi(\vec{s}) \mid \phi \in \Gamma \wedge \vec{s} \in S \wedge (\mathcal{M}, \vec{s}) \models \phi(\vec{s}) \}.\]

The standard model of arithmetic is denoted $\mathbb{N}$.

$\mathcal{L}^0$ is the language of set theory,  i.e. the set of all well-formed formulae generated by $\{\in\}$.

$\mathcal{L}^1$ is defined as a two-sorted language in the single binary relation symbol $\{\in\}$; we have a sort $\mathsf{Class}$ of classes (which covers the whole domain and whose variables and parameters are written in uppercase $X, Y, Z, A, B, C,$ etc.) and a sort $\mathsf{Set}$ of sets (which is a subsort of $\mathsf{Class}$ and whose variables and parameters are written in lowercase $x, y, z, a, b, c,$ etc.). The relation $\in$ is a predicate on the derived sort $\mathsf{Set} \times \mathsf{Class}$. 

Models in $\mathcal{L}^1$ are usually written in the form $(\mathcal{M}, \mathcal{A})$, where $\mathcal{M}$ is an $\mathcal{L}^0$-structure on the domain of sets, and $\mathcal{A}$ is a set of classes. It is sometimes convenient to regard an $\mathcal{L}^1$-structure $(\mathcal{M}, \mathcal{A})$ simply as its reduct $\mathcal{M}$ to the language $\mathcal{L}^0$; for example, if $(\mathcal{M}, \mathcal{A})$ is an $\mathcal{L}^1$-structure, then (unless otherwise stated), by {\em an element of} $(\mathcal{M}, \mathcal{A})$, is meant an element of sort $\mathsf{Set}$.

The notions of substructure, embedding and isomorphism are defined in the usual way. An embedding is {\em proper} if it is not onto. We write $\mathcal{M} \cong \mathcal{N}$ if $\mathcal{M}$ and $\mathcal{N}$ are isomorphic. If $\mathcal{S}$ is a common subset of $\mathcal{M}$ and $\mathcal{N}$, and there is an isomorphism between $\mathcal{M}$ and $\mathcal{N}$ fixing $\mathcal{S}$ pointwise, then we write $\mathcal{M} \cong_\mathcal{S} \mathcal{N}$.

An embedding $f : \mathcal{M} \rightarrow \mathcal{N}$ of $\mathcal{L}$-structures is $\Gamma${\em -elementary}, for some $\Gamma \subseteq \mathcal{L}$, if for each formula $\phi(\vec{x})$ in $\Gamma$, and for each $\vec{m} \in \mathcal{M}$, 
\[
\mathcal{M} \models \phi(\vec{m}) \Leftrightarrow \mathcal{N} \models \phi(f\hspace{2pt}(\vec{m})) .
\]
If there is such an embedding we write $\mathcal{M} \preceq_\Gamma \mathcal{N}$. $f$ is {\em elementary} if it $\mathcal{L}$-elementary.

It is also of interest to consider partial embeddings. $\llbracket \mathcal{M} \preceq_{\Gamma, \mathcal{S}}^{< \omega} \mathcal{N} \rrbracket$ denotes the set of partial functions $f$ from $\mathcal{M}$ to $\mathcal{N}$, with finite domain, fixing $\mathcal{S}$ pointwise, and such that for all $\phi(\vec{x}) \in \Gamma$ and for all $\vec{m} \in \mathcal{M}$,
\[
\mathcal{M} \models \phi(\vec{m}) \Rightarrow \mathcal{N} \models \phi(f\hspace{2pt}(\vec{m})) .
\]
PLEASE NOTE!: This definition uses `$\Rightarrow$', as opposed to the `$\Leftrightarrow$' used in the definition of `$\preceq$'. 

If $\Gamma$ is omitted, then it is assumed to be $\mathcal{L}^0$, and if $\mathcal{S}$ is omitted, then it is assumed to be $\varnothing$. Let $\mathbb{P} = \llbracket \mathcal{M} \preceq_{\Gamma, \mathcal{S}}^{<\omega} \mathcal{N} \rrbracket$. We endow $\mathbb{P}$ with the following partial order. For any $f,g \in \mathbb{P}$,
\[
f \leq^\mathbb{P} g \Leftrightarrow f\restriction_{\dom(g)} = g.
\]

In Section \ref{Models of set theory} we will introduce definitions for more types of embeddings that are relevant to the study of models of set theory.

The {\em uniquely existential quantifier} $\exists! x . \phi(x)$ is defined as $\exists x . (\phi(x) \wedge (\forall y . \phi(y) \rightarrow y = x)).$ The {\em bounded quantifiers}, $\forall u \in y . \phi(u, y)$ and $\exists u \in y . \phi(u, y)$, are defined as $\forall u . (u \in y \rightarrow \phi(u, y))$ and $\exists u . (u \in y \wedge \phi(u, y))$, respectively. $\hat\Delta_0 \subseteq \mathcal{L}^0$ is the set of formulae all of whose quantifiers are bounded. $\hat\Sigma_0$ and $\hat\Pi_0$ are defined as equal to $\hat\Delta_0$. Recursively, for every $n \in \mathbb{N}$: $\hat\Sigma_{n+1} \subseteq \mathcal{L}^0$ is the set of formulae of the form $\exists x . \phi$, where $\phi$ is in $\hat\Pi_n$; and dually, $\hat\Pi_{n+1} \subseteq \mathcal{L}^0$ is the set of formulae of the form $\forall x . \phi$, where $\phi$ is in $\hat\Sigma_n$. 
For each $n \in \mathbb{N}$, $\hat{\mathrm{B}}_n$ is defined as the closure of $\hat \Sigma_n$ under Boolean connectives and bounded quantifiers.

Suppose a background $\mathcal{L}^0$-theory $T$ is given. For each $n \in \mathbb{N}$ and each symbol `$\Gamma$' $\in$ $\{$ `$\Sigma$', `$\Pi$', `$\mathrm{B}$'$\}$, $\Gamma_n$ is defined as the set of formulae provably equivalent (in $T$) to a formula in $\hat{\Gamma}_n$. Moreover, for each $n \in \mathbb{N}$, $\Delta_{n} =_\df \Sigma_{n} \cap \Pi_{n}$. These sets of formulae are collectively called {\em the L\'evy hierarchy}, and we say that they measure a formula's {\em L\'evy complexity}. This hierarchy is developed in \cite{Lev65}.

If $\phi$ is an $\mathcal{L}^0$-formula and $t$ is an element of a model or an $\mathcal{L}^0$-term, such that none of the variables of $t$ occur in $\phi$, then $\phi^t$ denotes the formula obtained from $\phi$ by replacing each quantifier of the form `$\boxminus x$' by `$\boxminus x \in t$', where $\boxminus \in \{\exists, \forall\}$.

The $\mathcal{P}$-{\em bounded quantifiers} $\forall x \subseteq y . \phi(x, y)$ and $\exists x \subseteq y . \phi(x, y)$ are defined as $\forall x . (x \subseteq y \rightarrow \phi(x, y))$ and $\exists x . (x \subseteq y \wedge \phi(x, y))$, respectively. For each $n \in \mathbb{N}$, we define sets $\hat\Sigma^\mathcal{P}_n$, $\hat\Pi^\mathcal{P}_n$, $\hat\Delta^\mathcal{P}_n$, $\hat{\mathrm{B}}^\mathcal{P}_n$, $\Sigma^\mathcal{P}_n$, $\Pi^\mathcal{P}_n$, $\Delta^\mathcal{P}_n$ and $\mathrm B^\mathcal{P}_n$ analogously as above, but replacing ``bounded'' by ``bounded or $\mathcal{P}$-bounded''. These sets of formulae are called {\em the Takahashi hierarchy}, and we say that they measure a formula's {\em Takahashi complexity}. Many facts about this hierarchy (in the context of $\mathrm{ZFC}$) are established in \cite{Tak72}. It appears like these results also hold in the context of $\mathrm{KP}^\mathcal{P}$ (apart from its Theorem 6, which might require $\mathrm{KP}^\mathcal{P} + \textnormal{Choice}$).

When a set of formulae is denoted with a name that includes free variables, for example $p(\vec{x})$, then it is assumed that each formula in the set has at most the free variables $\vec{x}$. Moreover, if $\vec{a}$ are terms or elements of a model, then $p(\vec{a}) = \{\phi(\vec{a}) \mid \phi(\vec{x}) \in p(\vec{x})\}$.  

A {\em type} $p(\vec{x})$ {\em over a theory} $T$ or {\em over a model} $\mathcal{M}$, are defined in the usual way. Given a tuple $\vec{a} \in \mathcal{M}$, a subset $\Gamma \subseteq \mathcal{L}$ and a subset $S \subseteq \mathcal{M}$, {\em the $\Gamma$-type of $\vec{a}$ over $\mathcal{M}$ with parameters in $S$} is the set $\{\phi(\vec{x}, \vec{b}) \mid \phi \in \Gamma \wedge \vec{b} \in S \wedge \mathcal{M} \models \phi(\vec{a}, \vec{b})\}$, denoted $\mathrm{tp}_{\Gamma, S}(\vec{a})$. 

We fix a G\"odel numbering of the syntactical objects in the languages $\mathcal{L}^0$ and $\mathcal{L}^1$, and assume that any syntactic object in these languages equals its G\"odel number. A type $p(\vec{x}, \vec{b})$ over $\mathcal{M}$ is {\em recursive} if $\{ \ulcorner \phi(\vec{x}, \vec{y}) \urcorner  \mid \phi(\vec{x}, \vec{b}) \in p(\vec{x}, \vec{b})\}$ is a recursive set, where $\ulcorner \phi(\vec{x}, \vec{y}) \urcorner$ denotes the G\"odel code of $\phi(\vec{x}, \vec{y})$ (henceforth formulae will be identified with their G\"odel codes). $\mathcal{M}$ is {\em recursively $\Gamma$-saturated} if it realizes every recursive $\Gamma$-type over $\mathcal{M}$. 

The notions of {\em poset} (or {\em partial order}) and {\em linear order} are defined in the usual way. 

Let $\mathbb{P}$ be a poset. We say that a formula $\phi(x)$ {\em holds for unboundedly many} $x \in \mathbb{P}$, if for any $a \in \mathbb{P}$, there is $b >^\mathbb{P} a$ such that $\phi(b)$. 

An {\em embedding $i : \mathbb{P} \rightarrow \mathbb{P}'$ of posets}, is just a special case of embeddings of structures, i.e. it is an embedding of $\{\leq\}$-structures. Let $i : \mathbb{P} \rightarrow \mathbb{P}'$ be an embedding of posets. $y \in \mathbb{P}'$ is an {\em upper bound} of $i$ if $\forall x \in \mathbb{P} . i(x) < y$. If such a $y$ exists then $i$ is {\em bounded above}. $i$ is {\em topless} if it is bounded above but does not have a $\mathbb{P}'$-least upper bound.

A self-embedding $i : \mathbb{P} \rightarrow \mathbb{P}$ is {\em contractive} if for all $x \in \mathbb{P}$, we have $i(x) <_\mathbb{P} x$. 

Let $\mathbb{P}$ be a poset. Given $x \in \mathbb{P}$, define $\mathbb{P}_{\leq x}$ as the substructure of $\mathbb{P}$ on $\{y \in \mathbb{P} \mid y \leq_\mathbb{P} x\}$; and similarly, if $X \in \mathbb{P}$, define $\mathbb{P}_{\leq X}$ as the substructure of $\mathbb{P}$ on $\{y \in \mathbb{P} \mid \exists x \in X . y \leq_\mathbb{P} x\}$. We have analogous definitions for when `$\leq$' is replaced by `$<$', `$\geq$' or `$>$'.

For any ordinal $\alpha$ and linearly ordered set $(\mathbb{L}, <^\mathbb{L})$, the lexicographic ordering on the set $\mathbb{L}^{<\alpha}$ is defined in the usual way and denoted $<^\mathrm{lex}$.

Let $\mathbb{P}$ be a poset. A subset $\mathcal{D} \subseteq \mathbb{P}$ is {\em dense} if for any $x \in \mathbb{P}$ there is $y \in \mathcal{D}$ such that $y \leq x$. A {\em filter} $\mathcal{F}$ on $\mathbb{P}$ is a non-empty subset of $\mathbb{P}$, such that $\forall x, y \in \mathbb{P} . ((x \in \mathcal{F} \wedge x \leq y) \rightarrow y \in \mathcal{F})$ ({\em upwards closed}) and $\forall x, y \in \mathcal{F} . \exists z \in \mathcal{F} . (z \leq x \wedge z \leq y)$ ({\em downwards directed}). A filter $\mathcal{F}$ is an {\em ultrafilter} if it is {\em maximal}, i.e. if there is no filter $\mathcal{F}'$ on $\mathbb{P}$ such that $\mathcal{F} \subsetneq \mathcal{F}'$. Let $\mathbf{D}$ be a set of dense subsets of $\mathbb{P}$. A filter $\mathcal{F}$ is $\mathbf{D}${\em -generic}, if $\forall \mathcal{D} \in \mathbf{D} . \mathcal{D} \cap \mathcal{F} \neq \varnothing$.

\begin{lemma}\label{generic filter existence}
	Let $\mathbb{P}$ be a poset with an element $p$. If $\mathbf{D}$ is a countable set of dense subsets of $\mathbb{P}$, then there is a $\mathbf{D}$-generic filter $\mathcal{F}$ on $\mathbb{P}$ containing $p$.
\end{lemma}

\begin{lemma}\label{ultrafilter existence}
	Let $\mathbb{P}$ be a poset and let $\mathcal{F}$ be a filter on $\mathbb{P}$. There is an ultrafilter $\mathcal{U}$ such that $\mathcal{F} \subseteq \mathcal{U}$.
\end{lemma}

\section{Power Kripke-Platek set theory}\label{tour KP}

The set theory $\mathrm{KP}^\mathcal{P}$ may be viewed as the natural extension of Kripke-Platek set theory $\mathrm{KP}$ ``generated'' by adding the Powerset axiom.

\begin{ax}[{\em Power Kripke-Platek set theory}, $\mathrm{KP}^\mathcal{P}$] $\mathrm{KP}^\mathcal{P}$ is the $\mathcal{L}^0$-theory given by these axioms and axiom schemata:
	\[
	\begin{array}{ll}
	\textnormal{Extensionality} & \forall x . \forall y . ((\forall u . u \in x \leftrightarrow u \in y) \rightarrow x = y) \\
\textnormal{Pair} & \forall u . \forall v . \exists x . \forall w . (w \in x \leftrightarrow (w = u \vee w = v)) \\
\textnormal{Union} & \forall x . \exists u . \forall r . (r \in u \leftrightarrow \exists v \in x . r \in v) \\
\textnormal{Powerset} & \forall u . \exists x . \forall v . (v \in x \leftrightarrow v \subseteq u) \\
\textnormal{Infinity} & \exists x . (\varnothing \in x \wedge \forall u \in x . \{u\} \in x) \\
\Delta_0^\mathcal{P} \textnormal{-Separation} & \forall x . \exists y . \forall u . (u \in y \leftrightarrow (u \in x \wedge  \phi(u))) \\
\Delta_0^\mathcal{P} \textnormal{-Collection} & \forall x . (\forall u \in x . \exists v . \phi(u, v) \rightarrow \exists y . \forall u \in x . \exists v \in y . \phi(u, v)) \\
\Pi_1^\mathcal{P} \textnormal{-Foundation} & \exists x . \phi(x) \rightarrow \exists y . (\phi(y) \wedge \forall v \in y . \neg \phi(v)) \\
	\end{array}
	\]
Above and in the following, if a schema named $\Gamma$-[name] is specified by a formula involving a meta-variable for a subformula (e.g. $\phi$ above), then this meta-variable ranges over $\Gamma$. If $\Gamma$ is omitted from such a name, then the formula ranges over $\mathcal{L}^0$.
\end{ax}

We also consider these axioms and schemata:
\[
\begin{array}{ll}
\textnormal{Strong }\Gamma\textnormal{-Collection} & \forall x . \exists y . \forall u \in x . (\exists v . \phi(u, v) \rightarrow \exists v' \in y . \phi(u, v'))\\
\Gamma \textnormal{-Replacement} & 
\forall x . (\forall u \in x . \exists! v . \phi(u, v) \rightarrow \exists y . \forall v . (v \in y \leftrightarrow \exists u \in x . \phi(u, v)))\\
\textnormal{Choice} &
\forall x . ((\forall u \in x . u \neq \varnothing) \rightarrow \exists f : x \rightarrow \bigcup x . \forall u \in x . f\hspace{2pt}(u) \in u).
\end{array}
\]

Apart from adding the Powerset axiom, $\mathrm{KP}^\mathcal{P}$ differs from $\mathrm{KP}$ in that the schemata of Separation, Collection and Foundation are extended to broader sets of formulae, using the Takahashi hierarchy instead of the L\'evy hierarchy. ($\mathrm{KP}$ has $\Delta_0 \textnormal{-Separation}$, $\Delta_0 \textnormal{-Collection}$ and $\Pi_1 \textnormal{-Foundation}$.) Since 
\[
\mathcal{P}(u) = y \Leftrightarrow (\forall v \subseteq u . v \in y) \wedge \forall v \in y . \forall r \in v . r \in u),
\]
the Takahashi hierarchy treats the powerset operation as a bounded operation. It is in this sense that $\mathrm{KP}^\mathcal{P}$ is ``generated'' from $\mathrm{KP}$ by adding powersets.

The theory $\mathrm{KP}$ has received a great deal of attention, because of its importance to G\"odel's $L$ (the hierarchy of constructible sets), definability theory, recursion theory and infinitary logic. The ``bible'' on this subject is \cite{Bar75}. The main sources on $\mathrm{KP}^\mathcal{P}$ seem to be the papers by Friedman and Mathias that are discussed later on in this section.

There is of course much to say about what can and cannot be proved in these theories. $\mathrm{KP}$ proves $\Delta_1$-Separation, $\Sigma_1$-Collection and $\Sigma_1$-Replacement. $\mathrm{KP}^\mathcal{P}$ proves $\Delta_1^\mathcal{P}$-Separation, $\Sigma_1^\mathcal{P}$-Collection and $\Sigma_1^\mathcal{P}$-Replacement. Both theories enjoy a decent recursion theorem. In $\mathrm{KP}$ we have $\Sigma_1$-Recursion and in $\mathrm{KP}^\mathcal{P}$ we have $\Sigma_1^\mathcal{P}$-Recursion. $\Gamma$-Recursion is the statement in the meta-theory (in relation to an object base theory) that for any $\psi(x, y) \in \Gamma$ provably defining a total function $G$ on the universe, there is $\phi(x, y) \in \Gamma$ provably defining a a total function $F$ on the universe, such that provably $\forall x . F(x) = G(F\restriction_x)$.

$\Sigma_1$-Recursion is quite important in that it enables development of first-order logic and model theory. Moreover, it enables $\mathrm{KP}$ to prove the totality of the rank-function, but (in the absence of Powerset) it is not sufficient to establish that the function $\alpha \mapsto V_\alpha$ is total on the ordinals. However, thanks to $\Sigma_1^\mathcal{P}$-Recursion, $\mathrm{KP}^\mathcal{P}$ does prove the latter claim, and this is needed for certain arguments in Section \ref{Existence of embeddings between models of set theory}, particularly in the proof of our Friedman-style embedding theorem. Also of interest, though not used in this paper, is that neither of the theories proves the existence of an uncountable ordinal. (This may be seen from the short discussion about the Church-Kleene ordinal later on in the section.) However, $\mathrm{KP}^\mathcal{P}$ augmented with the axiom of Choice proves the existence of an uncountable ordinal, essentially because Choice gives us that $\mathcal{P}(\omega)$ can be well-ordered.

There is also a philosophical reason for considering $\mathrm{KP}$ and $\mathrm{KP}^\mathcal{P}$, in that they encapsulate a more parsimonious ontology of sets than $\mathrm{ZF}$. If $a$ is an element of an $\mathcal{L}^0$-structure $\mathcal{M}$, then $a_\mathcal{M}$ denotes the set $\{x \in \mathcal{M} \mid \mathcal{M} \models x \in a\}$. A set $x$ is {\em transitive} if $\forall u \in x . u \subseteq x$. It is {\em supertransitive} if additionally $\forall u \in x . \forall v \subseteq u . v \in x$. The {\em transitive closure} $\mathrm{TC}(x)$ of a set $x$ is its closure under elements of elements, i.e. the least superset of $x$ such that $\forall u \in \mathrm{TC}(x) . \forall r \in u . r \in \mathrm{TC}(x)$. The {\em supertransitive closure} $\mathrm{STC}(x)$ of a set $x$ is its closure under elements of elements and subsets of elements, i.e. the least superset of $x$ such that $\forall u \in \mathrm{STC}(x) . \forall r \in u . r \in \mathrm{STC}(x)$ and $\forall u \in \mathrm{STC}(x) . \forall v \subseteq u . v \in \mathrm{STC}(x)$. $\mathrm{KP}$ proves that $\mathrm{TC}(x)$ exists for all $x$, and $\mathrm{KP}^\mathcal{P}$ proves that $\mathrm{STC}(x)$ exists for all $x$.

If $\mathcal{M}$ is a model of $\mathrm{KP}$, $a$ is an element of  $\mathcal{M}$, and $\phi(x)$ is a $\Delta_0$-formula of set theory, then it is straightforward to show that 
\[
\mathcal{M} \models \phi(a) \Leftrightarrow (\mathrm{TC}^\mathcal{M}(\{a\}))_\mathcal{M} \models \phi(a).
\]
The reason for this equivalence is that a quantifier in a $\Delta_0$-formula can only range over a subset of the transitive closure of $\{a\}$. 

Similarly, if $\mathcal{M}$ is a model of $\mathrm{KP}^\mathcal{P}$, $a$ is an element of  $\mathcal{M}$, and $\phi(x)$ is a $\Delta_0^\mathcal{P}$-formula of set theory, then 
\[
\mathcal{M} \models \phi(a) \Leftrightarrow (\mathrm{STC}^\mathcal{M}(\{a\}))_\mathcal{M} \models \phi(a).
\]

Thus, the Separation and Collection schemata of $\mathrm{KP}$ and $\mathrm{KP}^\mathcal{P}$ only apply to formulae whose truth depends exclusively on the part of the model which is ``below'' the parameters and free variables appearing in the formula (in the respective senses specified above). 

Friedman's groundbreaking paper \cite{Fri73}, established several important results in the model theory of $\mathrm{KP}^\mathcal{P}$. Section \ref{Existence of embeddings between models of set theory} is concerned with generalizing and refining one of these results (Theorem 4.1 of that paper), as well as related results. In its simplest form, this result is that every countable non-standard model of $\mathrm{KP}^\mathcal{P} + \Sigma_1^\mathcal{P} \textnormal{-Separation}$ has a proper self-embedding. 

A second important result of Friedman's paper (its Theorem 2.3) is that every countable standard model of $\mathrm{KP}^\mathcal{P}$ is the well-founded part of a non-standard model of $\mathrm{KP}^\mathcal{P}$.

Thirdly, let us also consider Theorem 2.6 of Friedman's paper. This theorem says that any countable model of $\mathrm{KP}$ can be extended to a model of $\mathrm{KP}^\mathcal{P}$ with the same ordinals. The {\em ordinal height} of a standard model of $\mathrm{KP}$ is the ordinal representing the order type of the ordinals of the model. An ordinal is said to be {\em admissible} if it is the ordinal height of some model of $\mathrm{KP}$. This notion turns out to be closely connected with recursion theory. For example, the first admissible ordinal is the {\em Church-Kleene ordinal} $\omega_1^{\mathrm{CK}}$, which may also be characterized as the least ordinal which is not order-isomorphic to a recursive well-ordering of the natural numbers. So in particular, Friedman's theorem shows that every countable admissible ordinal is also the ordinal height of some model of $\mathrm{KP}^\mathcal{P}$. 

Another important paper on $\mathrm{KP}^\mathcal{P}$ is Mathias's \cite{Mat01}, which contains a large body of results on weak set theories. See its Section 6 for results on $\mathrm{KP}^\mathcal{P}$. One of many results established there is its Theorem 6.47, which shows that $\mathrm{KP}^\mathcal{P} + V = L$ proves the consistency of $\mathrm{KP}^\mathcal{P}$, where $V = L$ is the statement that every set is G\"odel constructible.

From the perspective of this paper, the main $\mathrm{KP}^\mathcal{P}$-style set theory of interest is $\mathrm{KP}^\mathcal{P} + \Sigma_1^\mathcal{P} \textnormal{-Separation}$, because it is for non-standard countable models of this theory that Friedman's embedding theorem holds universally. Returning to the philosophical discussion on parsimony above, this theory of affirms the existence of more sets, as arising from applying Separation to $\Sigma_1^\mathcal{P}$-formulae. The truth of such formulae can be given a similar characterization as we gave for $\Delta_0^\mathcal{P}$-formulae: If $\mathcal{M} \models \mathrm{KP}^\mathcal{P}$, $p \in \mathcal{M}$ and $\delta(x, y) \in \Delta_0^\mathcal{P}$, then 
\[
\mathcal{M} \models \exists x . \delta(x, p) \Leftrightarrow \exists a \in \mathcal{M} . \big[ (\mathrm{STC}^\mathcal{M}(\{a, p\}))_\mathcal{M} \models \delta(a, p) \big].
\]

\begin{prop} Useful relationships between axioms to be considered:
\begin{enumerate}[{\rm (a)}]
\item\label{KPP-strong} $\mathrm{KP}^\mathcal{P} \vdash \Delta_1^\mathcal{P}\textnormal{-Separation}, \Sigma_1^\mathcal{P}\textnormal{-Collection}, \Sigma_1^\mathcal{P}\textnormal{-Replacement}.$
\item $\mathrm{KP}^\mathcal{P} \vdash $ ``the function $\alpha \mapsto V_\alpha$ is total on the ordinals''.
\item $\mathrm{KP}^\mathcal{P} + \Sigma_1^\mathcal{P} \textnormal{-Separation} \vdash \textnormal{Strong }\Sigma_1^\mathcal{P}\textnormal{-Collection}$.
\item For each $n \in \mathbb{N}$, $\mathrm{KP}^\mathcal{P} \vdash \Sigma_n^\mathcal{P} \textnormal{-Separation} \leftrightarrow \mathrm B_n^\mathcal{P} \textnormal{-Separation}.$
\end{enumerate}
\end{prop}
\begin{proof}
For (\ref{KPP-strong}) we refer to \cite{Mat01}. The others are routine.
\end{proof}

We shall now establish some basic closure properties in the Takahashi hierarchy. These will often be used without explicit reference to this proposition.

\begin{prop}\label{tak_clo} 
Take $\mathrm{KP}^\mathcal{P}$ as base theory. 
\begin{enumerate}[{\rm (a)}]
\item\label{tak_clo_D1} $\Delta_1^\mathcal{P}$ is closed under boolean connectives. Moreover, if $f$ is a $\Sigma_1^\mathcal{P}$-function and $\phi \in \Delta_1^\mathcal{P}$, then the formulae $\forall x \in f(y) . \phi$, $\exists x \in f(y) . \phi$, $\forall x \subseteq f(y) . \phi$ and $\exists x \subseteq f(y) . \phi$ are all in $\Delta_1^\mathcal{P}$.
\item\label{tak_clo_S1} $\Sigma_1^\mathcal{P}$ is closed under disjunction and conjunction. Moreover, if $f$ is a $\Sigma_1^\mathcal{P}$-function and $\phi \in \Sigma_1^\mathcal{P}$, then the formulae $\forall x \in f(y) . \phi$, $\exists x \in f(y) . \phi$, $\forall x \subseteq f(y) . \phi$ and $\exists x \subseteq f(y) . \phi$ are all in $\Sigma_1^\mathcal{P}$.
\item If $f$ is a $\Sigma_1^\mathcal{P}$-function, then it is a $\Delta_1^\mathcal{P}$-function.
\end{enumerate}

\end{prop}
\begin{proof}
We follow the proof of Theorem 1 in \cite{Tak72}. Let $\phi$, $\psi$ be $\mathcal{L}^0$-formulae. By mere logic,
\begin{align*}
\neg \exists x . \phi(x) &\leftrightarrow \forall x . \neg \phi(x), \\
\exists x . \phi(x) \wedge \exists y . \psi(y) &\leftrightarrow \exists x . \exists y . (\phi(x) \wedge \phi(y)), \\
\exists x . \phi(x) \vee \exists y . \psi(y) &\leftrightarrow \exists z . (\phi(z) \vee \psi(z)). 
\end{align*}

Assume that $\phi$ and $\psi$ are in $\Sigma_1^\mathcal{P}$. By set-theory,
\begin{align*}
\exists x . \exists y . \phi(x, y) &\leftrightarrow \exists z . \exists x \in z . \exists y \in z . \phi(x, y), \\
\forall u \in x . \exists y . \phi(u, x, y) &\leftrightarrow \exists z . \forall u \in x . \exists y \in z . \phi(u, x, y), \\
\forall u \subseteq x . \exists y . \phi(u, x, y) &\leftrightarrow \exists z . \forall u \subseteq x . \exists y \in z . \phi(u, x, y). 
\end{align*}
In particular, the first equivalence follows from Pair, the second follows from $\Sigma_1^\mathcal{P}$-Collection, and the third follows from Powerset and $\Sigma_1^\mathcal{P}$-Collection.

Suppose that $f$ is defined by a $\Sigma_1^\mathcal{P}$-formula $F(x, y)$, such that $f(x) = y \leftrightarrow F(x, y)$. By considering the formula
\[
F'(x, y) \equiv \forall y\hspace{1pt}' . (F(x, y\hspace{1pt}') \rightarrow y = y\hspace{1pt}'),
\]
we see that $f$ is actually a $\Delta_1^\mathcal{P}$-function. Now note that by logic, for any $\phi \in \mathcal{L}^0$,
\begin{align*}
\exists x \in f(y) . \phi(x, y) &\leftrightarrow \exists z . (f(y) = z \wedge \exists x \in z . \phi(x, y)) \\
&\leftrightarrow \forall z . (f(y) = z \rightarrow \exists x \in z . \phi(x, y)).
\end{align*}
With all these equivalences at hand, the proposition is easily verified.
\end{proof}

The fact that $\mathrm{KP}^\mathcal{P}$ proves the existence of the $V$-hierarchy is very useful. For example, it enables the following result.

\begin{prop}\label{KPP found}
	For each $1 \leq k < \omega$, $\mathrm{KP}^\mathcal{P} + \Sigma_k^\mathcal{P} \textnormal{-Separation} \vdash \mathrm{B}_k^\mathcal{P} \textnormal{-Foundation}$.
\end{prop}
\begin{proof}
	Recall that $\Sigma_k^\mathcal{P} \textnormal{-Separation}$ implies $\mathrm{B}_k^\mathcal{P} \textnormal{-Separation}$. 
	Let $\phi(x) \in \mathrm{B}_k^\mathcal{P}[x]$. Suppose there is $a$ such that $\phi(a)$. By $\mathrm{B}_k^\mathcal{P}$-Separation, let 
	$$A = \{ x \in V_{\rnk(a) + 1} \mid \phi(x) \},$$
	and note that $a \in A.$
	By $\Sigma_1$-Separation, let 
	\[R = \{ \xi < \rnk(a) + 1 \mid \exists x \in A . \rnk(x) = \xi \} . \]
	Since $R$ is a non-empty set of ordinals, it has a least element $\rho$. Let $a' \in A$ such that $\rnk(a') = \rho$. Then we have $\forall x \in a . \neg \phi(x)$, as desired.
\end{proof}

\section{ZFC, GBC and ``Ord is weakly compact''}\label{ZFC and GBC}

A set $s$ is {\em ordinal definable} if there is $n \in \mathbb{N}$, ordinals $\alpha, \beta_1, \dots, \beta_n$, and a formula $\phi$, for which $s$ is the unique set such that $(V_\alpha, \in\restriction_{V_\alpha}) \models \phi(s, \beta_1, \dots, \beta_n)$. $\mathrm{OD}$ is defined as the class of ordinal definable sets. A set $t$ is {\em hereditarily ordinal definable} if every element in the transitive closure of $\{t\}$ is ordinal definable. $\mathrm{HOD}$ is defined as the class of hereditarily ordinal definable sets.

$\mathrm{ZFC}$ plus the axiom
\[
\begin{array}{ll}
V = \mathrm{HOD} & \textnormal{``every set is hereditarily ordinal definable'',}
\end{array}
\]
has definable Skolem functions.

If $\mathcal{L}$ is an expansion of $\mathcal{L}^0$ with more symbols, then $\mathrm{ZF}(\mathcal{L})$ denotes the theory 
$$\mathrm{ZF} + \mathcal{L} \textnormal{-Separation} + \mathcal{L} \textnormal{-Replacement},$$ 
by which is meant that the schemata of Separation and Replacement are extended to all formulae in $\mathcal{L}$. $\mathrm{ZFC}(\mathcal{L})$ is defined analogously.

We will also consider the axiom ``$\mathrm{Ord}$ is weakly compact'', in the context of $\mathrm{GBC}$. It is defined as ``Every binary tree of height $\mathrm{Ord}$ has a branch.'' The new notions used in the definiens are now to be defined. Let $\alpha$ be an ordinal. A {\em binary tree} is a (possibly class) structure $\mathcal{T}$ with a binary relation $<_\mathcal{T}$, such that:
\begin{enumerate}[(i)]
	\item Every element of $\mathcal{T}$ (called a {\em node}) is a function from an ordinal to $2$;
	\item For every $f \in \mathcal{T}$ and every ordinal $\xi < \dom(f)$, we have $f\restriction_\xi \in \mathcal{T}$;
	\item For all $f, g \in \mathcal{T}$, 
	\[
	f <_\mathcal{T} g \Leftrightarrow \dom(f) < \dom(g) \wedge g\restriction_{\dom(f)} = f;
	\]
\end{enumerate}
Suppose that $\mathcal{T}$ is a binary tree. The {\em height} of $\mathcal{T}$, denoted $\mathrm{height}(\mathcal{T})$, is $\{ \xi \mid \exists f \in \mathcal{T} . \dom(f) = \xi \}$ (which is either an ordinal or the class $\mathrm{Ord}$). A {\em branch} in $\mathcal{T}$ is a (possibly class) function $F : \mathrm{Ord} \rightarrow \mathcal{T}$, such that for all ordinals $\xi \in \mathrm{height}(\mathcal{T})$, $\dom(F(\alpha)) = \alpha$, and for all ordinals $\alpha < \beta \in \mathrm{height}(\mathcal{T})$, $F(\alpha) <_\mathcal{T} F(\beta)$. Moreover, for each ordinal $\alpha \in \mathrm{height}(\mathcal{T})$, we define 
\[
\mathcal{T}_\alpha =_\df \mathcal{T}\restriction_{\{ f \in \mathcal{T} \mid \dom(f) < \alpha \}}.
\]

\begin{prop}\label{GBC weakly compact global choice}
$\mathrm{GBC} + \textnormal{``$\mathrm{Ord}$ is weakly compact''} \vdash \textnormal{Global Choice}$, where
\[
\begin{array}{ll}
	\textnormal{Global Choice} & \exists F . ( (F : V \setminus \{\varnothing\} \rightarrow V) \wedge \forall x \in V \setminus \{\varnothing\} . F(x) \in x ). \\
\end{array}
\]
\end{prop}
\begin{proof}[Proof-sketch]
Consider the class tree of well-orderings of $V_\alpha$, for all ordinals $\alpha$: The well-orderings are ordered by $P \leq Q \Leftrightarrow_\df P = Q\restriction_{\dom(P)}$. Using $\textnormal{``$\mathrm{Ord}$ is weakly compact''}$, one can show that this tree has a branch. From this branch a class-function witnessing global choice can be constructed.
\end{proof}

In \cite{Ena04} it is shown that the $\mathcal{L}^0$-consequences of $\mathrm{GBC} + \textnormal{``$\mathrm{Ord}$ is weakly compact''}$ are the same as for $\mathrm{ZFC} + \Phi$, where 
$$\Phi = \{ \exists \kappa . \textnormal{``$\kappa$ is $n$-Mahlo and $V_\kappa \prec_{\Sigma_n} V$''} \mid n \in \mathbb{N}\}.$$ 
In particular, they are equiconsistent. Let us therefore define $n$-Mahlo and explain why $V_\kappa \prec_{\Sigma_n} V$ can be expressed as a sentence. 

$V_\kappa \prec_{\Sigma_n} V$ is expressed by a sentence saying that for all $\Sigma_n$-formulae $\phi(\vec{x})$ of set theory and for all $\vec{a} \in V_\kappa$ matching the length of $\vec{x}$, we have $((V_\kappa, \in\restriction_{V_\kappa}) \models \phi(\vec{a})) \leftrightarrow \mathrm{Sat}_{\Sigma_n}(\phi, \vec{a})$. Here we utilize the partial satisfaction relations $\mathrm{Sat}_{\Sigma_n}$, introduced to set theory in \cite{Lev65}.

In contrast to the result above, Enayat has communicated to the author that there are countable models of $\mathrm{ZFC} + \Phi$ that do not expand to models of $\mathrm{GBC} + \textnormal{``$\mathrm{Ord}$ is weakly compact''}$. In particular, such is the fate of {\em Paris models}, i.e. models of $\mathrm{ZFC}$ each of whose ordinals is definable in the model. An outline of a proof: Let $\mathcal{M}$ be a Paris model. There is no ordinal $\alpha$ in $\mathcal{M}$, such that $\mathcal{M}_\alpha \prec \mathcal{M}$, where $\mathcal{M}_\alpha =_\df ((V_\alpha^\mathcal{M})_\mathcal{M}, \in^\mathcal{M}\restriction_{(V_\alpha^\mathcal{M})_\mathcal{M}})$, because that would entail that a correct satisfaction relation is definable in $\mathcal{M}$ contradicting Tarski's well-known theorem on the undefinability of truth. Suppose that $\mathcal{M}$ expands to a model $(\mathcal{M}, \mathcal{A}) \models \mathrm{GBC} + \textnormal{``$\mathrm{Ord}$ is weakly compact''}$. Then by the proof of Theorem 4.5(i) in \cite{Ena01}, $\mathcal{M}$ has a safe satisfaction relation $\mathrm{Sat} \in \mathcal{A}$ (see the definition preceding Lemma \ref{rec sat reflection}) below. But then there is, by Lemma \ref{rec sat reflection}, unboundedly many ordinals $\alpha$ in $\mathcal{M}$ such that $\mathcal{M}_\alpha \prec \mathcal{M}$.

Moreover, it is shown in \cite{EH17} that for every model $\mathcal{M} \models \mathrm{ZFC}$ and the collection $\mathcal{A}$ of definable subsets of $\mathcal{M}$, we have $(\mathcal{M}, \mathcal{A}) \models \mathrm{GBC} + \neg \textnormal{``$\mathrm{Ord}$ is weakly compact''}$.

\section{Non-standard models of set theory}\label{Models of set theory}

This section contains basic material on non-standard models of set theory. An expanded version of this section with more proofs is found as \S 4.6 of \cite{Gor18}.

If $\mathcal{M}$ is an $\mathcal{L}^0$-structure, and $a$ is a set or class in $\mathcal{M}$, then the {\em externalization} of $a$, denoted $a_\mathcal{M}$, is defined as $\{x \in \mathcal{M} \mid x \in^\mathcal{M} a\}$. If $E = \in^\mathcal{M}$, then the notation $a_E$ is also used for $a_\mathcal{M}$. If $f \in \mathcal{M}$ and $\mathcal{M} \models \text{``$f$ is a function''}$, then we say that $f$ is a {\em function internal} to $\mathcal{M}$. If $f \in \mathcal{M}$ is a function internal to $\mathcal{M}$, then $f_\mathcal{M}$ also denotes the {\em externalization} of this function: $\forall x, y \in \mathcal{M} . (f_\mathcal{M}(x) = y \leftrightarrow \mathcal{M} \models f\hspace{2pt}(x) = y)$. Moreover, if $a \in \mathcal{M}$ codes a structure internal to $\mathcal{M}$, then $a_\mathcal{M}$ also denotes the {\em externalization} of this structure; in particular, if $R$ codes a relation in $\mathcal{M}$, then $\forall x, y \in \mathcal{M} . ( x R_\mathcal{M} y \leftrightarrow \mathcal{M} \models \langle x, y \rangle \in R)$. For example, recall that $\mathbb{N}^\mathcal{M}$ denotes the interpretation of $\mathbb{N}$ in $\mathcal{M}$ (assuming that $\mathcal{M}$ satisfies that the standard model of arithmetic exists); then $\mathbb{N}^\mathcal{M}_\mathcal{M}$ denotes the externalization of this model (which might be non-standard). 

Let $\mathcal{M} \models \mathrm{KP}$. Then every element of $\mathcal{M}$ has a rank; so for any $\alpha \in \mathrm{Ord}^\mathcal{M}$, we can define 
\[
\mathcal{M}_\alpha =_\df \{m \in \mathcal{M} \mid \mathcal{M} \models \rnk(m) < \alpha\}.
\]
We say that an embedding $i : \mathcal{S} \rightarrow \mathcal{M}$ of an $\mathcal{L}^0$-structure $\mathcal{S}$ into $\mathcal{M}$ is {\em bounded} ({\em by} $\alpha \in \mathrm{Ord}^\mathcal{M}$) if $i(\mathcal{S}) \subseteq \mathcal{M}_\alpha$.

\begin{dfn}
Let $\mathcal{M} = (M, E)$ be a model in $\mathcal{L}^0$. It is {\em standard} if $E$ is well-founded. Assume that $\mathcal{M}$ is a model of $\mathrm{KP}$. Then the usual rank-function $\rnk : M \rightarrow \mathrm{Ord}^\mathcal{M}$ is definable in $\mathcal{M}$. Therefore $\mathcal{M}$ is non-standard iff $E \restriction_{\mathrm{Ord}^\mathcal{M}}$ is not well-founded. $m \in \mathcal{M}$ is {\em standard in} $\mathcal{M}$ if $\mathcal{M}\restriction_{m_E}$ is standard. 
\begin{itemize}
\item The {\em ordinal standard part} of $\mathcal{M}$, denoted $\mathrm{OSP}(\mathcal{M})$, is defined:
\[
\mathrm{OSP}(\mathcal{M}) =_\df \{\alpha \in \mathcal{M} \mid \text{``}\alpha \text{ is a standard ordinal of $\mathcal{M}$''}\}. 
\]
We say that $\mathcal{M}$ is {\em $\omega$-non-standard} if $\mathrm{OSP}(\mathcal{M}) = \omega$.
\item The {\em well-founded part} of $\mathcal{M}$, denoted $\mathrm{WFP}(\mathcal{M})$, is the substructure of $\mathcal{M}$ on the elements of standard rank:
\[
\mathrm{WFP}(\mathcal{M}) =_\df \mathcal{M}\restriction_{\{x \in \mathcal{M} \mid \text{``$x$ is standard in $\mathcal{M}$''}\}}.
\]
\item A set of the form $c_E \cap A$, where $c \in M$ and $A \subseteq M$, is said to be a subset of $A$ {\em coded in} $\mathcal{M}$. This notion is extended in the natural way to arbitrary injections into $M$. We define:
\[
\mathrm{Cod}_A(\mathcal{M}) =_\df \{c_E \cap A \mid c \in M\}.
\]
\item The {\em standard system of $\mathcal{M}$ over $A \subseteq \mathcal{M}$}, denoted $\mathrm{SSy}_A(\mathcal{M})$, is obtained by expanding $\mathcal{M}\restriction_A$ to an $\mathcal{L}^1$-structure, adding $\mathrm{Cod}_A(\mathcal{M})$ as classes:
\begin{align*}
\mathrm{SSy}_A(\mathcal{M}) &=_\df (\mathcal{M}\restriction_A, \mathrm{Cod}_A(\mathcal{M})), \\
x \in^{\mathrm{SSy}_A(\mathcal{M})} C &\Leftrightarrow_\df x \in^\mathcal{M} c, 
\end{align*}
for any $x \in A$, $c \in \mathcal{M}$ and $C \in \mathrm{Cod}_A(\mathcal{M})$, such that $C = c_E \cap A$.

Moreover, we define 
\[
\mathrm{SSy}(\mathcal{M}) =_\df \mathrm{SSy}_{\mathrm{WFP}(\mathcal{M})}(\mathcal{M}).
\]
\end{itemize}
\end{dfn}

\begin{dfn}\label{special embeddings}
Let $i : \mathcal{M} \rightarrow \mathcal{N}$ be an embedding of models of $\mathrm{KP}$.
\begin{itemize}
\item $i$ is {\em initial}, if
\[
\forall m \in \mathcal{M} . \forall n \in \mathcal{N} . (n \in^\mathcal{N} i(m) \rightarrow n \in i(\mathcal{M})).
\]
This is equivalent to:
\[
\forall m \in \mathcal{M} . i(m_\mathcal{M}) = i(m)_\mathcal{N}.
\]
\item $i$ is {\em $\mathcal{P}$-initial} (or {\em power-initial}), if it is initial and {\em powerset preserving} in the sense:
\[
\forall m \in \mathcal{M} . \forall n \in \mathcal{N} . (n \subseteq^\mathcal{N} m \rightarrow n \in i(\mathcal{M})).
\]
\item $i$ is {\em rank-initial}, if 
\[
\forall m \in \mathcal{M} . \forall n \in \mathcal{N} . (\rnk^\mathcal{N}(n) \leq^\mathcal{N} \rnk^\mathcal{N}(i(m)) \rightarrow n \in i(\mathcal{M})).
\]
\item $i$ is {\em topless}, if it is bounded and
\[
\forall \beta \in \mathrm{Ord}^\mathcal{N} \setminus i(\mathcal{M}) . \exists \beta' \in \mathrm{Ord}^\mathcal{N} \setminus i(\mathcal{M}) . \beta' <^\mathcal{N} \beta. 
\]
\item $i$ is {\em strongly topless}, if it is bounded and for each $f \in \mathcal{N}$ with $\alpha, \beta \in \mathrm{Ord}^\mathcal{N}$ satisfying
\[
(\mathcal{N} \models f : \alpha \rightarrow \beta) \wedge \alpha_\mathcal{N} \supseteq i(\mathrm{Ord}^\mathcal{M}),
\]
there is $\nu \in \mathrm{Ord}^\mathcal{N} \setminus i(\mathcal{M})$ such that for all $\xi \in i(\mathrm{Ord}^\mathcal{M})$,
\[
f_\mathcal{N}(\xi) \not\in i(\mathcal{M}) \Leftrightarrow \mathcal{N} \models f\hspace{2pt}(\xi) > \nu. 
\]
\item $i$ is $\omega${\em -topless}, if it is bounded and not $\omega${\em -coded from above}, meaning that for each $f \in \mathcal{N}$ with $\alpha, \beta \in \mathrm{Ord}^\mathcal{N}$ satisfying
\[
(\mathcal{N} \models f : \alpha \rightarrow \beta) \wedge \alpha_\mathcal{N} \supseteq \omega \wedge \forall k < \omega . f_\mathcal{N}(k) \in \mathrm{Ord}^\mathcal{N} \setminus i(\mathcal{M}),
\]
there is $\nu \in \mathrm{Ord}^\mathcal{N} \setminus i(\mathcal{M})$ such that $\nu <^\mathcal{N} f_\mathcal{N}(k)$, for all $k < \omega$.
\end{itemize}

The notions of initiality are often combined with some notion of toplessness, yielding notions of {\em cut}. In particular, an embedding $i$ is a {\em rank-cut} if it is topless and rank-initial, and $i$ is a {\em strong rank-cut} if it is strongly topless and rank-initial.
\end{dfn}

It is easily seen that if $i$ is rank-initial and proper, then it is bounded, so the first condition of toplessness is satisfied. Moreover, we immediately obtain the following implications:
\[
\text{$i$ is initial } \Leftarrow \text{$i$ is $\mathcal{P}$-initial $\Leftarrow$ $i$ is rank-initial},
\]
\[
\text{$i$ is topless} \Leftarrow \text{$i$ is $\omega$-topless} \Leftarrow \text{$i$ is strongly topless}.
\]

Proofs of the following results are found in \S 4.6 of \cite{Gor18}. 

\begin{lemma}[$\Sigma_n^\mathcal{P}$-Overspill]\label{overspill}
Let $n \in \mathbb{N}$ and suppose that $\mathcal{S}$ is a rank-cut of $\mathcal{M} \models \mathrm{KP}^\mathcal{P} + \Pi_n^\mathcal{P} \textnormal{-Foundation}$, that $m \in \mathcal{M}$, and that $\sigma(x, y) \in \Sigma_n^\mathcal{P}[x, y]$. If for every ordinal $\xi \in \mathcal{S}$, there is an ordinal $\xi <^\mathcal{M} \zeta \in \mathcal{S}$ such that $\mathcal{M} \models \sigma(\zeta, m)$, then  for every ordinal $\nu \in \mathcal{M} \setminus \mathcal{S}$ there is an ordinal $\nu >^\mathcal{M} \mu \in \mathcal{M} \setminus \mathcal{S}$, such that $\mathcal{M} \models \sigma(\mu, m)$.
\end{lemma}

\begin{prop}\label{elem initial is rank-initial}
	Let $\mathcal{M} \models \mathrm{KP}^\mathcal{P}$ and suppose that $i : \mathcal{M} \rightarrow \mathcal{N}$ is an elementary embedding. Then $i$ is initial if, and only if, it is rank-initial.
\end{prop}

\begin{prop}\label{WFP rank-initial topless}
Let $\mathcal{M} \models \mathrm{KP}$. An element $m_0 \in \mathcal{M}$ is standard iff $\rnk^\mathcal{M}(m_0)$ is standard. If $\mathcal{M}$ is non-standard, then $\mathrm{WFP}(\mathcal{M})$ is a rank-cut of $\mathcal{M}$.
\end{prop}

\begin{prop}\label{comp emb}
	Let $i : \mathcal{M} \rightarrow \mathcal{N}$ and $j : \mathcal{N} \rightarrow \mathcal{O}$ be embeddings of models of $\mathrm{KP}$. 
\begin{enumerate}[{\normalfont (a)}]
	\item\label{comp emb topless} If $i$ is topless and $j$ is initial, then $j \circ i : \mathcal{M} \rightarrow \mathcal{O}$ is topless.
	\item\label{comp emb strongly topless} If $i$ is strongly topless and $j$ is rank-initial, then $j \circ i : \mathcal{M} \rightarrow \mathcal{O}$ is strongly topless.
\end{enumerate}		
\end{prop}
\begin{proof}
	(\ref{comp emb topless}) Let $\gamma \in \mathrm{Ord}(\mathcal{O}) \setminus j\circ i(\mathcal{M})$. By toplessness, there are $\beta' <^\mathcal{N} \beta \in \mathrm{Ord}(\mathcal{N}) \setminus i(\mathcal{M})$. If $j(\beta') <^\mathcal{O} \gamma$, then we are done. Otherwise, $\gamma <^\mathcal{O} j(\beta)$, so that by initiality $\gamma \in j(\mathcal{N})$. By toplessness, this yields $\beta'' \in \mathrm{Ord}^\mathcal{N} \setminus i(\mathcal{M})$ such that $j(\beta'') < \gamma$.
	
	(\ref{comp emb strongly topless}) Let $f \in \mathcal{O}$ with $\alpha, \beta \in \mathrm{Ord}^\mathcal{O}$ satisfying
\[
(\mathcal{O} \models f : \alpha \rightarrow \beta) \wedge \alpha_\mathcal{O} \supseteq (j \circ i)(\mathrm{Ord}^\mathcal{M}).
\]
By toplessness of $i$ and initiality of $j$, there is $\gamma \in \mathrm{Ord}^\mathcal{N}$, such that $\gamma_\mathcal{N} \supseteq i(\mathrm{Ord}^\mathcal{M})$ and $j(\gamma) \leq^\mathcal{O} \alpha$. In $\mathcal{O}$, let $f' : j(\gamma) \rightarrow j(\gamma)$ be the ``truncation'' of $f$ defined by $f'(\xi) = f(\xi)$, if $f(\xi) < j(\gamma)$, and by $f'(\xi) = 0$, if $f(\xi) \geq j(\gamma)$. Note that $\mathcal{O} \models \rnk(f') \leq j(\gamma) + 2$. So by rank-initiality, there is $g' \in \mathcal{N}$, such that $j(g') = f'$. Consequently, $\mathcal{N} \models g' : \gamma \rightarrow \gamma$, and by initiality of $j$, $(j(g'))_\mathcal{O} = f'_\mathcal{O}$. By strong toplessness of $i$, there is $\nu \in \mathrm{Ord}^\mathcal{N} \setminus i(\mathcal{M})$ such that for all $\xi \in i(\mathrm{Ord}^\mathcal{M})$,
\[
g'_\mathcal{N}(\xi) \not\in i(\mathcal{M}) \Leftrightarrow \mathcal{N} \models g'(\xi) > \nu. 
\]
It follows that for all $\xi \in (j \circ i)(\mathrm{Ord}^\mathcal{M})$,
\[
f_\mathcal{O}(\xi) \not\in (j \circ i)(\mathcal{M}) \Leftrightarrow \mathcal{O} \models f(\xi) > j(\nu). 
\]
So $j \circ i$ is strongly topless.
\end{proof}

\begin{lemma}\label{omega-topless existence}
	Let $\mathcal{M} \models \mathrm{KP}^\mathcal{P}$ be $\omega$-non-standard and let $\alpha_0 \in \mathrm{Ord}^\mathcal{M}$. For each $k < \omega$, let $\alpha_k \in \mathrm{Ord}^\mathcal{M}$ such that $\mathcal{M} \models \alpha_k = \alpha_0 + k^\mathcal{M}$. Then $\bigcup_{k < \omega} \mathcal{M}_{\alpha_k}$ is an $\omega$-topless rank-initial substructure of $\mathcal{M}$. 
\end{lemma}
\begin{proof}
	$\bigcup_{k < \omega} \mathcal{M}_{\alpha_k}$ is obviously rank-initial in $\mathcal{M}$. Let $m \in^\mathcal{M} \omega^\mathcal{M}$ be non-standard. Since $\alpha_0 +^\mathcal{M} m \in \mathcal{M}$, we have that $\bigcup_{k < \omega} \mathcal{M}_{\alpha_k}$ is bounded in $\mathcal{M}$. Moreover, $\bigcup_{k < \omega} \mathcal{M}_{\alpha_k}$ is topless, because otherwise $\gamma =_\df \sup\{\alpha_k \mid k < \omega \}$ exists in $\mathcal{M}$ and $\mathcal{M} \models \gamma - \alpha_0 = \omega$, which contradicts that $\mathcal{M}$ is $\omega$-non-standard. 
	
	Let $f : \alpha \rightarrow \beta$ be a function in $\mathcal{M}$, where $\alpha, \beta \in \mathrm{Ord}^\mathcal{M}$ and 
	\[
	\alpha_\mathcal{M} \supseteq \omega \wedge \forall k < \omega . f_\mathcal{M}(k) \in \mathrm{Ord}^\mathcal{M} \setminus \bigcup_{k < \omega} \mathcal{M}_{\alpha_k}.
	\] 
	Note that for all $k < \omega$,
	\[
	\mathcal{M} \models \forall \xi \leq k . f\hspace{2pt}(\xi) > \alpha_0 + \xi.
	\]
	So by $\Delta_0^\mathcal{P}$-Overspill, there is a non-standard $\mathring{k} \in^\mathcal{M} \omega^\mathcal{M}$, such that 
	\[
	\mathcal{M} \models \forall \xi \leq \mathring{k} . f\hspace{2pt}(\xi) > \alpha_0 + \xi.
	\]
	Hence, $\alpha_0 +^\mathcal{M} \mathring{k}$ witnesses $\omega$-toplessness of $\bigcup_{k < \omega} \mathcal{M}_{\alpha_k}$.
\end{proof}

The following classic result is proved as Theorem 6.15 in \cite{Jeh02}:

\begin{thm}[Mostowski's Collapse]\label{Mostowski collapse}
If $\mathcal{M}$ is a well-founded model of $\textnormal{Extensionality}$, then there is a unique isomorphism $\mathrm{Mos} : \mathcal{M} \rightarrow (T, \in\restriction_T)$, such that $T$ is transitive. Moreover, 
$$\forall x \in \mathcal{M} . \mathrm{Mos}(x) = \{\mathrm{Mos}(u) \mid u \in^\mathcal{M} x\}.$$
\end{thm}

This theorem motivates the following simplifying assumption:

\begin{ass}\label{ass trans}
Every well-founded $\mathcal{L}^0$-model $\mathcal{M}$ of $\textnormal{Extensionality}$ is a transitive set, or more precisely, is of the form $(T, \in\restriction_T)$ where $T$ is transitive and unique. Every embedding between well-founded $\mathcal{L}^0$-models of $\textnormal{Extensionality}$ is an inclusion function. 
\end{ass}

In particular, for any model $\mathcal{M}$ of $\mathrm{KP}$, $\mathrm{WFP}(\mathcal{M})$ is a transitive set and $\mathrm{OSP}(\mathcal{M})$ is an ordinal.

\begin{prop}\label{emb pres}
If $i : \mathcal{M} \rightarrow \mathcal{N}$ is an initial embedding between models of $\mathrm{KP}$, then:
\begin{enumerate}[{\normalfont (a)}]
\item\label{emb pres fixed} $i(x) = x$, for all $x \in \mathrm{WFP}(\mathcal{M})$.
\item\label{emb pres Delta_0} $i$ is $\Delta_0$-elementary. More sharply, for every $\sigma(\vec{x}) \in \Sigma_1[\vec{x}]$, and for every $\vec{a} \in \mathcal{M}$, 
\[
\mathcal{M} \models \sigma(\vec{a}) \Rightarrow \mathcal{N} \models \sigma(i(\vec{a})).
\]
\item\label{emb pres Delta_0^P} If $i$ is $\mathcal{P}$-initial, then it is $\Delta_0^\mathcal{P}$-elementary. More sharply, for every $\sigma(\vec{x}) \in \Sigma_1^\mathcal{P}[\vec{x}]$, and for every $\vec{a} \in \mathcal{M}$, 
\[
\mathcal{M} \models \sigma(\vec{a}) \Rightarrow \mathcal{N} \models \sigma(i(\vec{a})).
\]
\item\label{emb pres SSy_S} If $i$ is rank-initial and $\mathcal{S}$ is a common bounded substructure of $\mathcal{M}$ and $\mathcal{N}$, pointwise fixed by $i$, then $\mathrm{SSy}_\mathcal{S}(\mathcal{M}) = \mathrm{SSy}_{\mathcal{S}}(\mathcal{N})$.

\end{enumerate}
\end{prop}
The proof is routine.

\begin{cor}\label{rank-init equivalences in KPP}
If $i: \mathcal{M} \rightarrow \mathcal{N}$ is an embedding between models of $\mathrm{KP}^\mathcal{P}$, then the following are equivalent:
\begin{enumerate}[{\normalfont (a)}]
\item $i$ is rank-initial.
\item $i$ is $\mathcal{P}$-initial.
\item $i$ is initial, and for each ordinal $\xi$ in $\mathcal{M}$, $i(V_\xi^\mathcal{M}) = V_{i(\xi)}^\mathcal{M}$.
\end{enumerate}
\end{cor}
\begin{proof}
(a) $\Rightarrow$ (b) is immediate from the definition.

(b) $\Rightarrow$ (c): The formula $x = V_\alpha$ is $\Sigma_1^\mathcal{P}$, so by Proposition \ref{emb pres} (\ref{emb pres Delta_0^P}), $i(V_\alpha^\mathcal{M}) = V_{i(\alpha)}^\mathcal{N}$, for all ordinals $\alpha$ in $\mathcal{M}$.

(c) $\Rightarrow$ (a): Let $a \in \mathcal{M}$ with $\mathcal{M}$-rank $\alpha$, let $b = i(a)$ with $\mathcal{N}$-rank $\beta$, and let $c \in \mathcal{N}$ be of $\mathcal{N}$-rank $\gamma \leq^\mathcal{N} \beta$. Since $i$ is an embedding preserving $(\xi \mapsto V_\xi)$, we have that $b \in^\mathcal{N} i(V_{\alpha + 1})$ and $i(V_{\alpha + 1}) = V_{\beta'}$, for some $\beta' >^\mathcal{N} \beta$. Therefore $c \in i(V_{\alpha+1})$, so since $i$ is initial, we get that $c \in i(\mathcal{M})$.
\end{proof}

We shall now introduce partial satisfaction relations for the Takahashi hierarchy. The details of this are worked out in \S 4.4 of \cite{Gor18}. Recall that we have assumed that each $\mathcal{L}^0$-formula is identical to its G\"odel code. Moreover, recall that for each $n \in \mathbb{N}$, $\hat\Sigma^\mathcal{P}_n$ and $\hat\Pi^\mathcal{P}_n$ are defined recursively by: $\hat\Delta^\mathcal{P}_0 = \hat\Sigma^\mathcal{P}_0 = \hat\Pi^\mathcal{P}_n$ is the set of $\mathcal{L}^0$-formulae such that every quantifier is of the form `$\exists x \in y$', `$\exists x \subseteq y$', `$\forall x \in y$' or `$\forall x \subseteq y$'; if $\phi$ is $\hat\Sigma^\mathcal{P}_n$, then $\exists x . \phi$ is $\hat\Sigma^\mathcal{P}_n$ and $\forall x . \phi$ is $\hat\Pi^\mathcal{P}_{n+1}$; and dually, if $\phi$ is $\hat\Pi^\mathcal{P}_n$, then $\exists x . \phi$ is $\hat\Sigma^\mathcal{P}_{n+1}$ and $\forall x . \phi$ is $\hat\Pi^\mathcal{P}_{n}$. 

\begin{thm}\label{sat_thm}
For each $n \in \mathbb{N}$, there are formulae $\mathrm{Sat}_{\Sigma_n^\mathcal{P}}(\sigma, \vec{x}) \in \hat \Sigma_{\mathrm{max}(1, n)}^\mathcal{P}$ and $\mathrm{Sat}_{\Pi_n^\mathcal{P}}(\pi, \vec{x}) \in \hat \Pi_{\mathrm{max}(1, n)}^\mathcal{P}$, such that for any model $\mathcal{M} \models \mathrm{KP}^\mathcal{P}$, any $\sigma \in \hat\Sigma_n^\mathcal{P}[\vec{u}]$ and any 
$\pi \in \hat\Pi_n^\mathcal{P}[\vec{u}]$,
\begin{align*}
\mathcal{M} &\models \sigma(\vec{x}) \leftrightarrow \mathrm{Sat}_{\Sigma_n^\mathcal{P}}(\sigma, \vec{x}) \\
\mathcal{M} &\models \pi(\vec{x}) \leftrightarrow \mathrm{Sat}_{\Pi_n^\mathcal{P}}(\pi, \vec{x}).
\end{align*}
\end{thm}

We write $\mathrm{Sat}_{\Delta_0^\mathcal{P}}$ for $\mathrm{Sat}_{\Sigma_0^\mathcal{P}}$ and $\mathrm{Sat}_{\Pi_0^\mathcal{P}}$, as they are equivalent. By the theorem, $\mathrm{Sat}_{\Delta_0^\mathcal{P}} \in \Delta_1^\mathcal{P}$.

\begin{lemma}\label{hierarchical type coded}
	Let $n \in \mathbb{N}$, suppose that $\mathcal{M} \models \mathrm{KP}^\mathcal{P} + \Sigma_n^\mathcal{P}\textnormal{-Separation}$ is non-standard, and let $\mathcal{S}$ be a bounded substructure of $\mathcal{M}$. For each $\vec{a} \in \mathcal{M}$, we have that $\mathrm{tp}_{\hat\Sigma_n^\mathcal{P}, \mathcal{S}}(\vec{a})$ and $\mathrm{tp}_{\hat\Pi_n^\mathcal{P}, \mathcal{S}}(\vec{a})$ are coded in $\mathcal{M}$.
\end{lemma}
\begin{proof}
	Let $\alpha$ be an ordinal in $\mathcal{M}$, such that $\mathcal{S} \subseteq \mathcal{M}_\alpha$. Note that if $n = 0$, then $\mathcal{M} \models \Delta_1^\mathcal{P}\textnormal{-Separation}$, and for any $n \in \mathbb{N}$ we have both $\Sigma_n^\mathcal{P}$-Separation and $\Pi_n^\mathcal{P}$-Separation in $\mathcal{M}$. So by Theorem \ref{sat_thm}, we have these sets in $\mathcal{M}$:
	\begin{align*}
	s &= \{ \phi(\vec{x}, \vec{v}) \mid \rnk(\vec{v}) < \alpha \wedge \mathrm{Sat}_{\Sigma_n^\mathcal{P}}( \phi(\vec{x}, \vec{y}), \vec{a}, \vec{v})\}^\mathcal{M}, \\
	p &= \{ \phi(\vec{x}, \vec{v}) \mid \rnk(\vec{v}) < \alpha \wedge \mathrm{Sat}_{\Pi_n^\mathcal{P}}( \phi(\vec{x}, \vec{y}), \vec{a}, \vec{v})\}^\mathcal{M}.
	\end{align*}
	By Theorem \ref{sat_thm}, $s$ codes $\mathrm{tp}_{\hat\Sigma_n^\mathcal{P}, \mathcal{S}}(\vec{a})$, and $p$ codes $\mathrm{tp}_{\hat\Pi_n^\mathcal{P}, \mathcal{S}}(\vec{a})$.
\end{proof}

Theorem \ref{rec sat char} below characterizing recursively saturated models of $\mathrm{ZF}$ is useful. To state it we introduce this definition: Let $\mathcal{L}^0_\mathrm{Sat}$ be the language obtained by adding a new binary predicate $\mathrm{Sat}$ to $\mathcal{L}^0$. We say that $\mathcal{M}$ {\em admits a safe satisfaction relation} if $\mathcal{M}$ expands to an $\mathcal{L}^0_\mathrm{Sat}$-structure $(\mathcal{M}, \mathrm{Sat}^\mathcal{M})$, such that
\begin{enumerate}[{\normalfont (i)}]
	\item $(\mathcal{M}, \mathrm{Sat}^\mathcal{M}) \models \mathrm{ZF}(\mathcal{L}^0_\mathrm{Sat})$,
	\item For each $n \in \mathbb{N}$:
	\begin{align*}
	(\mathcal{M}, \mathrm{Sat}^\mathcal{M}) &\models \forall \sigma \in \hat\Sigma_n . (\mathrm{Sat}(\sigma, x) \leftrightarrow \mathrm{Sat}_{\Sigma_n}(\sigma, x)) \\
	(\mathcal{M}, \mathrm{Sat}^\mathcal{M}) &\models \forall \pi \in \hat\Pi_n . (\mathrm{Sat}(\pi, x) \leftrightarrow \mathrm{Sat}_{\Pi_n}(\pi, x))
	\end{align*}
\end{enumerate}
We say that $\mathrm{Sat}^\mathcal{M}$ is a {\em safe satisfaction relation} on $\mathcal{M}$. For the second condition considered alone, we say that $\mathrm{Sat}^\mathcal{M}$ is {\em correct for standard formulae}.

The following result is found as Theorem 3.2 (and the remark preceding it) in \cite{Sch78}.

\begin{thm}\label{rec sat reflection}
	Let $\mathcal{M}$ be a model of $\mathrm{ZF}$ that admits a safe satisfaction relation. There are unboundedly many $\alpha \in \mathrm{Ord}^\mathcal{M}$, such that $\mathcal{M}_\alpha \preceq \mathcal{M}.$ 
\end{thm}

The forward direction of the following result follows from Theorems 1.3 and 3.4 in \cite{Sch78}. The converse is easier, it follows by overspill from the observation that a model of  $\mathrm{ZF}$ with a safe satisfaction relation codes any recursive type.

\begin{thm}\label{rec sat char}
Let $\mathcal{M} \models \mathrm{ZF}$ be countable. The following conditions are equivalent:
\begin{enumerate}[{\normalfont (a)}]
\item\label{rec sat char rec} $\mathcal{M}$ is recursively saturated.
\item\label{rec sat char truth} $\mathcal{M}$ is $\omega$-non-standard and admits a safe satisfaction relation.
\end{enumerate}
Moreover, even if $\mathcal{M}$ is not assumed countable, we have  (\ref{rec sat char truth}) $\Rightarrow$ (\ref{rec sat char rec}).
\end{thm}

\begin{dfn}
	Let $i : \mathcal{M} \rightarrow \mathcal{M}$ be a self-embedding of a model $\mathcal{M}$ of $\mathrm{KP}$.
	\begin{itemize}
		\item $x \in \mathcal{M}$ is a {\em fixed point} of $i$, if $i(x) = x$. The substructure of $\mathcal{M}$ of fixed points of $i$ is denoted $\mathrm{Fix}(i)$. 
		\item $X \subseteq \mathcal{M}$ is {\em pointwise fixed} by $i$, if every $x \in X$ is fixed by $i$. $x \in \mathcal{M}$ is {\em pointwise fixed} by $i$ (or an {\em initial fixed point} of $i$), if $x_\mathcal{M}$ is pointwise fixed by $i$. The substructure of $\mathcal{M}$ of elements pointwise fixed by $i$, is denoted $\mathrm{Fix}^\mathrm{initial}(i)$. 
		\item $x \in \mathcal{M}$ is a {\em rank-initial  fixed point} of $\mathcal{M}$, if $\{y \in \mathcal{M} \mid \rnk(y) \leq \rnk(x)\}$ is pointwise fixed by $i$. The substructure of $\mathcal{M}$ of rank-initial fixed points of $i$ is denoted $\mathrm{Fix}^\rnk(i)$.
	\end{itemize}
	
	We say that $i$ is {\em contractive on $A \subseteq \mathcal{M}$} if for all $x \in A$, we have $\mathcal{M} \models \rnk(i(x)) < \rnk(x)$.
\end{dfn}

Assume that $\mathcal{M}$ is extensional and $i : \mathcal{M} \rightarrow \mathcal{M}$ is initial. Then $x \in \mathcal{M}$ is a fixed point of $i$ if it is pointwise fixed by $i$. It follows that 
$$\mathrm{Fix}(i) \supseteq \mathrm{Fix}^\mathrm{initial}(i) \supseteq \mathrm{Fix}^\rnk(i).$$

\begin{lemma}\label{rank-initial Fix is Sigma_1}
	Suppose that $\mathcal{M} \models \mathrm{KP}^\mathcal{P}$ and that $i$ is a rank-initial self-embedding of $\mathcal{M}$ such that $\mathcal{S} = \mathcal{M}\restriction_{\mathrm{Fix}(i)}$ is a rank-initial substructure of $\mathcal{M}$. Then $\mathcal{S} \preceq_{\Sigma_1^\mathcal{P}} \mathcal{M}$.
\end{lemma}
\begin{proof}
	We verify $\mathcal{S} \preceq_{\Sigma_1^\mathcal{P} } \mathcal{M}$ using the Tarski-Vaught Test (it applies since $\Sigma_1^\mathcal{P}$ is closed under subformulae). Let $\delta(x, \vec{y}) \in \Delta_0^\mathcal{P}[x, \vec{y}]$, let $\vec{s} \in \mathcal{S}$, and assume that $\mathcal{M} \models \exists x . \delta(x, \vec{s})$. We shall now work in $\mathcal{M}$: Let $\xi$ be the least ordinal such that $\exists x \in V_{\xi + 1} . \delta(x, \vec{s})$. We shall show that $i(\xi) = \xi$. Suppose not, then either $i(\xi) < \xi$ or $i(\xi) > \xi$. If $i(\xi) < \xi$, then $\exists x \in V_{i(\xi) + 1} . \delta(x, \vec{s})$, contradicting that $\xi$ is the least ordinal with this property. If $i(\xi) > \xi$, then by rank-initiality there is an ordinal $\zeta < \xi$ such that $i(\zeta) = \xi$. But then $\exists x \in V_{\zeta + 1} . \delta(x, \vec{s})$, again contradicting that $\xi$ is the least ordinal with this property.
	
	By $\Delta_0^\mathcal{P}$-Separation in $\mathcal{M}$, let $D = \{ x \in V_{\xi + 1} \mid \delta(x, \vec{s}) \}^\mathcal{M}$. Since $i$ is $\Delta_0^\mathcal{P}$-elementary and $\xi, \vec{s} \in \mathrm{Fix}(i)$, we have 
	\[
	\mathcal{M} \models \big( (\rnk(x) = \xi \wedge \delta(x, \vec{s})) \leftrightarrow (\rnk(i(x)) = \xi \wedge \delta(i(x), \vec{s})) \big).
	\]
	It immediately follows that $i(D) \subseteq D$. But by rank-initiality, every $x$ of rank $\xi$ in $\mathcal{M}$ is a value of $i$, so we even get that $i(D) = D$. Let $d \in D$. By initiality and $D \in \mathrm{Fix}(i)$, we have $d \in \mathrm{Fix}(i)$; and by construction of $D$, $\mathcal{M} \models \delta(d, \vec{s})$, as desired.
\end{proof}

\begin{lemma}\label{rank-initial Fix is elementary}
	Suppose that $\mathcal{M} \models \mathrm{KP}^\mathcal{P}$ has definable Skolem functions and that $i$ is an automorphism of $\mathcal{M}$ such that $\mathcal{S} = \mathcal{M}\restriction_{\mathrm{Fix}(i)}$. Then $\mathcal{S} \preceq \mathcal{M}$.
\end{lemma}
\begin{proof}
Again, we apply The Tarski Test. Let $\phi(x, \vec{y}) \in \mathcal{L}^0$, let $\vec{s} \in \mathcal{S}$, and assume that $\mathcal{M} \models \exists x . \phi(x, \vec{s})$. Let $m \in \mathcal{M}$ be a witness of this fact. Let $f$ be a Skolem function for $\phi(x, \vec{y})$, defined in $\mathcal{M}$ by a formula $\psi(x, \vec{y})$. Then $\mathcal{M} \models \psi(m, \vec{s})$, and since $i$ is an automorphism fixing $\mathcal{S}$ pointwise, $\mathcal{M} \models \psi(i(m), \vec{s})$. But $\psi$ defines a function, so $\mathcal{M} \models m = i(m)$, whence $m \in \mathcal{S}$ as desired.
\end{proof}

\section{Embeddings between models of set theory}\label{Existence of embeddings between models of set theory}

In \S 4 of \cite{Fri73}, a back-and-forth technique was pioneered that utilizes partial satisfaction relations and the ability of non-standard models to code types over themselves (as indicated in Lemma \ref{hierarchical type coded}). Here we will prove refinements of set theoretic results in \S 4 of \cite{Fri73}, as well as generalizations of arithmetic results in \cite{BE18} and \cite{Res87b} to set theory. We will do so by casting the results in the conceptual framework of forcing. We do so because: 
\begin{itemize}
	\item The conceptual framework of forcing allows a {\em modular} design of the proofs, clarifying which assumptions are needed for what, and whereby new pieces can be added to a proof without having to re-write the other parts. So it serves as an efficient bookkeeping device.
	\item It enables us to look at these results from a different angle, and potentially apply theory that has been developed for usage in forcing.
\end{itemize}

\begin{lemma}\label{Friedman lemma}
	Let $\mathcal{M} \models \mathrm{KP}^\mathcal{P} + \Sigma_1^\mathcal{P}\textnormal{-Separation}$ and $\mathcal{N} \models \mathrm{KP}^\mathcal{P}$ be countable and non-standard, and let $\mathcal{S}$ be a common rank-cut of $\mathcal{M}$ and $\mathcal{N}$. Moreover, let $\mathbb{P} = \llbracket \mathcal{M} \preceq_{\Sigma_1^\mathcal{P}, \mathcal{S}}^{<\omega} \mathcal{N}_\beta \rrbracket$ and let $\beta \in \mathrm{Ord}^\mathcal{N} \setminus \mathcal{S}$.
	\begin{enumerate}[{\normalfont (a)}]
		\item\label{Friedman lemma forwards} If $\mathrm{SSy}_\mathcal{S}(\mathcal{M}) \subseteq \mathrm{SSy}_{\mathcal{S}}(\mathcal{N})$, then
		\[
		\mathcal{C}_{m} =_\df \big\{ f \in \mathbb{P} \mid m \in \dom(f\hspace{2pt}) \big\}
		\]
		is dense in $\mathbb{P}$, for each $m \in \mathcal{M}$.
		\item\label{Friedman lemma backwards} If $\mathrm{SSy}_\mathcal{S}(\mathcal{M}) = \mathrm{SSy}_{\mathcal{S}}(\mathcal{N})$, then 
		\begin{align*}
		\mathcal{D}_{m, n} =_\df \big\{ f \in \mathbb{P} \mid & m \in \dom(f\hspace{2pt}) \wedge ((\mathcal{N} \models \rnk(n) \leq \rnk(f(m))) \rightarrow n \in \mathrm{image}(f\hspace{2pt})) \big\}
		\end{align*}		
		is dense in $\mathbb{P}$, for each $m \in \mathcal{M}$ and $n \in \mathcal{N}$.
		\item\label{Friedman lemma downwards} If $\mathcal{N} = \mathcal{M}$, then 
		\[
		\mathcal{E}_{\alpha} =_\df \big\{ f \in \mathbb{P} \mid \exists m \in \dom(f\hspace{2pt}) . (f\hspace{2pt}(m) \neq m \wedge \mathcal{M} \models \rnk(m) = \alpha) \big\}
		\]
		is dense in $\mathbb{P}$, for each $\alpha \in \mathrm{Ord}^\mathcal{M} \setminus \mathcal{S}$.
 	\end{enumerate}	
\end{lemma}
Note that $\mathcal{N}_\beta$ is rank-initial in $\mathcal{N}$, so by absoluteness of $\Delta_0^\mathcal{P}$-formulas over rank-initial substructures, we have for any $n \in \mathcal{N}_\beta$, for any $s \in \mathcal{S}$ and for any $\delta(x, y, z) \in \Delta_0^\mathcal{P}[x, y, z]$, that
\[
\mathcal{N} \models \exists x \in V_\beta . \delta(x, n, s) \Leftrightarrow \mathcal{N}_\beta \models \exists x . \delta(x, n, s). \tag{$\dagger$}
\]
\begin{proof}
	If $\mathbb{P} = \varnothing$, then the result is trivial, so assume that $\mathbb{P} \neq \varnothing$. This means in particular that $\mathrm{Th}_{\Sigma_1^\mathcal{P}, \mathcal{S}}(\mathcal{M}) \subseteq \mathrm{Th}_{\Sigma_1^\mathcal{P}, \mathcal{S}}(\mathcal{N}_\beta)$.
	
	(\ref{Friedman lemma forwards}) Let $g \in \mathbb{P}$. Unravel $g$ as a $\gamma$-sequence of ordered pairs $\langle m_\xi, n_\xi\rangle_{\xi < \gamma}$, where $\gamma < \omega$. Let $m_\gamma \in \mathcal{M}$ be arbitrary. We need to find $f$ in $\mathbb{P}$ extending $g$, such that $m_\gamma \in \dom(f\hspace{2pt})$.
	
	Using $\mathrm{Sat}_{\Sigma_1^\mathcal{P}}$, we have by Lemma \ref{hierarchical type coded} and $\Sigma_1^\mathcal{P}$-Separation that there is a code $c$ in $\mathcal{M}$ for 
	\begin{align*}
	\big\{\langle \delta, s\rangle \mid & \delta \in \Delta_0^\mathcal{P}[x, \langle y_\xi \rangle_{\xi < \gamma}, y, z] \wedge s \in \mathcal{S} \wedge \\
	& \mathcal{M} \models \exists x . \delta(x, \langle m_\xi \rangle_{\xi < \gamma}, m_\gamma, s)\big\}.
	\end{align*}
	By $\mathrm{SSy}(\mathcal{M}) \subseteq \mathrm{SSy}(\mathcal{N}_\beta)$, this set has a code $d$ in $\mathcal{N}_\beta$ as well.
	We define the formula
	\[
	\phi(\zeta, p, \langle y_\xi \rangle_{\xi < \gamma + 1}) \equiv \exists y . \forall  \langle \delta, t\rangle  \in p \cap V_\zeta . \exists x . \mathrm{Sat}_{\Delta_0^\mathcal{P}}( \delta, x, \langle y_\xi \rangle_{\xi < \gamma}, y, t).
	\]
	Since $\mathrm{Sat}_{\Delta_0^\mathcal{P}} \in \Delta_1^\mathcal{P}$, we have $\phi \in \Sigma_1^\mathcal{P}$. For every ordinal $\zeta \in \mathcal{S}$, we have $c \cap V_\zeta = d \cap V_\zeta \in \mathcal{S}$, and as witnessed by $m_{\gamma}$, we have $\mathcal{M} \models \phi(\zeta, c, \langle m_\xi \rangle_{\xi < \gamma})$. So by $g \in \mathbb{P}$, we have $\mathcal{N}_\beta \models \phi(\zeta, d, \langle n_\xi \rangle_{\xi < \gamma})$ for every ordinal $\zeta \in \mathcal{S}$. Since $\mathcal{S}$ is topless in $\mathcal{N}_\beta$, there is by $\Sigma_1^\mathcal{P}$-Overspill an ordinal $\nu$ in $\mathcal{N}_\beta \setminus \mathcal{S}$, such that $\mathcal{N}_\beta \models \phi(\nu, d, \langle n_\xi \rangle_{\xi < \gamma})$. Set $n_{\gamma}$ to some witness of this fact. Put $f = g \cup \{\langle m_\gamma, n_\gamma \rangle \}$. We proceed to verify that $f$ is $\Sigma_1^\mathcal{P}$-elementary. Let $s \in \mathcal{S}$ and let $\delta(x, \langle y_\xi \rangle_{\xi < \gamma+1}, z) \in \Delta_0^\mathcal{P}[x, \langle y_\xi \rangle_{\xi < \gamma+1}, z]$. Now, as desired,
	\begin{align*}
	\mathcal{M} \models \exists x . \delta(x, \langle m_\xi \rangle_{\xi < \gamma + 1}, s) & \Rightarrow \langle\delta, s\rangle \in d \\
	& \Rightarrow \mathcal{N}_\beta \models \exists x . \delta(x, \langle n_\xi \rangle_{\xi < \gamma + 1}, s).
	\end{align*}
	The second implication follows from the choice of $n_\gamma$ and Theorem \ref{sat_thm}.
	
	(\ref{Friedman lemma backwards}) Let $g \in \mathbb{P}$. Unravel $g$ as a $\gamma$-sequence of ordered pairs $\langle m_\xi, n_\xi\rangle_{\xi < \gamma}$, where $\gamma < \omega$. Let $\xi_0 < \gamma$, and suppose that $n_\gamma \in \mathcal{N}_\beta$, such that $\mathcal{N}_\beta \models \rnk(n_\gamma) \leq \rnk(g(m_{\xi_0}))$. By (\ref{Friedman lemma forwards}), it suffices to find $f$ in $\mathbb{P}$ extending $g$, such that $n_\gamma \in \mathrm{image}(f\hspace{2pt})$.
	
	By $\mathcal{N} \models \Delta_1\textnormal{-Separation}$, ($\dagger$) and Lemma \ref{hierarchical type coded}, there is a code $d'$ in $\mathcal{N}_\beta$ for 
	\begin{align*}
	\{\langle\delta, s\rangle \mid & \delta \in \Delta_0^\mathcal{P}[x, \langle y_\xi \rangle_{\xi < \gamma}, y, z] \wedge s \in \mathcal{S} \wedge \\
	& \mathcal{N}_\beta \models \forall x . \delta(x, \langle n_\xi \rangle_{\xi < \gamma}, n_\gamma, s)\}.
	\end{align*}
	By $\mathrm{SSy}_\mathcal{S}(\mathcal{M}) \supseteq \mathrm{SSy}_\mathcal{S}(\mathcal{N})$, let $c'$ be its code in $\mathcal{M}$. We define the formula
	\[
	\psi(\zeta, p, \langle y_\xi \rangle_{\xi < \gamma}) \equiv \phantom{.} \exists y \subseteq V_{\mathrm{rank}(y_{\xi_0})}. \forall \langle\delta, t\rangle  \in p \cap V_\zeta . \forall x . \mathrm{Sat}_{\Delta_0^\mathcal{P}}( \delta , x, \langle y_\xi \rangle_{\xi < \gamma}, y, t).
	\]
	Since $\mathrm{Sat}_{\Delta_0^\mathcal{P}}$ and the function  $(u \mapsto V_{\mathrm{rank}(u)})$ are $\Delta_1^\mathcal{P}$, we have that $\psi$ is $\Pi_1^\mathcal{P}$. Moreover, $d' \cap V_\zeta = c' \cap V_\zeta$. Note that for every ordinal $\zeta \in \mathcal{S}$, we have $\mathcal{N}_\beta \models \psi(\zeta, d', \langle n_\xi \rangle_{\xi < \gamma})$, as witnessed by $n_{\gamma}$. It follows from $g \in \mathbb{P}$ that $\psi(\zeta, c', \langle m_\xi \rangle_{\xi < \gamma})$ is satisfied in $\mathcal{M}$ for every ordinal $\zeta \in \mathcal{S}$, whence by $\Pi_1^\mathcal{P}$-Overspill we have $\mathcal{M} \models \psi(\mu, c', \langle m_\xi \rangle_{\xi < \gamma})$, for some ordinal $\mu \in \mathcal{M} \setminus \mathcal{S}$. (It follows from $\Sigma_1$-Separation, Proposition \ref{KPP found} and Lemma \ref{overspill} that $\mathcal{M}$ satisfies $\Pi_1^\mathcal{P}$-Overspill.) Let $m_\gamma$ be some witness of this fact, and put $f = g \cup \{\langle m_\gamma, n_\gamma \rangle \}$. We proceed to verify that $f \in \mathbb{P}$. Let $s \in \mathcal{S}$ and let $\delta(x, \langle y_\xi \rangle_{\xi < \gamma+1}, z) \in \Delta_0^\mathcal{P}[x, \langle y_\xi \rangle_{\xi < \gamma+1}, z]$. Now, as desired,
	\begin{align*}
	\mathcal{N}_\beta \models \forall x . \delta(x, \langle n_\xi \rangle_{\xi < \gamma + 1}, s) & \Rightarrow \langle\delta, s\rangle \in c' \\
	& \Rightarrow \mathcal{M} \models \forall x . \delta(x, \langle m_\xi \rangle_{\xi < \gamma + 1}, s).
	\end{align*}
	The second implication follows from the choice of $m_\gamma$ and Theorem \ref{sat_thm}.
	
	(\ref{Friedman lemma downwards}) Let $\alpha \in \mathrm{Ord}^\mathcal{M} \setminus \mathcal{S}$, and let $g \in \mathbb{P}$. Unravel $g$ as a $\gamma$-sequence of ordered pairs $\langle m_\xi, n_\xi\rangle_{\xi < \gamma}$, where $\gamma < \omega$. We need to find $m_\gamma \neq n_\gamma$, such that $\mathcal{M} \models \rnk(m_\gamma) = \alpha$ and $g \cup \{\langle m_\gamma, n_\gamma \rangle \} \in \mathbb{P} = \llbracket \mathcal{M} \preceq_{\Sigma_1^\mathcal{P}, \mathcal{S}}^{< \omega} \mathcal{N}_\beta \rrbracket$. Note that by rank-initiality and toplessness, there is $\alpha' \in \mathrm{Ord}^\mathcal{M} \setminus \mathcal{S}$, such that $(\alpha' + 3 \leq \alpha)^\mathcal{M}$ and $\mathcal{M}_{\alpha'} \supseteq \mathcal{S}$. 
	
	We proceed to work in $\mathcal{M}$: The set $V_{\alpha+1} \setminus V_\alpha$ of sets of rank $\alpha$ has cardinality $\beth_{\alpha+1}$, while the set $\mathcal{P}(V_{\alpha'} \times V_{\alpha'}) \subseteq V_{\alpha' + 3}$ is strictly smaller, of cardinality less than or equal to $\beth_{\alpha' + 3}$. (Here we used $\mathcal{M} \models \textnormal{Powerset}$, and the recursive definition $\beth_0 = 0$, $\beth_{\xi + 1} = 2^{\beth_\xi}$, $\beth_\xi = \sup\{\beth_\zeta \mid \zeta < \xi\}$ for limits $\xi$.) 
	We define a function $t : V_{\alpha+1} \setminus V_\alpha \rightarrow V_{\alpha'+3}$ by
	\begin{align*}
	t(v) = \{ \langle \delta, s \rangle \mid & \delta \in \Delta_0^\mathcal{P}[x, \langle y_\xi \rangle_{\xi < \gamma}, y_\gamma, z] \wedge s \in \mathcal{S} \wedge \\
	& \exists x . \mathrm{Sat}_{\Delta_0^\mathcal{P}}(\delta, x, \langle m_\xi \rangle_{\xi < \gamma}, v, s) \},
	\end{align*}
	for each $v \in V_{\alpha+1} \setminus V_\alpha$. $t$ exists by $\Sigma_1^\mathcal{P}$-Separation. Since $t$ has a domain of strictly larger cardinality than its co-domain, there are $m, m'$ of rank $\alpha$, such that $m \neq m'$ and $t(m) = t(m')$.
	
	We return to working in the meta-theory: $m$ and $m'$ have the same $\Sigma_1^\mathcal{P}$-type with parameters in $\mathcal{S} \cup \langle m_\xi \rangle_{\xi < \gamma}$. In other words, for every $s \in \mathcal{S}$ and every $\delta(x, \langle y_\xi \rangle_{\xi < \gamma}, y_\gamma, z) \in  \Delta_0^\mathcal{P}[x, \langle y_\xi \rangle_{\xi < \gamma}, y_\gamma, z]$, we have 
	\[
	\mathcal{M} \models \exists x . \delta(x, \langle m_\xi \rangle_{\xi < \gamma}, m, s) \leftrightarrow \exists x . \delta(x, \langle m_\xi \rangle_{\xi < \gamma}, m', s). \tag{$\dagger$}
	\]
	On the other hand, by (\ref{Friedman lemma forwards}) and by $g \in \mathbb{P}$, there are $n$, $n'$, such that for every $s \in \mathcal{S}$ and every $\delta(x, \langle y_\xi \rangle_{\xi < \gamma}, y_\gamma, y_{\gamma + 1}, z) \in  \Delta_0^\mathcal{P}[x, \langle y_\xi \rangle_{\xi < \gamma}, y_\gamma, y_{\gamma + 1}, z]$, we have
	\[
	\mathcal{M} \models \exists x . \delta(x, \langle m_\xi \rangle_{\xi < \gamma}, m, m', s) \rightarrow \exists x \in V_\beta . \delta(x, \langle n_\xi \rangle_{\xi < \gamma}, n, n', s). \tag{$\ddagger$}
	\]
	By ($\ddagger$), $n \neq n'$, whence $m \neq n$ or $m \neq n'$. If $m \neq n$, then $g \cup \{ \langle m, n \rangle \} \in \mathcal{E}_\alpha$ by ($\ddagger$). If $m \neq n'$, then by ($\dagger$) and ($\ddagger$),
	\begin{align*}
	\mathcal{M} \models \exists x . \delta(x, \langle m_\xi \rangle_{\xi < \gamma}, m, s) &\Leftrightarrow \mathcal{M} \models \exists x . \delta(x, \langle m_\xi \rangle_{\xi < \gamma}, m', s) \\
	&\Rightarrow \mathcal{M} \models \exists x \in V_\beta . \delta(x, \langle n_\xi \rangle_{\xi < \gamma}, n', s),
	\end{align*}
	so $g \cup \{ \langle m, n' \rangle \} \in \mathcal{E}_\alpha$. In either case we are done.
\end{proof}

Based on this Lemma, we can prove a theorem that refines results in \S 4 of \cite{Fri73}. If $i, j : \mathcal{M} \rightarrow \mathcal{N}$ are rank-initial embeddings between models of $\mathrm{KP}^\mathcal{P}$, then we write $i <^\mathrm{rank} j$ to indicate that $\mathcal{N}\restriction_{\mathrm{image}(i)}$ is a rank-initial substructure of $\mathcal{N}\restriction_{\mathrm{image}(j)}$.

\begin{thm}[Friedman-style]\label{Friedman thm}
Let $\mathcal{M} \models \mathrm{KP}^\mathcal{P} + \Sigma_1^\mathcal{P}\textnormal{-Separation}$ and $\mathcal{N} \models \mathrm{KP}^\mathcal{P}$ be countable and non-standard, and let $\mathcal{S}$ be a shared rank-cut of $\mathcal{M}$ and $\mathcal{N}$. Moreover, let $m_0 \in \mathcal{M}$, let $n_0 \in \mathcal{N}$, and let $\beta \in \mathrm{Ord}^\mathcal{N}$. Then the following are equivalent:
\begin{enumerate}[{\normalfont (a)}]
\item\label{Friedman thm emb} There is a rank-initial embedding $i : \mathcal{M} \rightarrow \mathcal{N}$, fixing $\mathcal{S}$ pointwise, such that $i(m_0) = n_0$ and $i(\mathcal{M}) \subseteq \mathcal{N}_\beta$.
\item\label{Friedman thm pres} $\mathrm{SSy}_\mathcal{S}(\mathcal{M}) = \mathrm{SSy}_{\mathcal{S}}(\mathcal{N})$, and for all $s \in \mathcal{S}$ and $\delta(x, y, z) \in \Delta_0^\mathcal{P}[x, y, z]$:
$$\mathcal{M} \models \exists x . \delta(x, m_0, s) \Rightarrow \mathcal{N} \models \exists x \in V_\beta . \delta(x, n_0, s).$$
\item\label{Friedman thm extra continuum} There is a map $g \mapsto i_g$, from sequences $g : \omega \rightarrow 2$, to embeddings $i_g : \mathcal{M} \rightarrow \mathcal{N}$ satisfying \textnormal{(\ref{Friedman thm emb})}, such that for any $g <^\mathrm{lex} g' : \omega \rightarrow 2$, we have $i_g <^\rnk i_{g'}$.
\item\label{Friedman thm extra topless} There is a topless embedding $i : \mathcal{M} \rightarrow \mathcal{N}$ satisfying \textnormal{(\ref{Friedman thm emb})}.
\end{enumerate}
\end{thm}

\begin{proof}
Most of the work has already been done for (\ref{Friedman thm emb}) $\Leftrightarrow$ (\ref{Friedman thm pres}). The other equivalences are proved as Lemma \ref{Friedman thm extra} below.

(\ref{Friedman thm emb}) $\Rightarrow$ (\ref{Friedman thm pres}): The first conjunct follows from Proposition \ref{emb pres}(\ref{emb pres SSy_S}). The second conjunct follows from Proposition \ref{emb pres}(\ref{emb pres Delta_0^P}) and that $i(\mathcal{M})$ is rank-initial in $\mathcal{N}_\beta$.

(\ref{Friedman thm pres}) $\Rightarrow$ (\ref{Friedman thm emb}): Let $\mathbb{P} = \llbracket \mathcal{M} \preceq_{\Sigma_1^\mathcal{P}, \mathcal{S}}^{<\omega} \mathcal{N}_\beta \rrbracket$. By the second conjunct of (\ref{Friedman thm pres}), the function $f_0$ defined by $(m_0 \mapsto n_0)$, with domain $\{m_0\}$, is in $\mathbb{P}$. Using Lemma \ref{generic filter existence} and Lemma \ref{Friedman lemma} (\ref{Friedman lemma forwards}, \ref{Friedman lemma backwards}), we obtain a filter $\mathcal{I}$ on $\mathbb{P}$ which contains $f_0$ and is $\{\mathcal{D}_{m, n} \mid m \in \mathcal{M} \wedge n \in \mathcal{N}\}$-generic (and hence $\{\mathcal{C}_m \mid m \in \mathcal{M}\}$-generic as well). Let $i = \bigcup \mathcal{I}$. Since $\mathcal{I}$ is downwards directed, $i$ is a function. Clearly $\mathrm{image}(i) \subseteq \mathcal{N}_\beta$. Since $\mathcal{I}$ is $\{\mathcal{C}_m \mid m \in \mathcal{M}\}$-generic, $i$ has domain $\mathcal{M}$; and since $f_0 \in \mathcal{I}$, $i(m_0) = n_0$. To see that $i$ is rank-initial, let $m \in \mathcal{M}$, and let $n \in \mathcal{N}$ such that $\mathcal{N} \models \rnk(n) \leq \rnk(i(m))$. Since $\mathcal{I} \cap \mathcal{D}_{m, n} \neq \varnothing$, we have that $n$ is in the image of $i$.
\end{proof}

Friedman's theorem is especially powerful in conjunction with the following lemmata.

\begin{lemma}\label{Bound for sigma_1 lemma}
	Let $\mathcal{N} \models \mathrm{KP}^\mathcal{P} + \Sigma_1^\mathcal{P}\textnormal{-Separation}$, let $n \in \mathcal{N}$ and let $\mathcal{S}$ be a bounded substructure of $\mathcal{N}$. Then there is an ordinal $\beta \in \mathcal{N}$, such that for each $s \in \mathcal{S}$, and for each $\delta(x, y, z) \in \Delta_0^\mathcal{P}[x, y, z]$:
	$$(\mathcal{N}, n, s) \models (\exists x . \delta(x, n, s)) \leftrightarrow (\exists x \in V_\beta . \delta(x, n, s)).$$
\end{lemma}
\begin{proof}
	Let $\nu$ be an infinite ordinal in $\mathcal{N}$ such that $\mathcal{S} \subseteq \mathcal{N}_\nu$. We work in $\mathcal{N}$: Let $A = \Delta_0^\mathcal{P}[x, y, z] \times V_\nu$. By Strong $\Sigma_1^\mathcal{P}$-Collection there is a set $B$, such that for all $\langle  \delta , t\rangle \in A$, if $\exists x . \mathrm{Sat}_{\Delta_0^\mathcal{P}}( \delta , x, n, t)$, then there is $b \in B$ such that $\mathrm{Sat}_{\Delta_0^\mathcal{P}}(\delta, b, n, t)$. Setting $\beta = \rnk(B)$, the claim of the lemma follows from the properties of $\mathrm{Sat}_{\Delta_0^\mathcal{P}}$.
\end{proof}

\begin{lemma}\label{Bound for pi_2 lemma}
Let $\mathcal{N} \models \mathrm{KP}^\mathcal{P} + \Sigma_2^\mathcal{P}\textnormal{-Separation}$, let $n \in \mathcal{N}$ and let $\mathcal{S}$ be a bounded substructure of $\mathcal{N}$. Then there are unboundedly many ordinals $\beta \in \mathcal{N}$, such that for each $s \in \mathcal{S}$, and for each $\delta(x, x', y, z) \in \Delta_0^\mathcal{P}[x, x', y, z]$:
	$$(\mathcal{N}, n, s) \models (\forall x . \exists x' . \delta(x, x', n, s)) \rightarrow (\forall x \in V_\beta . \exists x' \in V_\beta . \delta(x, x', n, s)).$$
\end{lemma}
\begin{proof}
Let $\alpha$ be an ordinal in $\mathcal{N}$. Let $\beta_0 >^\mathcal{N} \alpha$ be an infinite ordinal in $\mathcal{N}$ such that $\mathcal{S} \subseteq \mathcal{N}_{\beta_0}$. We work in $\mathcal{N}$: By $\Sigma_2^\mathcal{P}$-Separation (which is equivalent to $\Pi_2^\mathcal{P}$-Separation), let 
$$D = \{ \langle \delta, t \rangle \in \Delta_0^\mathcal{P}[x, x', y, z] \times V_{\beta_0} \mid \forall x . \exists x' . \mathrm{Sat}_{\Delta_0}^\mathcal{P}(\delta, x, x', n, t) \}.$$ 
Recursively, for each $k < \omega$, let $\beta_{k+1}$ be the least ordinal such that 
$$\forall \langle \delta, t \rangle \in D . \big( \forall x \in V_{\beta_k} . \exists x' \in V_{\beta_{k+1}} . \mathrm{Sat}_{\Delta_0^\mathcal{P}}(\delta, x, x', n, s) \big).$$ 
The existence of the set $\{ \beta_k \mid k < \omega \}$ follows from $\Sigma_1^\mathcal{P}$-Recursion, because the functional formula defining the recursive step is $\Sigma_1^\mathcal{P}$, as seen when written out as $\phi(\beta_k, \beta_{k+1}) \wedge \forall \gamma < \beta_{k+1} . \neg \phi(\beta_k, \gamma),$ where $\phi(\beta_k, \beta_{k+1})$ is the formula $\forall \langle \delta, t \rangle \in D . \big( \forall x \in V_{\beta_k} . \exists x' \in V_{\beta_{k+1}} . \mathrm{Sat}_{\Delta_0^\mathcal{P}}(\delta, x, x', n, s) \big)$. Put $\beta = \mathrm{sup}\{ \beta_k \mid k < \omega \}$.

Let $s \in \mathcal{S}$ and let $\delta(x, x', y, z) \in \Delta_0^\mathcal{P}[x, x', y, z]$. To verify that 
$$(\mathcal{N}, n, s) \models (\forall x . \exists x' . \delta(x, x', n, s)) \rightarrow (\forall x \in V_\beta . \exists x' \in V_\beta . \delta(x, x', n, s)),$$ 
we work in $(\mathcal{N}, n, s)$: Suppose that $\forall x . \exists x' . \delta(x, x', n, s)$, and let $x \in V_\beta$. Then $x \in V_{\beta_k}$ for some $k < \omega$. By construction, there is $x' \in V_{\beta_{k+1}}$, such that $\mathrm{Sat}_{\Delta_0^\mathcal{P}}(\delta, x, x', n, s)$. So by the properties of $\mathrm{Sat}_{\Delta_0^\mathcal{P}}$, we have $\delta(x, x', n, s)$, as desired.
\end{proof}

\begin{lemma}\label{Friedman thm extra 2}
	Under the assumptions of Theorem \ref{Friedman thm}, for each embedding $i_1 : \mathcal{M} \rightarrow \mathcal{N}$ satisfying \textnormal{(\ref{Friedman thm emb})} of Theorem \ref{Friedman thm}, there is an embedding $i_0 <^\rnk i_1$ satisfying \textnormal{(\ref{Friedman thm emb})}.
\end{lemma}
\begin{proof}
	By Lemma \ref{Bound for sigma_1 lemma}, there is $\alpha \in \mathrm{Ord}^\mathcal{M}$, such that 
	for all $s \in \mathcal{S}$ and all $\delta(x, y, z) \in \Delta_0^\mathcal{P}[x, y, z]$:
	\[
	\mathcal{M} \models \exists x . \delta(x, m, s) \Leftrightarrow \mathcal{M} \models \exists x \in V_\alpha . \delta(x, m, s).
	\]
	By Proposition \ref{emb pres}(\ref{emb pres Delta_0^P}) applied to $i_1$, we have for all $s \in \mathcal{S}$ and all $\delta(x, y, z) \in \Delta_0^\mathcal{P}[x, y, z]$ that
	\[
	\mathcal{M} \models \exists x \in V_\alpha . \delta(x, m, s) \Rightarrow \mathcal{N} \models \exists x \in V_{i_1(\alpha)} . \delta(x, n, s),
	\]
	and consequently that 
	\[
	\mathcal{M} \models \exists x . \delta(x, m, s) \Rightarrow \mathcal{N} \models \exists x \in V_{i_1(\alpha)} . \delta(x, n, s).
	\]
	So by (\ref{Friedman thm pres}) $\Rightarrow$ (\ref{Friedman thm emb}) of Theorem \ref{Friedman thm}, there is a rank-initial embedding $i_0 : \mathcal{M} \rightarrow \mathcal{N}$, such that $i_0(m) = n$ and $i_0(\mathcal{M}) \subseteq V_{i_1(\alpha)}^\mathcal{N}$. Since $i_1(\alpha) \in i_1(\mathcal{M}) \setminus i_0(\mathcal{M})$, we are done.
\end{proof}

\begin{lemma}\label{Friedman thm extra}
	These statements are equivalent to (\ref{Friedman thm emb}) in Theorem \ref{Friedman thm}:
	\begin{enumerate}[{\normalfont (a)}]
		\setcounter{enumi}{2}
		\item There is a map $g \mapsto i_g$, from sequences $g : \omega \rightarrow 2$, to embeddings $i_g : \mathcal{M} \rightarrow \mathcal{N}$ satisfying \textnormal{(\ref{Friedman thm emb})}, such that for any $g <^\mathrm{lex} g' : \omega \rightarrow 2$, we have $i_g <^\rnk i_{g'}$.
		\item There is a topless embedding $i : \mathcal{M} \rightarrow \mathcal{N}$ satisfying \textnormal{(\ref{Friedman thm emb})}.
	\end{enumerate}
\end{lemma}
\begin{proof}
	It suffices to show that (\ref{Friedman thm emb}) $\Rightarrow$ (\ref{Friedman thm extra continuum}) $\Rightarrow$ (\ref{Friedman thm extra topless}).
	
	(\ref{Friedman thm emb}) $\Rightarrow$ (\ref{Friedman thm extra continuum}): 
	 Let $(a_\xi)_{\xi < \omega}$ and $(b_\xi)_{\xi < \omega}$ be enumerations of $\mathcal{M}$ and $\mathcal{N}$, respectively, with infinitely many repetitions of each element. For each $g : \omega \rightarrow 2$, we shall construct a distinct $i_g : \mathcal{M} \rightarrow \mathcal{N}$. To do so, we first construct approximations of the $i_g$.
	 
	 For any $\gamma < \omega$, we allow ourselves to denote any function $f : \gamma \rightarrow 2$ as an explicit sequence of values $f\hspace{2pt}(0),f\hspace{2pt}(1), \dots, f\hspace{2pt}(\gamma - 1)$. For each $\gamma < \omega$, we shall construct a finite subdomain $D_\gamma \subseteq \mathcal{M}$, and for each $f : \gamma \rightarrow 2$, we shall construct an embedding $i_f$. We do so by this recursive construction on $\gamma < \omega$:
	 \begin{enumerate}
	 	\item For $i_\varnothing : \mathcal{M} \rightarrow \mathcal{N}$, choose any embedding satisfying (\ref{Friedman thm emb}).
	 	\item Put $D_\varnothing = \{m\}$.
	 	\item Suppose that $i_f$ has been constructed for some $f : \gamma \rightarrow 2$, where $\gamma < \omega$. Put $i_{f, 1} = i_f$. Applying Lemma \ref{Friedman thm extra 2} to $i_{f,1}$, with $D_\gamma^\mathcal{M}$ in place of $m$ and with $i_f\hspace{2pt}(D_\gamma^\mathcal{M})$ in place of $n$, we choose an embedding $i_{f, 0} : \mathcal{M} \rightarrow \mathcal{N}$ such that :
	 	\begin{enumerate}[(i)]
	 		\item $i_{f, 0}, i_{f, 1}$ are rank-initial, with all values of rank below $\beta$ in $\mathcal{N}$;
	 		\item $i_{f, 0}\restriction_{D_\gamma} = i_{f, 1}\restriction_{D_\gamma} = i_f\restriction_{D_\gamma}$;
	 		\item $i_{f,0} <^\rnk i_{f,1}$.
	 	\end{enumerate}
	 	\item\label{D_constr} Put $D_{\gamma + 1}$ to be a finite subdomain of $\mathcal{M}$, such that:
	 	\begin{enumerate}[(i)]
	 		\item $D_{\gamma} \subseteq D_{\gamma + 1}$; 
	 		\item $a_\gamma \in D_{\gamma + 1}$;
	 		\item\label{D_constr_rank_init} For each $f : \gamma+1 \rightarrow 2$, if $\big( \rnk(b_\gamma) \leq \sup\{\rnk(i_f(a)) \mid a \in D_\gamma\} \big)^\mathcal{N}$, then $i_{f}^{-1}(b_\gamma) \in D_{\gamma + 1}$ (note that $i_{f}^{-1}(b_\gamma)$ exists by rank-initiality of $i_f$);
	 		\item\label{D_constr_homo} For each $f : \gamma \rightarrow 2$, we have that $i_{f, 1}^{-1}(\nu) \in D_{\gamma + 1}$, for some $\nu \in i_{f, 1}(\mathrm{Ord}^\mathcal{M}) \setminus i_{f, 0}(\mathrm{Ord}^\mathcal{M})$.
	 	\end{enumerate}
	 \end{enumerate}
	 Note that every $a \in \mathcal{M}$ is in $D_\gamma$ for some $\gamma < \omega$. Moreover, for every $\gamma < \omega$, if $f <^\mathrm{lex} f' : \gamma \rightarrow 2$, then $i_f <^\rnk i_{f'}$.
	 
	 Now, for each $g : \omega \rightarrow 2$, define $i_g : \mathcal{M} \rightarrow \mathcal{N}$ by
	 \[
	 	i_g(a) = i_{g\restriction_{\gamma}}(a), 
	 \]
	 for each $a \in \mathcal{M}$, where $\gamma < \omega$ is such that $a \in D_\gamma$.
	 Note that for each $\gamma < \omega$, $i_g\restriction_{D_\gamma} = i_{g\restriction_\gamma}$. We now verify that these $i_g$ have the desired properties. Let $g : \omega \rightarrow 2$.
	 \begin{itemize}
	 	\item $i_g$ is an embedding: Let $\phi(x)$ be a quantifier free formula and let $a \in \mathcal{M}$. Then $a \in D_\gamma$ for some $\gamma < \omega$, so since $i_{g\restriction_{\gamma}}$ is an embedding, $\mathcal{M} \models \phi(a) \Rightarrow \mathcal{N} \models \phi(i_g(a)).$
	 	\item $i_g(m) = n$: $m \in D_\varnothing$ and $i_g\restriction_{D_\varnothing} = i_\varnothing\restriction_{D_\varnothing}$.
	 	\item $i_g(\mathcal{M}) \subseteq \mathcal{N}_\beta$: Let $a \in \mathcal{M}$ and pick $\gamma < \omega$ such that $a \in D_\gamma$. Then $i_g(a) = i_{g\restriction_\gamma}(a) \in \mathcal{N}_\beta$.
	 	\item $i_g$ is rank-initial: Let $a \in \mathcal{M}$ and $b \in \mathcal{N}$, such that $\mathcal{N} \models \rnk(b) \leq \rnk(i_g(a))$. Pick $\gamma < \omega$ such that $a \in D_\gamma$. By (\ref{D_constr_rank_init}), $i_{g\restriction_{\gamma + 1}}^{-1}(b) \in D_{\gamma + 1}$. So $b \in i_g(\mathcal{M})$.
	 	\item If $g <^\mathrm{lex} g' : \omega \rightarrow 2$, then $i_g <^\rnk i_{g'}$: Let $\gamma < \omega$ be the least such that $g(\gamma) < g'(\gamma)$. By (\ref{D_constr_homo}), $i_{g'} >^\rnk i_{g\restriction_{\gamma + 1}} \geq^\rnk i_g$.
	 \end{itemize}
	 
	 (\ref{Friedman thm extra continuum}) $\Rightarrow$ (\ref{Friedman thm extra topless}): Since $\mathcal{N}$ is countable, there are only $\aleph_0$ many ordinals in $\mathcal{N}$ which top a substructure, so by (\ref{Friedman thm extra continuum}) we are done.
\end{proof}

The following two results sharpen results in \S 4 of \cite{Fri73}.

\begin{cor}\label{Friedman cor}
Let $\mathcal{M} \models \mathrm{KP}^\mathcal{P}  + \Sigma_1^\mathcal{P}\textnormal{-Separation}$ and $\mathcal{N} \models \mathrm{KP}^\mathcal{P} + \Sigma_1^\mathcal{P}\textnormal{-Separation}$ be countable and non-standard. Let $\mathcal{S}$ be a common rank-cut of $\mathcal{M}$ and $\mathcal{N}$. Then the following are equivalent:
\begin{enumerate}[{\normalfont (a)}]
\item There is a rank-initial $i : \mathcal{M} \rightarrow \mathcal{N}$, fixing $\mathcal{S}$ pointwise.
\item[{\normalfont (a')}] There is a rank-cut $i : \mathcal{M} \rightarrow \mathcal{N}$, fixing $\mathcal{S}$ pointwise.
\item $\mathrm{SSy}(\mathcal{M}) = \mathrm{SSy}(\mathcal{N})$, and $\mathrm{Th}_{\Sigma_1^\mathcal{P}, S}(\mathcal{M}) \subseteq \mathrm{Th}_{\Sigma_1^\mathcal{P}, S}(\mathcal{N})$.
\end{enumerate}
\end{cor}
\begin{proof}
(\ref{Friedman thm emb}) $\Rightarrow$ (\ref{Friedman thm pres}) is proved just as for Theorem \ref{Friedman thm}.

(\ref{Friedman thm pres}) $\Rightarrow$ (\ref{Friedman thm emb}') follows from Theorem \ref{Friedman thm} by letting $\beta \in \mathcal{N}$ be as obtained from Lemma \ref{Bound for sigma_1 lemma}, and setting $m_0 = \varnothing^\mathcal{M}$ and $n_0 = \varnothing^\mathcal{N}$.
\end{proof}

\begin{thm}\label{Friedman selfembedding}
Let $\mathcal{M} \models \mathrm{KP}^\mathcal{P} + \Sigma_1^\mathcal{P}\textnormal{-Separation}$ be countable and non-standard. Let $\mathcal{S}$ be a rank-cut of $\mathcal{M}$. Then there is a rank-cut $i : \mathcal{M} \rightarrow \mathcal{M}$, fixing $\mathcal{S}$ pointwise, such that
\[
\forall \alpha \in \mathrm{Ord}^\mathcal{M} \setminus \mathcal{S} . \exists m \in \mathcal{M} . (\rnk^\mathcal{M}(m) = \alpha \wedge i(m) \neq m).
\]
\end{thm}
\begin{proof}
Let $\beta \in \mathcal{M}$ be be the ordinal bound obtained from Lemma \ref{Bound for sigma_1 lemma}, and let $\mathbb{P} = \llbracket \mathcal{M} \preceq_{\Sigma_1^\mathcal{P}, \mathcal{S}}^{<\omega} \mathcal{M}_\beta \rrbracket$. Put $m_0 = n_0 = \varnothing^\mathcal{M}$. We adjust the proof of (\ref{Friedman thm pres}) $\Rightarrow$ (\ref{Friedman thm emb}) in Theorem \ref{Friedman thm}, by setting $\mathcal{I}$ to be a $\{\mathcal{C}_m \mid m \in \mathcal{M}\} \cup \{\mathcal{D}_{m, n} \mid m \in \mathcal{M} \wedge n \in \mathcal{N}\} \cup \{\mathcal{E}_{\alpha} \mid \alpha \in \mathrm{Ord}^\mathcal{M} \setminus \mathcal{S}\}$-generic filter on $\mathbb{P}$ (utilizing Lemma \ref{Friedman lemma} (\ref{Friedman lemma downwards})). Put $i = \bigcup \mathcal{I}$. It only remains to verify that
\[
\forall \alpha \in \mathrm{Ord}^\mathcal{M} \setminus \mathcal{S} . \exists m \in \mathcal{M} . (\rnk^\mathcal{M}(m) = \alpha \wedge i(m) \neq m).
\]
But this follows from that $\mathcal{I}$ intersects $\mathcal{E}_\alpha$, for each $\alpha \in \mathrm{Ord}^\mathcal{M} \setminus \mathcal{S}$.
\end{proof}

The result above says in particular that every countable non-standard model of $\mathrm{KP}^\mathcal{P} + \Sigma_1^\mathcal{P}\textnormal{-Separation}$ has a proper rank-initial self-embedding. As a remark, there is a related theorem by Hamkins, where no initiality is required from the embedding. In particular, it is shown in \cite{Ham13} that every countable model $\mathcal{M}$ of KP has a self-embedding $i : \mathcal{M} \rightarrow L^\mathcal{M}$, where $L^\mathcal{M} \subseteq \mathcal{M}$ is the G\"odel constructible universe of $\mathcal{M}$. 

The extra clause in the above Corollary, ensuring that something is moved on every rank above the cut, enables the following characterization.

\begin{thm}[Bahrami-Enayat-style]\label{characterize topless substructure}
	Let $\mathcal{M} \models \mathrm{KP}^\mathcal{P}  + \Sigma_1^\mathcal{P}\textnormal{-Separation}$ be countable and non-standard, and let $\mathcal{S}$ be a topless substructure of $\mathcal{M}$. The following are equivalent:
	\begin{enumerate}[{\normalfont (a)}]
		\item\label{characterize topless substructure emb} $\mathcal{S} = \mathrm{Fix}^\rnk(i)$, for some rank-initial self-embedding $i : \mathcal{M} \rightarrow \mathcal{M}$.
		\item[{\normalfont (\ref{characterize topless substructure emb}')}] $\mathcal{S} = \mathrm{Fix}^\rnk(i)$, for some proper topless rank-initial self-embedding $i : \mathcal{M} \rightarrow \mathcal{M}$.
		\item\label{characterize topless substructure substr} $\mathcal{S}$ is rank-initial in $\mathcal{M}$.
	\end{enumerate}
\end{thm}
\begin{proof}
	(\ref{characterize topless substructure emb}) $\Rightarrow$ (\ref{characterize topless substructure substr}) is immediate from the definition of $\mathrm{Fix}^\rnk$.
	
	(\ref{characterize topless substructure substr}) $\Rightarrow$ (\ref{characterize topless substructure emb}') follows from Corollary \ref{Friedman selfembedding}.
\end{proof}

\begin{thm}[Wilkie-style]\label{Wilkie theorem}
	Suppose that $\mathcal{M} \models \mathrm{KP}^\mathcal{P}  + \Sigma_1^\mathcal{P}\textnormal{-Separation}$ $+ \Pi_2^\mathcal{P}\textnormal{-Foundation}$ and $\mathcal{N} \models \mathrm{KP}^\mathcal{P}$ are countable and non-standard. Let $\mathcal{S}$ be a common rank-cut of $\mathcal{M}$ and $\mathcal{N}$, and let $\beta \in \mathrm{Ord}^{\mathcal{N}}$. Then the following are equivalent:
	\begin{enumerate}[{\normalfont (a)}]
		\item\label{Wilkie emb} For any ordinal $\alpha <^\mathcal{N} \beta$, there is a rank-initial embedding $i : \mathcal{M} \rightarrow \mathcal{N}$, fixing $\mathcal{S}$ pointwise, such that  $\mathcal{N}_\alpha \subseteq i(\mathcal{M}) \subseteq \mathcal{N}_\beta$.
		\item[{\normalfont (a')}] For any ordinal $\alpha <^\mathcal{N} \beta$, there is a rank-cut $i : \mathcal{M} \rightarrow \mathcal{N}$, fixing $\mathcal{S}$ pointwise, such that  $\mathcal{N}_\alpha \subseteq i(\mathcal{M}) \subseteq \mathcal{N}_\beta$.
		\item\label{Wilkie pres} $\mathrm{SSy}_\mathcal{S}(\mathcal{M}) = \mathrm{SSy}_\mathcal{S}(\mathcal{N})$, and for all $s \in \mathcal{S}$ and $\delta(x, y, z) \in \Delta_0^\mathcal{P}[x, y, z]$:
$$(\mathcal{M}, s) \models \forall x . \exists y . \delta(x, y, s) \Rightarrow (\mathcal{N}, s) \models \forall x \in V_\beta . \exists y \in V_\beta . \delta(x, y, s).$$
	\end{enumerate}
\end{thm}
\begin{proof}
(\ref{Wilkie emb}) $\Rightarrow$ (\ref{Wilkie pres}): It is easy to see that (\ref{Wilkie emb}) $\Rightarrow$ $\beta$ is a limit ordinal in $\mathcal{N}$. Let $s \in \mathcal{S}$. Let $\delta(x, y, z)$ be $\Delta_0^\mathcal{P}[x, y, z]$ and assume that $\mathcal{M} \models \forall x . \exists y . \delta(x, y, s)$. Given Theorem \ref{Friedman thm}, it only remains to show that $\mathcal{N} \models \forall x \in V_\beta . \exists y \in V_\beta . \delta(x, y, s)$. Let $a \in \mathcal{N}_\beta$ be arbitrary and set $\alpha = (\rnk(a)+1)^\mathcal{N}$. Since $\beta$ is a limit, $\alpha < \beta$. By (\ref{Wilkie emb}), there is a rank-initial embedding $i : \mathcal{M} \rightarrow \mathcal{N}$, such that $a \in i(\mathcal{M}) \subseteq \mathcal{N}_\beta$. Pick $m \in \mathcal{M}$ such that $\mathcal{M} \models \delta(i^{-1}(a), m, s)$. Then $i(m) \in \mathcal{N}_\beta$, and by Proposition \ref{emb pres}, $\mathcal{N} \models \delta(a, i(m), s)$, as desired.

(\ref{Wilkie pres}) $\Rightarrow$ (\ref{Wilkie emb}'): Let $\alpha <^\mathcal{N} \beta$ and let $n = V_\alpha^\mathcal{N}$. Using Lemma \ref{hierarchical type coded} and $\Delta_1^\mathcal{P}$-Collection, let $d$ be a code in $\mathcal{N}$ for the following set:
\[
D = \{\delta(x, y, s) \mid \delta \in \Delta_0^\mathcal{P} \wedge s \in \mathcal{S} \wedge \mathcal{N} \models \forall x \in V_\beta . \delta(x, n, s)\}.
\]

Using $\mathrm{SSy}(\mathcal{M}) = \mathrm{SSy}(\mathcal{N})$ let $c$ be a code for this set in $\mathcal{M}$. Define the formulae
\begin{align*}
\phi_{< \beta}(\zeta) &\equiv \exists y \in V_\beta . \forall \delta \in d \cap V_\zeta . \forall x \in V_\beta . \mathrm{Sat}_{\Delta_0^\mathcal{P}}(\delta, x, y, s) \\
\phi(\zeta) &\equiv \exists y . \forall \delta \in c \cap V_\zeta . \forall x . \mathrm{Sat}_{\Delta_0^\mathcal{P}}(\delta, x, y, s)
\end{align*}
Note that $\phi$ is $\Sigma_2^\mathcal{P}$. Moreover, $c \cap V_\zeta = d \cap V_\zeta$ and $n$ witnesses $\phi_{<\beta}(\zeta)$, for all ordinals $\zeta \in \mathcal{S}$. So by the second conjunct of (\ref{Wilkie pres}), $\mathcal{M} \models \phi(\zeta)$ for all ordinals $\zeta \in \mathcal{S}$, whence by $\Sigma_2^\mathcal{P}$-Overspill, $\mathcal{M} \models \phi(\mu)$ for some ordinal $\mu \in \mathcal{M} \setminus \mathcal{S}$. Letting $m \in \mathcal{M}$ be a witness of this fact, we have $\mathcal{M} \models \forall x . \delta(x, m, s)$, for all $\delta(x, y, s) \in D$. Now (\ref{Wilkie emb}) is obtained by plugging $m$ and $n$ into Theorem $\ref{Friedman thm}$.
\end{proof}

\begin{cor}
Let $\mathcal{M} \models \mathrm{KP}^\mathcal{P}  + \Sigma_1^\mathcal{P}\textnormal{-Separation} + \Pi_2^\mathcal{P}\textnormal{-Foundation}$ and $\mathcal{N} \models \mathrm{KP}^\mathcal{P} + \Sigma_2^\mathcal{P}\textnormal{-Separation}$ be countable and non-standard. Let $\mathcal{S}$ be a common rank-cut of $\mathcal{M}$ and $\mathcal{N}$. Then the following are equivalent:
\begin{enumerate}[{\normalfont (a)}]
\item\label{Wilkie emb 2} For any ordinal $\alpha \in \mathcal{N}$, there is a rank-initial embedding $i : \mathcal{M} \rightarrow \mathcal{N}$, fixing $\mathcal{S}$ pointwise, such that $\mathcal{N}_\alpha \subseteq i(\mathcal{M})$.
\item[{\normalfont (a')}] For any ordinal $\alpha \in \mathcal{N}$, there is a rank-cut $i : \mathcal{M} \rightarrow \mathcal{N}$, fixing $\mathcal{S}$ pointwise, such that $\mathcal{N}_\alpha \subseteq i(\mathcal{M})$.
\item\label{Wilkie pres 2} $\mathrm{SSy}(\mathcal{M}) = \mathrm{SSy}(\mathcal{N})$, and $\mathrm{Th}_{\Pi_2^\mathcal{P}, S}(\mathcal{M}) \subseteq \mathrm{Th}_{\Pi_2^\mathcal{P}, S}(\mathcal{N})$.
\end{enumerate}
\end{cor}
\begin{proof}
(\ref{Wilkie emb 2}) $\Rightarrow$ (\ref{Wilkie pres 2}) is proved just as for Theorem \ref{Wilkie theorem}.

(\ref{Wilkie pres 2}) $\Rightarrow$ (\ref{Wilkie emb 2}') follows from Theorem \ref{Wilkie theorem} by letting $\beta >^\mathcal{N} \alpha$ be as obtained from Lemma \ref{Bound for pi_2 lemma}.
\end{proof}

\begin{cor}\label{Wilkie selfembedding}
Let $\mathcal{M} \models \mathrm{KP}^\mathcal{P} + \Sigma_2^\mathcal{P}\textnormal{-Separation} + \Pi_2^\mathcal{P}\textnormal{-Foundation}$ be countable and non-standard. Let $\mathcal{S}$ be a rank-cut of $\mathcal{M}$. For any $\alpha \in \mathrm{Ord}^\mathcal{M}$, there is a rank-cut $i : \mathcal{M} \rightarrow \mathcal{M}$, fixing $\mathcal{S}$ pointwise, such that $\mathcal{M}_\alpha \subseteq i(\mathcal{M})$ and
\[
\forall \alpha \in \mathrm{Ord}^\mathcal{M} \setminus \mathcal{S} . \exists m \in \mathcal{M} . (\rnk^\mathcal{M}(m) = \alpha \wedge i(m) \neq m).
\]
\end{cor}
\begin{proof}
Let $\mathcal{N} = \mathcal{M}$ and let $\beta >^\mathcal{M} \alpha$ be as obtained from Lemma \ref{Bound for pi_2 lemma}. Then condition (\ref{Wilkie pres}) of Theorem \ref{Wilkie theorem} is satisfied. Repeat the proof of Theorem \ref{Wilkie theorem} (\ref{Wilkie pres}) $\Rightarrow$ (\ref{Wilkie emb}') with $\mathcal{N} = \mathcal{M}$, except that at the last step: apply Theorem \ref{Friedman selfembedding} instead of Theorem \ref{Friedman thm}.
\end{proof}

Now that we have explored necessary and sufficient conditions for constructing embeddings between models, we turn to the question of constructing isomorphisms between models. For this purpose we shall restrict ourselves to recursively saturated models of $\mathrm{ZF}$.

\begin{lemma}\label{rec sat iso lemma}
Let $\mathcal{M}$ and $\mathcal{N}$ be countable recursively saturated models of $\mathrm{ZF}$, and let $\mathcal{S}$ be a common rank-initial $\omega$-topless substructure of $\mathcal{M}$ and $\mathcal{N}$. Moreover, let $\mathbb{P} = \llbracket \mathcal{M} \preceq_\mathcal{S}^{< \omega} \mathcal{N} \rrbracket$.

If $\mathrm{SSy}_\mathcal{S}(\mathcal{M}) \subseteq \mathrm{SSy}_\mathcal{S}(\mathcal{N})$, then 
\[
\mathcal{C}'_m =_\df \{ f \in \mathbb{P} \mid m \in \dom(f\hspace{2pt}) \}
\]
is dense in $\mathbb{P}$, for each $m \in \mathcal{M}$.
\end{lemma}
\begin{proof}
By Theorem \ref{rec sat char}, $\mathcal{M}$ and $\mathcal{N}$ are $\omega$-non-standard and there are expansions $(\mathcal{M}, \mathrm{Sat}^\mathcal{M})$ and $(\mathcal{N}, \mathrm{Sat}^\mathcal{N})$ satisfying condition (\ref{rec sat char truth}) of that theorem. Recall that this condition says that these are satisfaction classes that are correct for all formulae in $\mathcal{L}^0$ of standard complexity, and that the expanded structures satisfy Separation and Replacement for all formulae in the expanded language $\mathcal{L}^0_\mathrm{Sat}$.

Let $g \in \mathbb{P}$. We unravel it as $g = \{\langle m_\xi, n_\xi \rangle \mid \xi < \gamma \}$, for some $\gamma < \omega$. Let $m_\gamma \in \mathcal{M}$ be arbitrary. By $\mathcal{L}^0_\mathrm{Sat}$-Separation, there is a code $c$ in $\mathcal{M}$ for the set
\begin{align*}
\{ & \langle \delta, s \rangle \in (\mathcal{L}^0[\langle x_\xi \rangle_{\xi < \gamma}, x_\gamma, z] \cap \mathcal{S}) \times \mathcal{S} \mid \\
& (\mathcal{M}, \mathrm{Sat}^\mathcal{M}) \models \mathrm{Sat}(\delta, \langle m_\xi \rangle_{\xi < \gamma}, m_\gamma, s) \}.
\end{align*}
Since $\mathrm{SSy}_\mathcal{S}(\mathcal{M}) = \mathrm{SSy}_\mathcal{S}(\mathcal{N})$, this set is also coded by some $d$ in $\mathcal{N}$.

We define a formula:
\begin{align*}
\phi(\zeta, k, \langle x_\xi \rangle_{\xi < \gamma}, q) \equiv & \mathrm{Ord}(\zeta) \wedge k < \omega \wedge \\
& \exists x_\gamma . \forall \delta \in \Sigma_{k} . \forall t \in V_\zeta . \\
& (\mathrm{Sat}(\delta, \langle x_\xi \rangle_{\xi < \gamma}, x_\gamma, t) \leftrightarrow \langle \delta, t \rangle \in q).
\end{align*}
By construction of $c$ and correctness of $\mathrm{Sat}^\mathcal{M}$, we have that 
$$\mathcal{M} \models \phi(\zeta, k, \langle m_\xi \rangle_{\xi < \gamma}, c),$$ 
for each $\zeta \in \mathrm{Ord}^\mathcal{M} \cap \mathcal{S}$ and each $k < \omega = \mathrm{OSP}(\mathcal{M})$. So since $g$ is elementary, and since $\mathcal{M}_\zeta = \mathcal{N}_\zeta$ and $c_\mathcal{M} \cap \mathcal{M}_\zeta = d_\mathcal{N} \cap \mathcal{N}_\zeta$ for each $\zeta \in \mathrm{Ord}^\mathcal{N} \cap \mathcal{S}$, we also have that $\mathcal{N} \models \phi(\zeta, k, \langle n_\xi \rangle_{\xi < \gamma}, d)$ for each $\zeta \in \mathrm{Ord}^\mathcal{N} \cap \mathcal{S}$ and each $k < \omega$. Pick some $\nu \in \mathrm{Ord}^\mathcal{N} \setminus \mathcal{S}$. Now by Overspill on $\mathcal{S}$, for each $k < \omega$ there is $\nu'_k \in \mathrm{Ord}^\mathcal{N} \setminus \mathcal{S}$ such that $(\mathcal{N}, \mathrm{Sat}^\mathcal{N}) \models \nu'_k < \nu \wedge \phi(\nu'_k, k, \langle n_\xi \rangle_{\xi < \gamma}, d)$.

Pick some non-standard $o <^\mathcal{N} \omega^\mathcal{N}$. Working in $\mathcal{N}$, we construct a partial function $(k \mapsto \nu_k) : o \rightarrow \nu + 1$, such that for each $k < o$,
\[
\nu_k = \sup \{ \zeta \mid \zeta < \nu \wedge \phi(\zeta, k, \langle n_\xi \rangle_{\xi < \gamma}, d) \}.
\]
We return to reasoning in the meta-theory. By the Overspill-argument above, this function is total on $\omega$, and $\nu_k \not\in \mathcal{S}$ for each $k < \omega$. Moreover, by logic, $\nu_k \geq \nu_l$ for all $k \leq l < \omega$. So by $\omega$-toplessness, there is $\nu_\infty \in \mathrm{Ord}^\mathcal{N} \setminus \mathcal{S}$, such that for each $k < \omega$, $\nu_\infty <^\mathcal{N} \nu_k$. So for each $k < \omega$, we have $(\mathcal{N}, \mathrm{Sat}^\mathcal{N}) \models \phi(\nu_\infty, k, \langle n_\xi \rangle_{\xi < \gamma}, d)$, whence by Overspill on $\mathrm{WFP}(\mathcal{N})$, there is a non-standard $k_\infty \in^\mathcal{N} \omega^\mathcal{N}$ such that $(\mathcal{N}, \mathrm{Sat}^\mathcal{N}) \models \phi(\nu_\infty, k_\infty, \langle n_\xi \rangle_{\xi < \gamma}, d)$. Let $n_\gamma \in \mathcal{N}$ be a witness of this fact. Note that for all $s \in \mathcal{S}$ and for all $\delta \in \mathcal{L}^0[\langle x_\xi \rangle_{\xi < \gamma}, x_\gamma, z]$,
\[
\mathcal{N} \models \delta(\langle n_\xi \rangle_{\xi < \gamma}, n_\gamma, s) \Leftrightarrow \mathcal{N} \models \langle \delta, s \rangle \in d.
\]

Let $f = g \cup \{\langle m_\gamma, n_\gamma \rangle \}$. We need to show that $f \in \mathcal{C}'_{m_\gamma}$; it only remains to verify that $f$ is elementary. Now observe that for any $s \in \mathcal{S}$, and any formula $\delta(\langle x_\xi \rangle_{\xi < \gamma + 1}, z) \in \mathcal{L}^0[\langle x_\xi \rangle_{\xi < \gamma + 1}, z]$,
\begin{align*}
\mathcal{M} \models \delta(\langle m_\xi \rangle_{\xi < \gamma + 1}, s) &\Leftrightarrow \mathcal{M} \models \langle \delta, s \rangle \in c \\
&\Leftrightarrow \mathcal{N} \models \langle \delta, s \rangle \in d \\
&\Leftrightarrow \mathcal{N} \models \delta(\langle n_\xi \rangle_{\xi < \gamma + 1}, s).
\end{align*}
Therefore, $f \in \mathcal{C}'_{m_\gamma}$ as desired.
\end{proof}

\begin{thm}\label{rec sat iso thm}
Let $\mathcal{M}$ and $\mathcal{N}$ be countable recursively saturated models of $\mathrm{ZF}$, and let $\mathcal{S}$ be a common rank-initial $\omega$-topless substructure of $\mathcal{M}$ and $\mathcal{N}$. Let $m_0 \in \mathcal{M}$ and let $n_0 \in \mathcal{N}$. The following are equivalent:

\begin{enumerate}[{\normalfont (a)}]
	\item There is an isomorphism $i : \mathcal{M} \rightarrow \mathcal{N}$, fixing $\mathcal{S}$ pointwise, such that $i(m_0) = n_0$.
	\item $\mathrm{SSy}_\mathcal{S}(\mathcal{M}) = \mathrm{SSy}_\mathcal{S}(\mathcal{N})$ and $\mathrm{Th}_\mathcal{S}((\mathcal{M}, m_0)) = \mathrm{Th}_\mathcal{S}((\mathcal{N}, n_0))$.
\end{enumerate}
\end{thm}
\begin{proof}
The forward direction is clear since $i$ is an isomorphism.

Let $\mathbb{P} = \llbracket \mathcal{M} \preceq_\mathcal{S}^{< \omega} \mathcal{N} \rrbracket$. Since $\mathrm{Th}_\mathcal{S}((\mathcal{M}, m_0)) = \mathrm{Th}_\mathcal{S}((\mathcal{N}, n_0))$, the function $(m_0 \mapsto n_0)$ is in $\mathbb{P}$. For each $m \in \mathcal{M}$ and each $n \in  \mathcal{N}$, let
\begin{align*}
\mathcal{C}'_m &=_\df \{ f \in \mathbb{P} \mid m \in \dom(f\hspace{2pt}) \}, \\
\mathcal{D}'_n &=_\df \{ f \in \mathbb{P} \mid n \in \mathrm{image}(f\hspace{2pt}) \}. \\
\end{align*}
By Lemma \ref{rec sat iso lemma}, $\mathcal{C}'_m$ and $\mathcal{D}'_n$ are dense in $\mathbb{P}$ for all $m \in \mathcal{M}$ and all $n \in \mathcal{N}$. By Lemma \ref{generic filter existence}, there is a $\mathcal{C}'_m \cup \mathcal{D}'_n$-generic filter $\mathcal{I}$ on $\mathbb{P}$ containing $(m_0 \mapsto n_0)$. Let $i = \bigcup \mathcal{I}$. 

By the genericity, $\dom(i) = \mathcal{M}$ and $\mathrm{image}(i) = \mathcal{N}$. Moreover, by the filter properties, for any $\vec{m} \in \mathcal{M}$, some finite extension $f \in \mathbb{P}$ of $i\restriction_{\vec{m}}$ is in $\mathcal{I}$. So by elementarity of $f$ and arbitrariness of $\vec{m}$, we have that $i$ is an isomorphism.
\end{proof}

The following Theorem is an improvement of Theorem \ref{rec sat reflection}. The proof given here is meant to be simpler and more accessible than the one given in \cite{Res87a}.

\begin{thm}[Ressayre]\label{Ressayre thm}
	Let $\mathcal{M}$ be a countable recursively saturated model of $\mathrm{ZF}$. There are unboundedly many $\alpha \in \mathrm{Ord}^\mathcal{M}$, such that for all $S \in \mathcal{M}_\alpha$ we have $\mathcal{M} \cong_{S_\mathcal{M}} \mathcal{M}_\alpha \preceq \mathcal{M}$.
\end{thm}
\begin{proof}
	Let $\alpha >^\mathcal{M} \omega^\mathcal{M}$ be as obtained from Theorem \ref{rec sat reflection}. Thus $\mathcal{M}_\alpha \preceq \mathcal{M}$. Let $S \in \mathcal{M}_\alpha$ be arbitrary. For each $m \in^\mathcal{M} \omega^\mathcal{M}$, let $\sigma_m = \rnk(S) + m$ as evaluated in $\mathcal{M}$. Since $\mathcal{M}_\alpha \preceq \mathcal{M}$, we have $\sigma_m < \alpha$ for each $m \in^\mathcal{M} \omega^\mathcal{M}$.
	
	Since $\mathcal{M}$ is recursively saturated, it is $\omega$-non-standard. Let $\mathcal{M}_{S, \omega} = \bigcup_{k < \omega} \mathcal{M}_{\sigma_k}$ (note that we take this union only over standard $k$). By Lemma \ref{omega-topless existence}, $\mathcal{M}_{S, \omega}$ is a common rank-initial $\omega$-topless substructure of $\mathcal{M}$ and $\mathcal{M}_\alpha$; and obviously $S_\mathcal{M} \subseteq \mathcal{M}_{S, \omega}$. 
	
	By rank-initiality, $\mathrm{SSy}_{\mathcal{M}_{S, \omega}}(\mathcal{M_\alpha}) = \mathrm{SSy}_{\mathcal{M}_{S, \omega}}(\mathcal{M})$, and by $\mathcal{M}_\alpha \preceq \mathcal{M}$, we have $\mathrm{Th}_{\mathcal{M}_{S, \omega}}(\mathcal{M_\alpha}) = \mathrm{Th}_{\mathcal{M}_{S, \omega}}(\mathcal{M})$. So it follows from Theorem \ref{rec sat iso thm} that $\mathcal{M} \cong_{S_\mathcal{M}} \mathcal{M}_\alpha$.
\end{proof}

\begin{thm}[Ressayre-style]\label{Ressayre characterization Sigma_1-Separation}
	Let $\mathcal{M} \models \mathrm{KP}^\mathcal{P}$ be countable and non-standard. The following are equivalent:
	\begin{enumerate}[{\normalfont (a)}]
		\item\label{Ressayre characterization Sigma_1-Separation Separation} $\mathcal{M} \models \Sigma_1^\mathcal{P} \textnormal{-Separation}.$
		\item\label{Ressayre characterization Sigma_1-Separation Self-embedding} For every $\alpha \in \mathrm{Ord}^\mathcal{M}$, there is a proper rank-initial self-embedding of $\mathcal{M}$ which fixes $\mathcal{M}_\alpha$ pointwise.
	\end{enumerate}
\end{thm}
\begin{proof}
	(\ref{Ressayre characterization Sigma_1-Separation Separation}) $\Rightarrow$ (\ref{Ressayre characterization Sigma_1-Separation Self-embedding}): Let $\mathcal{N} = \bigcup_{\xi \in \mathrm{OSP}(\mathcal{M})} \mathcal{M}_{\alpha + \xi}.$ Note that $\mathcal{N}$ is a rank-cut of $\mathcal{M}$. So by Theorem \ref{Friedman selfembedding}, we are done.
	
	(\ref{Ressayre characterization Sigma_1-Separation Self-embedding}) $\Rightarrow$ (\ref{Ressayre characterization Sigma_1-Separation Separation}): Let $\phi(x) \in \Sigma_1^\mathcal{P}[x]$ and let $a \in \mathcal{M}$. Let $i$ be a rank-initial self-embedding of $\mathcal{M}$ which fixes $\mathcal{M}_{(\rnk(a) + 1)^\mathcal{M}}$ pointwise and which satisfies $i(\mathcal{M}) \subseteq \mathcal{M}_\mu$, for some $\mu \in \mathrm{Ord}^\mathcal{M}$. Let us write $\phi(x)$ as $\exists y . \delta(y, x)$, where $\delta \in \Delta_0^\mathcal{P}$. Since $i$ fixes $a_\mathcal{M}$ pointwise, we have
	\[
	\mathcal{M} \models \forall x \in a . ( \exists y . \delta(y, x) \leftrightarrow \exists y \in V_\mu . \delta(y, x)).
	\]
	So $\{ x \in a \mid \phi(x) \} = \{ x \in a \mid \exists y \in V_\mu . \delta(y, x)) \}$, which exists by $\Delta_0^\mathcal{P}$-Separation.
\end{proof}

\section{Iterated ultrapowers with special self-embeddings}\label{Existence of models of set theory with automorphisms}

It is convenient to fix some objects which will be discussed throughout this section. Fix a countable model $(\mathcal{M}, \mathcal{A}) \models \mathrm{GBC} + \text{``$\mathrm{Ord}$ is weakly compact''}$. Fix $\mathbb{B}$ to be the boolean algebra $\{A \subseteq \mathrm{Ord}^\mathcal{M} \mid A \in \mathcal{A}\}$ induced by $\mathcal{A}$. Fix $\mathbb{P}$ to be the partial order of unbounded sets in $\mathbb{B}$ ordered under inclusion. Fix a filter $\mathcal{U}$ on $\mathbb{P}$. Note that by Proposition \ref{GBC weakly compact global choice}, $(\mathcal{M}, \mathcal{A}) \models \textnormal{Global Choice}$.

$\mathcal{U}$ is $\mathbb{P}$-{\em generic over} $(\mathcal{M}, \mathcal{A})$, or simply $(\mathcal{M}, \mathcal{A})$-{\em generic}, if it intersects every dense subset of $\mathbb{P}$ that is parametrically definable in $(\mathcal{M, A})$. $\mathcal{U}$ is $(\mathcal{M, A})$-{\em complete} if for every $a \in \mathcal{M}$ and every $f : \mathrm{Ord}^\mathcal{M} \rightarrow a_\mathcal{M}$ that is coded in $\mathcal{A}$, there is $b \in a_\mathcal{M}$ such that $f^{-1}(b) \in \mathcal{U}$. Considering the characteristic functions of the classes in $\mathcal{A}$, we can easily see that if $\mathcal{U}$ is $(\mathcal{M, A})$-{\em complete}, then it is an {\em ultrafilter} on $\mathbb{B}$, i.e. for any $A \in \mathbb{B}$, we have $A \in \mathcal{U}$ or $\mathrm{Ord}^\mathcal{M} \setminus A \in \mathcal{U}$. 

Let $P : \mathrm{Ord}^\mathcal{M} \times \mathrm{Ord}^\mathcal{M} \rightarrow \mathrm{Ord}^\mathcal{M}$ be a bijection coded in $\mathcal{A}$. For each $g : \mathrm{Ord}^\mathcal{M} \rightarrow \{0,1\}^\mathcal{M}$ coded in $\mathcal{A}$, and each $\alpha \in \mathrm{Ord}^\mathcal{M}$, define $S^g_\alpha =_\mathrm{df} \{\xi \in \mathrm{Ord}^\mathcal{M} \mid g(P(\alpha, \xi)) = 1\}$. Thus, $g$ may be thought of as coding an $\mathrm{Ord}^\mathcal{M}$-sequence of sets; indeed $(\alpha \mapsto S^g_\alpha) : \mathrm{Ord}^\mathcal{M} \rightarrow \mathbb{B}$. $\mathcal{U}$ is $(\mathcal{M, A})$-{\em iterable} if for every $g : \mathrm{Ord}^\mathcal{M} \rightarrow \{0,1\}^\mathcal{M}$ coded in $\mathcal{A}$, we have $\{\alpha \mid S^g_\alpha \in \mathcal{U}\} \in \mathcal{A}$.

A filter $\mathcal{U}$ is $(\mathcal{M, A})$-{\em canonically Ramsey} if for every $n \in \mathbb{N}$ and $f : [\mathrm{Ord}^\mathcal{M}]^n \rightarrow \mathrm{Ord}^\mathcal{M}$ coded in $\mathcal{A}$, there is $H \in \mathcal{U}$ and $S \subseteq \{1, \dots, n\}$, such that for any $\alpha_1, \dots, \alpha_n$ and $\beta_1, \dots, \beta_n$ in $H$, 
$$f(\alpha_1, \dots, \alpha_n) = f(\beta_1, \dots, \beta_n) \leftrightarrow \forall m \in S . \alpha_m = \beta_m.$$
We say that $f$ is {\em canonical} on $H$.

The following theorem is proved in \cite[p. 48]{Ena04}. Combined with Lemma \ref{generic filter existence}, it establishes the existence of an ultrafilter on $\mathbb{B}$, which is $(\mathcal{M, A})$-complete, $(\mathcal{M, A})$-iterable and $(\mathcal{M, A})$-canonically Ramsey, under the assumption that $(\mathcal{M, A}) \models \text{``$\mathrm{Ord}$ is weakly compact''}$.

\begin{thm}\label{weakly compact gives nice ultrafilter}
Let $(\mathcal{M, A})$ a model of $\mathrm{GBC} + \text{``$\mathrm{Ord}$ is weakly compact''}$. If $\mathcal{U}$ is $(\mathcal{M, A})$-generic, then $\mathcal{U}$ is 
\begin{enumerate}[{\normalfont (a)}]
\item $(\mathcal{M, A})$-complete,
\item $(\mathcal{M, A})$-iterable, and
\item $(\mathcal{M, A})$-canonically Ramsey.
\end{enumerate}
\end{thm}

Using such an ultrafilter, we shall now construct an iterated ultrapower of $(\mathcal{M}, \mathcal{A}$). A more detailed account of this construction is found in \cite{EKM17}.

\begin{constr}\label{constr iter ultra}
Suppose that $\mathcal{U}$ is a non-principal $(\mathcal{M}, \mathcal{A})$-iterable ultrafilter on $\mathbb{B}$. Then for any $n \in \mathbb{N}$, an ultrafilter $\mathcal{U}^n$ can be recursively constructed on $\mathbb{B}^n =_\mathrm{df} \{A \subseteq (\mathrm{Ord}^\mathcal{M})^n \mid A \in \mathcal{A}\}$ as follows:

First, we extend the definition of iterability. An ultrafilter $\mathcal{V}$ on $\mathbb{B}^n$ is $(\mathcal{M}, \mathcal{A})$-{\em iterable} if for any function $(\alpha \mapsto S_\alpha) : \mathrm{Ord}^\mathcal{M} \rightarrow \mathbb{B}^n$ coded in $\mathcal{A}$, we have $\{\alpha \mid S_\alpha \in \mathcal{V}\} \in \mathcal{A}$. 

$\mathcal{U}^0$ is the trivial (principal) ultrafilter $\{\{\langle\rangle\}\}$ on the boolean algebra $\{\langle\rangle, \{\langle\rangle\}\}$, where $\langle\rangle$ is the empty tuple.

For any $n \in \mathbb{N}$, $X \in \mathbb{B}^{n+1}$ and any $\alpha \in \mathrm{Ord}^\mathcal{M}$, define 
$$X_{\alpha} =_\mathrm{df} \{ \langle \alpha_2, \dots, \alpha_{n+1}\rangle \mid \langle \alpha, \alpha_2, \dots, \alpha_{n+1}\rangle \in X\}$$
$$X \in \mathcal{U}^{n+1} \Leftrightarrow_\mathrm{df} \big\{\alpha \mid X_{\alpha} \in \mathcal{U}^n \big\} \in \mathcal{U}.$$

Note that there are other equivalent definitions:
$$
\begin{array}{cl}
& X \in \mathcal{U}^{n} \\
\Leftrightarrow& \big\{\alpha_1 \mid \{ \langle \alpha_2, \dots, \alpha_{n}\rangle \mid \langle \alpha_1, \alpha_2, \dots, \alpha_{n}\rangle \in X\} \in \mathcal{U}^{n-1} \big\} \in \mathcal{U}  \\ 
\Leftrightarrow& \{ \alpha_1 \mid \{ \alpha_2 \mid \dots \{ \alpha_n \mid \langle \alpha_1, \dots, \alpha_n \rangle \in X\} \in \mathcal{U} \} \in \mathcal{U} \dots \} \in \mathcal{U}  \\
\Leftrightarrow& \{ \langle \alpha_1, \dots \alpha_{n-1} \rangle \mid \{ \alpha_n \mid \langle \alpha_1, \dots \alpha_{n} \rangle \in X \} \in \mathcal{U} \} \in \mathcal{U}^{n-1}
\end{array}
$$

By the setup, $\mathcal{U}^1 = \mathcal{U}$, which is an $(\mathcal{M}, \mathcal{A})$-iterable ultrafilter on $\mathbb{B}^1$ by assumption. Assuming that $\mathcal{U}^n$ is an $(\mathcal{M}, \mathcal{A})$-iterable ultrafilter on $\mathbb{B}^n$, we shall show that $\mathcal{U}^{n+1}$ is an $(\mathcal{M}, \mathcal{A})$-iterable ultrafilter on $\mathbb{B}^{n+1}$. 

Let $X \in \mathcal{U}^{n+1}$. If $X \subseteq Y \in \mathbb{B}^{n+1}$, then $X_\alpha \subseteq Y_\alpha$, for each $\alpha \in \mathrm{Ord}^\mathcal{M}$, and by iterability of $\mathcal{U}^{n}$, $\{\alpha \mid Y_\alpha \in \mathcal{U}\} \in \mathcal{A}$. So $Y \in \mathcal{U}^{n+1}$ by upwards closure of $\mathcal{U}^n$ and $\mathcal{U}$. Similarly, the iterability of $\mathcal{U}^n$ and the finite intersection and maximality properties of $\mathcal{U}^n$ and $\mathcal{U}$ imply that $\mathcal{U}^{n+1}$ has the finite intersection and maximality properties, respectively. To show iterability, suppose that the function $(\xi \mapsto S^g_\xi) : \mathrm{Ord}^\mathcal{M} \rightarrow \mathbb{B}^{n+1}$ is coded in $\mathcal{A}$. Then 
$$\{\xi \mid S^g_\xi \in \mathcal{U}^{n+1}\} = \big\{\xi \mid \{\alpha \mid (S^g_\xi)_\alpha \in \mathcal{U}^n \} \in \mathcal{U} \big\} \in \mathcal{A}$$ 
by iterability of  $\mathcal{U}$. So $\mathcal{U}^{n+1}$ is also $(\mathcal{M}, \mathcal{A})$-iterable. We have proved:

\begin{lemma}
If $\mathcal{U}$ is an $(\mathcal{M, A})$-iterable ultrafilter on $\mathbb{B}$, then $\mathcal{U}^n$ is an $(\mathcal{M, A})$-iterable ultrafilter on $\mathbb{B}^n$, for every $n \in \mathbb{N}$.
\end{lemma}

Since $\mathcal{U}$ is a non-principal ultrafilter, it contains all final segments of $\mathrm{Ord}^\mathcal{M}$. So by induction, we have 
$$\{ \langle \alpha_1, \dots, \alpha_n \rangle \mid \alpha_1 < \dots < \alpha_n \in \mathrm{Ord}^\mathcal{M}\} \in \mathcal{U}^n,$$ 
for every $n \in \mathbb{N}$.

Lastly, we shall extend the definition of completeness and show that each $\mathcal{U}^n$ has this property. An ultrafilter $\mathcal{V}$ on $\mathbb{B}^n$ is $(\mathcal{M, A})$-{\em complete} if for any $m < n$ and any functions $f : (\mathrm{Ord}^\mathcal{M})^m \rightarrow M$ and $g : (\mathrm{Ord}^\mathcal{M})^n \rightarrow M$ coded in $\mathcal{A}$, such that 
$$\{\langle \alpha_1, \dots, \alpha_n \rangle \mid g(\alpha_1, \dots, \alpha_n) \in f(\alpha_1, \dots, \alpha_m)\} \in \mathcal{V}^n,$$ 
we have that $\mathcal{A}$ codes a function $f' : (\mathrm{Ord}^\mathcal{M})^m \rightarrow M$, such that
$$\{\langle \alpha_1, \dots, \alpha_n \rangle \mid g(\alpha_1, \dots, \alpha_n) = f'(\alpha_1, \dots, \alpha_m)\} \in \mathcal{V}^n.$$
\end{constr}

\begin{lemma}\label{iter complete}
If $\mathcal{U}$ is an $(\mathcal{M, A})$-complete and $(\mathcal{M, A})$-iterable ultrafilter, then $\mathcal{U}^n$ is $(\mathcal{M, A})$-complete, for every $n \in \mathbb{N}$.
\end{lemma}
\begin{proof}
Suppose that $m < n$ and that 
\begin{align*}
f &: (\mathrm{Ord}^\mathcal{M})^m \rightarrow \mathcal{M} \\
g &: (\mathrm{Ord}^\mathcal{M})^n \rightarrow \mathcal{M}
\end{align*}
satisfy
$$\{\langle \alpha_1, \dots, \alpha_n \rangle \mid g(\alpha_1, \dots, \alpha_n) \in f(\alpha_1, \dots, \alpha_m)\} \in \mathcal{U}^n.$$
We may assume that $m + 1 = n$. The above is equivalent to
\[
A =_\df \big\{\langle \alpha_1, \dots, \alpha_m \rangle \mid \{\alpha_n \mid g(\alpha_1, \dots, \alpha_n) \in f(\alpha_1, \dots, \alpha_m)\} \in \mathcal{U} \big\} \in \mathcal{U}^m. 
\]
Moreover, by completeness of $\mathcal{U}$, for any $\langle \alpha_1, \dots, \alpha_m \rangle \in A$, there is $y \in f(\alpha_1, \dots, \alpha_m)$ such that
\[
\{\alpha_n \mid g(\alpha_1, \dots, \alpha_n) = y\} \in \mathcal{U}. \tag{$\dagger$}
\]
On the other hand, we have by iterability of $\mathcal{U}$ (utilizing a bijection coded in $\mathcal{A}$ between $\mathrm{Ord}^\mathcal{M}$ and $\mathcal{M}^n$) that
\[
f' =_\df \big\{\langle \langle \alpha_1, \dots, \alpha_m \rangle, y \rangle \mid \{\alpha_n \mid g(\alpha_1, \dots, \alpha_n) = y\} \in \mathcal{U} \big\} \in \mathcal{A}.
\]
Since $\mathcal{U}$ is an ultrafilter, $f'$ is a function whose domain is a superset of $A \in \mathcal{U}$, and by extending it arbitrarily we may assume its domain is $(\mathrm{Ord}^\mathcal{M})^m$.
Now by ($\dagger$), we have 
$$\{\langle \alpha_1, \dots, \alpha_m \rangle \mid \{\alpha_n \mid g(\alpha_1, \dots, \alpha_n) = f'(\alpha_1, \dots, \alpha_m)\} \in \mathcal{U}\} \in \mathcal{U}^m.$$
Thus
$$\{\langle \alpha_1, \dots, \alpha_n \rangle \mid g(\alpha_1, \dots, \alpha_n) = f'(\alpha_1, \dots, \alpha_m)\} \in \mathcal{U}^n,$$
as desired.
\end{proof}

\begin{constr}\label{constr iter power}
Let $\mathcal{L}^0_\mathcal{A}$ be the language obtained from $\mathcal{L}^0$ by adding constant symbols for all elements of $\mathcal{M}$ and adding relation and function symbols for all relations and functions on $\mathcal{M}$ coded in $\mathcal{A}$. $(\mathcal{M}, A)_{A \in \mathcal{A}}$ denotes the canonical expansion of $\mathcal{M}$ to $\mathcal{L}^0_\mathcal{A}$ determined by $(\mathcal{M}, \mathcal{A})$. Assume that $\mathcal{U}$ is a non-principal $(\mathcal{M}, \mathcal{A})$-iterable ultrafilter on $\mathrm{Ord}^\mathcal{M}$ and let $\mathbb{L}$ be a linear order. We construct $\mathrm{Ult}_{\mathcal{U}, \mathbb{L}}(\mathcal{M}, \mathcal{A})$ as follows:

For each $n \in \mathbb{N}$, define 
\begin{align*}
\Gamma_n =_\mathrm{df} \big\{ & \phi(\xs) \in \mathcal{L}^0_\mathcal{A} \mid \{\langle \alpha_1, \dots, \alpha_n \rangle \mid (\mathcal{M}, A)_{A \in \mathcal{A}} \models \phi(\alpha_1, \dots, \alpha_n)\} \in \mathcal{U}^n\big\}.
\end{align*}

Since $\mathcal{U}^n$ is an ultrafilter on $(\mathrm{Ord}^\mathcal{M})^n$, each $\Gamma_n$ is a complete $n$-type over $\mathcal{M}$ in the language $\mathcal{L}^0_\mathcal{A}$. Moreover, each $\Gamma_n$ contains the elementary diagram of $(\mathcal{M}, A)_{A \in \mathcal{A}}$.

For each $l \in \mathbb{L}$, let $c_l$ be a new constant symbol, and let $\mathcal{L}^0_{\mathcal{A}, \mathbb{L}}$ be the language generated by $\mathcal{L}^0_\mathcal{A} \cup \{c_l \mid l \in \mathbb{L}\}$. Define 
\begin{align*}
T_{\mathcal{U}, \mathbb{L}} =_\mathrm{df} \{\phi(c_{l_1}, \dots, c_{l_n}) \in \mathcal{L}^0_{\mathcal{A}, \mathbb{L}} \mid & n \in \mathbb{N} \wedge (l_1 <_\mathbb{L} \dots <_\mathbb{L} l_n \in \mathbb{L}) \wedge \phi(\xs) \in \Gamma_n\}.
\end{align*}

$T_{\mathcal{U}, \mathbb{L}}$ is complete and contains the elementary diagram of $(\mathcal{M}, \mathcal{A})$, because the same holds for each $\Gamma_n$. By Construction \ref{constr iter ultra},
$$T_{\mathcal{U}, \mathbb{L}} \vdash c_{l_1} < c_{l_2} \in \mathrm{Ord} \text{, for any $l_1 <_\mathbb{L} l_2$.}$$ 
Moreover, $T_{\mathcal{U}, \mathbb{L}}$ has definable Skolem functions: For each $\mathcal{L}^0_{\mathcal{A}, \mathbb{L}}$-formula $\exists x . \phi(x)$, we can prove in $T_{\mathcal{U}, \mathbb{L}}$ that the set of witnesses of $\exists x . \phi(x, y)$ of least rank exists, and provided this set is non-empty an element is picked out by a global choice function coded in $\mathcal{A}$. Thus we can define the {\em iterated ultrapower of} $(\mathcal{M, A})$ {\em modulo} $\mathcal{U}$ {\em along} $\mathbb{L}$ as 
\[
\mathrm{Ult}_{\mathcal{U}, \mathbb{L}}(\mathcal{M}, \mathcal{A}) =_\mathrm{df} \text{``the prime model of $T_{\mathcal{U}, \mathbb{L}}$''}.
\]
In particular, every element of $\mathrm{Ult}_{\mathcal{U}, \mathbb{L}}(\mathcal{M}, \mathcal{A})$ is of the form $f(c_{l_1}, \dots, c_{l_n})$, where $l_1 < \dots < l_n \in \mathbb{L}$ and $f \in \mathcal{A}$ (considered as a function symbol of $\mathcal{L}^0_{\mathcal{A}, \mathbb{L}}$). Note that for any $A \in \mathcal{A}$, any function $f$ coded in $\mathcal{A}$ and for any $l_1, \dots, l_n \in \mathbb{L}$, where $n \in \mathbb{N}$, we have
\[
\begin{array}{cl}
&\mathrm{Ult}_{\mathcal{U}, \mathbb{L}}(\mathcal{M}, \mathcal{A}) \models f(c_{l_1}, \dots, c_{l_n}) \in A \\
\Leftrightarrow & \{ \xi \in \mathrm{Ord}^\mathcal{M} \mid (\mathcal{M}, A)_{A \in \mathcal{A}} \models f(\xi) \in A \} \in \mathcal{U}.
\end{array}
\]
A different way of saying the same thing:
\begin{align*}
A^{\mathrm{Ult}_{\mathcal{U}, \mathbb{L}}(\mathcal{M}, \mathcal{A})} = \big\{ & (f(c_{l_1}, \dots, c_{l_n}))^{\mathrm{Ult}_{\mathcal{U}, \mathbb{L}}(\mathcal{M}, \mathcal{A})} \mid \{ \xi \in \mathrm{Ord}^\mathcal{M} \mid (\mathcal{M}, A)_{A \in \mathcal{A}} \models f(\xi) \in A \} \in \mathcal{U} \big\}.
\end{align*}

Since $T_{\mathcal{U}, \mathbb{L}}$ contains the elementary diagram of $(\mathcal{M}, \mathcal{A})$, the latter embeds elementarily in $\mathrm{Ult}_{\mathcal{U}, \mathbb{L}}(\mathcal{M}, \mathcal{A})$. For simplicity of presentation, we assume that this is an elementary extension. Note that if $\mathbb{L}$ is empty, then $\mathrm{Ult}_{\mathcal{U}, \mathbb{L}}(\mathcal{M}, \mathcal{A}) = (\mathcal{M, A})$. If $\mathcal{U}$ is non-principal, then it is easily seen from Construction \ref{constr iter ultra} that for any $l, l' \in \mathbb{L}$ and any $\alpha \in \mathrm{Ord}^\mathcal{M}$,
$$l <_\mathbb{L} l' \Leftrightarrow \alpha <_\mathbb{O} c_l <_\mathbb{O} c_{l'},$$
where $\mathbb{O} = \mathrm{Ord}^{\mathrm{Ult}_{\mathcal{U}, \mathbb{L}}(\mathcal{M}, \mathcal{A})}$. So $\mathbb{L}$ embeds into the linear order of the ordinals in $\mathrm{Ult}_{\mathcal{U}, \mathbb{L}}(\mathcal{M}, \mathcal{A})$, above the ordinals of $\mathcal{M}$.

It will be helpful to think of the ultrapower as a function (actually functor) of $\mathbb{L}$ rather than as a function of $(\mathcal{M, A})$, so we introduce the alternative notation 
\[
\mathcal{G}_{\mathcal{U}, (\mathcal{M, A})}(\mathbb{L}) =_\mathrm{df} \mathrm{Ult}_{\mathcal{U}, \mathbb{L}}(\mathcal{M}, \mathcal{A}).
\]

Suppose that $(\mathcal{M, A})$ is a countable model of $\mathrm{GBC} + \text{``$\mathrm{Ord}$ is weakly compact''}$ and let $\mathcal{U}$ be an iterable non-principal ultrafilter on $\mathbb{B}$. Given an embedding $i : \mathbb{K} \rightarrow \mathbb{L}$, we construct an embedding 
$$\mathcal{G}_{\mathcal{U}, (\mathcal{M, A})}(i) : \mathcal{G}_{\mathcal{U}, (\mathcal{M, A})}(\mathbb{K}) \rightarrow \mathcal{G}_{\mathcal{U}, (\mathcal{M, A})}(\mathbb{L})$$ 
as follows: Note that any $a \in \mathcal{G}_{\mathcal{U}, (\mathcal{M, A})}(\mathbb{K})$ is of the form $f(c_{k_1}, \dots, c_{k_n})$ for some $f \in \mathcal{A}$, $n \in \mathbb{N}$ and $k_1, \dots, k_n \in \mathbb{K}$. Define $\mathcal{G}_{\mathcal{U}, (\mathcal{M, A})}(i)(a) = f(c_{i(k_1)}, \dots, c_{i(k_n)})$.

As shown in Theorem \ref{Gaifman thm}, $\mathcal{G}_{\mathcal{U}, (\mathcal{M, A})}(i)$ is an elementary embedding, and further more, $\mathcal{G}_{\mathcal{U}, (\mathcal{M, A})}$ is a functor from the category of linear orders, with embeddings as morphisms, to the category of models of the $\mathcal{L}^0_\mathcal{A}$-theory of $(\mathcal{M}, A)_{A \in \mathcal{A}}$, with elementary embeddings as morphisms. We call this the {\em Gaifman functor} of $\mathcal{U}, (\mathcal{M, A})$ and denote it by $\mathcal{G}_{\mathcal{U}, (\mathcal{M, A})}$, or just $\mathcal{G}$ for short.
\end{constr}

Gaifman \cite{Gai76} essentially proved the theorem below for models of arithmetic. A substantial chunk of its generalization to models of set theory was proved for specific needs in \cite{Ena04}.

\begin{thm}[Gaifman-style]\label{Gaifman thm}
Suppose that $(\mathcal{M}, \mathcal{A}) \models \mathrm{GBC} + $``$\mathrm{Ord}$ is weakly compact'' is countable and let $\mathcal{U}$ be an $(\mathcal{M}, \mathcal{A})$-generic ultrafilter. Write $\mathcal{G} = \mathcal{G}_{\mathcal{U}, (\mathcal{M, A})}$ for the corresponding Gaifman functor. Let $i : \mathbb{K} \rightarrow \mathbb{L}$ be an embedding of linear orders.
\begin{enumerate}[{\normalfont (a)}]
\item \label{satisfaction} For each $n \in \mathbb{N}$ and each $\phi(x_1, \dots, x_n) \in \mathcal{L}^0_{\mathcal{A}, \mathbb{L}}$:
\begin{align*}
&\mathcal{G}(\mathbb{L}) \models \phi(c_{l_1}, \dots, c_{l_n}) \Leftrightarrow \\
&\big\{ \langle \alpha_1, \dots, \alpha_n \rangle \in (\mathrm{Ord}^\mathcal{M})^n \mid (\mathcal{M}, A)_{A \in \mathcal{A}} \models \phi(\alpha_1, \dots, \alpha_n) \big\} \in \mathcal{U}^n.
\end{align*}
\item \label{elem} $\mathcal{G}(i) : \mathcal{G}(\mathbb{K}) \rightarrow \mathcal{G}(\mathbb{L})$ is an elementary embedding.
\item \label{func} $\mathcal{G}$ is a functor.
\item \label{cons} If $\mathbb{L} \neq \varnothing$, then $\mathrm{SSy}_\mathcal{M}(\mathcal{G}(\mathbb{L})) = (\mathcal{M}, \mathcal{A})$.
\item \label{card} $|\mathcal{G}(\mathbb{L})| = |\mathbb{L}| + \aleph_0$.
\item \label{init} $i$ is initial iff $\mathcal{G}(i)$ is rank-initial. 
\item \label{iso} $i$ is an isomorphism iff $\mathcal{G}(i)$ is an isomorphism.
\item \label{bnd} Let $l_0 \in \mathbb{L}$. $i$ is strictly bounded above by $l_0$ iff $\mathcal{G}(i)\restriction_{\mathrm{Ord}^{\mathcal{G}(\mathbb{K})}}$ is strictly bounded above by $c_{l_0}$.
\item \label{downcof} If $\mathbb{L} \setminus i(\mathbb{K})$ has no least element, then $\{ c_l \mid l \in \mathbb{L} \setminus i(\mathbb{K})\}$ is downwards cofinal in $\mathrm{Ord}^{\mathcal{G}(\mathbb{L})} \setminus \mathrm{Ord}^{\mathcal{G}(i(\mathbb{K}))}$.
\item \label{equal} Let $\mathbb{L}'$ be a linear order and let $j, j' : \mathbb{L} \rightarrow \mathbb{L}'$ be embeddings. $i$ is an equalizer of $j, j' : \mathbb{L} \rightarrow \mathbb{L}'$ iff $\mathcal{G}(i)$ is an equalizer of $\mathcal{G}(j), \mathcal{G}(j') : \mathcal{G}(\mathbb{L}) \rightarrow \mathcal{G}(\mathbb{L}')$.
\item \label{contr} Let $i' : \mathbb{K} \rightarrow \mathbb{L}$ be an embedding. We have $\forall k \in \mathbb{K} . i(k) < i'(k)$ iff $\forall \xi \in \mathrm{Ord}^{\mathcal{G}(\mathbb{K})} \setminus \mathrm{Ord}^\mathcal{M} . \mathcal{G}(i)(\xi) < \mathcal{G}(i')(\xi)$.
\end{enumerate}
\end{thm}

{\em Remark. } (\ref{elem}) and (\ref{init}) imply that $(\mathcal{M, A})$ is a rank-initial elementary substructure of $\mathcal{G}(\mathbb{L})$. It follows from (\ref{equal}) that if $j : \mathbb{L} \rightarrow \mathbb{L}$ is a self-embedding with no fixed point, then the fixed point set of $\mathcal{G}(j)$ is $\mathcal{M}$ (consider the equalizer of $j$ and $\id_\mathbb{L}$). It follows from (\ref{contr}) that $i : \mathbb{L} \rightarrow \mathbb{L}$ is contractive iff $\mathcal{G}(i)$ is contractive on $\mathcal{M} \setminus \mathcal{S}$.

\begin{proof}
(\ref{satisfaction}) This is immediate from Construction \ref{constr iter power}.

(\ref{elem}) We may assume that $\mathbb{K} \subseteq \mathbb{L}$ and that $i$ is the corresponding inclusion function. This has the convenient consequence that $\mathcal{L}^0_{\mathcal{A}, \mathbb{K}} \subseteq \mathcal{L}^0_{\mathcal{A}, \mathbb{L}}$. Let $\phi(\vec{c}) \in \mathcal{L}^0_{\mathcal{A}, \mathbb{K}}$ be a sentence, where $\vec{c}$ is a tuple of constants. By (\ref{satisfaction}), 
\[
\mathcal{G}(\mathbb{K}) \models \phi(\vec{c}) \Leftrightarrow \mathcal{G}(\mathbb{L}) \models \phi(\vec{c}).
\]
Since every element of $\mathcal{G}(\mathbb{K})$ interprets a term, this equivalence establishes $\mathcal{G}(\mathbb{K}) \preceq \mathcal{G}(\mathbb{L})$, as $\mathcal{L}^0_{\mathcal{A}, \mathbb{K}}$-structures.

(\ref{func}) It is clear that $\mathcal{G}(\id_\mathbb{L}) = \id_{\mathcal{G}(\mathbb{L})}$. It only remains to verify that composition is preserved. Let $j : \mathbb{L} \rightarrow \mathbb{L}'$ and $j' : \mathbb{L}' \rightarrow \mathbb{L}''$ be embeddings of linear orders. Let $a$ be an arbitrary element of $\mathcal{G}(\mathbb{L})$. Then $a = f(c_{l_1}, \dots, c_{l_n})$, for some $f \in \mathcal{A}$, $n \in \mathbb{N}$ and $l_1, \dots, l_n \in \mathbb{L}$. $\mathcal{G}(j'\circ j)(a) = f(c_{j'\circ j(l_1)}, \dots, c_{j'\circ j(l_n)}) = f(c_{j'(j(l_1))}, \dots, c_{j'(j(l_n))}) = (\mathcal{G}(j') \circ \mathcal{G}(j))(a)$, as desired.

(\ref{cons}) We start with $\mathcal{A} \subseteq \mathrm{Cod}_\mathcal{M}(\mathcal{G}(\mathbb{L}))$: Let $A \in \mathcal{A}$. Since $(\mathcal{M}, \mathcal{A}) \models \mathrm{GBC}$, the function $f_A : \mathrm{Ord}^\mathcal{M} \rightarrow \mathcal{M}$, defined by $f_A(\xi) = V_\xi \cap A$ for all $\xi \in \mathrm{Ord}^\mathcal{M}$, is coded in $\mathcal{A}$. Since $\mathbb{L} \neq \varnothing$, let $l \in \mathbb{L}$. Now by (\ref{satisfaction}), for each $a \in \mathcal{M}$,
\[
\begin{array}{cl}
&\mathcal{G}(\mathbb{L}) \models a \in f_A(c_l) \\
\Leftrightarrow & \{ \alpha \in \mathrm{Ord^\mathcal{M}} \mid (\mathcal{M}, A)_{A \in \mathcal{A}} \models a \in f_A(\alpha) \} \in \mathcal{U} \\
\Leftrightarrow & a \in A,
\end{array}
\]
so $f_A(c_l)$ codes $A$.

We proceed with $\mathrm{Cod}_\mathcal{M}(\mathcal{G}(\mathbb{L})) \subseteq \mathcal{A}$: Let $b \in \mathcal{G}(\mathbb{L})$. Then $\mathcal{G}(\mathbb{L}) \models b = f(c_{l_1}, \dots, c_{l_n})$, for some $n \in \mathbb{N}$ and $l_1, \dots, l_n \in \mathbb{L}$. We need to show that 
\[
\{ x \in \mathcal{M} \mid \mathcal{G}(\mathbb{L}) \models x \in f(c_{l_1}, \dots, c_{l_n}) \} \in \mathcal{A}.
\]
By (\ref{satisfaction}), this amounts to showing that
\begin{align*}
\big\{ & x \in \mathcal{M} \mid \{ \langle \alpha_1, \dots, \alpha_n \rangle \in (\mathrm{Ord}^\mathcal{M})^n \mid (\mathcal{M}, A)_{A \in \mathcal{A}} \models x \in f(\alpha_1, \dots, \alpha_n) \} \in \mathcal{U}^n \big\} \in \mathcal{A}.
\end{align*}
Letting $w : \mathrm{Ord}^\mathcal{M} \rightarrow \mathcal{M}$ be a well-ordering of $\mathcal{M}$ coded in $\mathcal{A}$, the above is equivalent to
\begin{align*}
\big\{ \xi \in \mathrm{Ord}^\mathcal{M} \mid \{ \langle \alpha_1, \dots, \alpha_n \rangle \in (\mathrm{Ord}^\mathcal{M})^n \mid \phantom{\{}(\mathcal{M}, A)_{A \in \mathcal{A}} \models w(\xi) \in f(\alpha_1, \dots, \alpha_n) \} \in \mathcal{U}^n \big\} \in \mathcal{A}.
\end{align*}
This last statement holds since $\mathcal{U}^n$ is $(\mathcal{M}, \mathcal{A})$-iterable on $\mathbb{B}^n$.

(\ref{card}) Since $(\mathcal{M}, \mathcal{A})$ is countable, $|\mathbb{L}| + \aleph_0 = |\mathcal{L}^0_{\mathcal{A}, \mathbb{L}}|$. So since $\mathcal{G}(\mathbb{L})$ is a prime model in that language, we have $|\mathcal{G}(\mathbb{L})| = |\mathbb{L}| + \aleph_0$.

(\ref{init}) We may assume that $\mathbb{L}$ extends $\mathbb{K}$ and that $\mathcal{G}(\mathbb{L})$ extends $\mathcal{G}(\mathbb{K})$. By Proposition \ref{elem initial is rank-initial}, it suffices to show that $i$ is initial. Let $a \in \mathcal{G}(\mathbb{K})$ and $b \in \mathcal{G}(\mathbb{L})$, such that $\mathcal{G}(\mathbb{L}) \models b \in a$. We need to show that $b \in \mathcal{G}(\mathbb{K})$. Note that $\mathcal{G}(\mathbb{K}) \models a = f(c_{l_1}, \dots, c_{l_m})$ and $\mathcal{G}(\mathbb{L}) \models b = g(c_{l_1}, \dots, c_{l_n})$, for some $m \leq n \in \mathbb{N}$, $l_1, \dots, l_m \in \mathbb{K}$, $ l_{m+1}, \dots, l_n \in \mathbb{L}$ and $f, g \in \mathcal{A}$. By (\ref{satisfaction}), we have that 
$$\big\{\alpha_1, \dots \alpha_n \mid (\mathcal{M, A}) \models g(\alpha_1, \dots, \alpha_n) \in f(\alpha_1, \dots, \alpha_m)\big\} \in \mathcal{U}^n.$$
So by $(\mathcal{M, A})$-completeness of $\mathcal{U}$ and Lemma \ref{iter complete}, there is $f' : (\mathrm{Ord}^\mathcal{M})^m \rightarrow \mathcal{M}$ in $\mathcal{A}$, such that
\[
\big\{\alpha_1, \dots \alpha_n \mid (\mathcal{M, A}) \models g(\alpha_1, \dots, \alpha_n) = f'(\alpha_1, \dots, \alpha_m)\big\} \in \mathcal{U}^n,
\]
whence $\mathcal{G}(\mathbb{L}) \models b = f'(c_{l_1}, \dots, c_{l_m})$. But $f'(c_{l_1}, \dots, c_{l_m}) \in \mathcal{G}(\mathbb{K})$. So $i$ is initial.

(\ref{iso}) ($\Leftarrow$) follows from that the orderings embed into the respective sets of ordinals of the models, and that any isomorphism of the models preserves the order of their ordinals. ($\Rightarrow$) follows from that functors preserve isomorphisms. 

(\ref{bnd}) ($\Leftarrow$) is obvious. For ($\Rightarrow$), we may assume that $\mathbb{K}$ is a linear suborder of $\mathbb{L}$ that is strictly bounded above by $l_0 \in \mathbb{L}$. Note that $\mathcal{G}(\mathbb{K}) \prec \mathcal{G}(\mathbb{L}_{< l_0}) \prec \mathcal{G}(\mathbb{L})$. So every ordinal of $\mathcal{G}(\mathbb{K})$ is an ordinal of $\mathcal{G}(\mathbb{L}_{< l_0})$, and by (\ref{init}), every ordinal of $\mathcal{G}(\mathbb{L}_{< l_0})$ is an ordinal of $\mathcal{G}(\mathbb{L})$ below $c_{l_0}$.

(\ref{downcof}) We may assume that $\mathbb{K} \subseteq \mathbb{L}$. Suppose that $\mathbb{L} \setminus \mathbb{K}$ has no least element. Let $\alpha \in \mathrm{Ord}^{\mathcal{G}(\mathbb{L})} \setminus \mathrm{Ord}^{\mathcal{G}(\mathbb{K})}$. Then $\mathcal{G}(\mathbb{L}) \models \alpha = f(c_{l_1}, \dots, c_{l_n})$, for some $n \in \mathbb{N}$ and $l_1 < \dots < l_n \in \mathbb{L}$. Let $1 \leq n^\circ \leq n$ be the least natural number such that there is $l \in \mathbb{L}\setminus \mathbb{K}$ with $l < l_{n^\circ}$. Let $l^* \in \mathbb{L}\setminus \mathbb{K}$ witness this for $n^\circ$. To show that $\mathcal{G}(\mathbb{L}) \models l^* < \alpha$, it suffices to show that
$$\{\langle \xi_1, \dots, \xi_{n^\circ - 1}, \xi^*, \xi_{n^\circ}, \dots, \xi_n \rangle \mid \xi^* < f(\xi_1, \dots, \xi_n) \} \in \mathcal{U}^{n+1}.$$
Suppose not. Then
$$\{\langle \xi_1, \dots, \xi_{n^\circ - 1}, \xi^*, \xi_{n^\circ}, \dots, \xi_n \rangle \mid \xi^* \geq f(\xi_1, \dots, \xi_n) \} \in \mathcal{U}^{n+1 },$$
so by completeness
\begin{align*}
\big\{ & \langle \xi_1, \dots, \xi_{n^\circ - 1} \rangle \mid  \exists \xi . \{ \langle \xi_{n^\circ}, \dots, \xi_n \rangle \mid \xi = f(\xi_1, \dots, \xi_n)\} \in \mathcal{U}^{n - n^\circ + 1} \big\} \in \mathcal{U}^{n^\circ - 1}.
\end{align*}
Hence, by iterability, we can code a function $f' : (\mathrm{Ord}^\mathcal{M})^{n^\circ -1} \rightarrow \mathrm{Ord}^\mathcal{M}$ by the class
$$\{\langle \langle \xi_1, \dots, \xi_{n^\circ - 1} \rangle, \xi \rangle \mid \{ \langle \xi_{n^\circ}, \dots, \xi_n \rangle \mid \xi = f(\xi_1, \dots, \xi_n)\} \in \mathcal{U}^{n - n^\circ + 1} \} \in \mathcal{A},$$
and $\mathcal{G}(\mathbb{L}) \models f'(l_{c_1}, \dots, l_{c_{n^\circ - 1}}) = \alpha$. But this means that $\alpha \in \mathcal{G}(\mathbb{K})$, contradicting assumption. 

(\ref{equal}) ($\Leftarrow$) is obvious. For ($\Rightarrow$), assume that $i : \mathbb{K} \rightarrow \mathbb{L}$ is an equalizer of $j, j' : \mathbb{L} \rightarrow \mathbb{L}'$, i.e. we may assume that $\mathbb{K}$ is the linear suborder of $\mathbb{L}$ on $\{l \in \mathbb{L} \mid j(l) = j'(l)\}$. It suffices to show that for all elements $x$ of $\mathcal{G}(\mathbb{L})$, we have $\mathcal{G}(j)(x) = \mathcal{G}(j')(x) \leftrightarrow x \in \mathcal{G}(\mathbb{K})$. ($\leftarrow$) is obvious. For ($\rightarrow$), suppose that $a \in \mathcal{G}(\mathbb{L}) \setminus \mathcal{G}(\mathbb{K})$. Let $n^*$ be the least natural number such that $\mathcal{G}(\mathbb{L}) \models a = f(c_{l_1}, \dots, c_{l_{n^*}})$, for some $f \in \mathcal{A}$ and $l_1 < \dots < l_{n^*} \in \mathbb{L}$. Suppose towards a contradiction that $\mathcal{G}(\mathbb{L}) \models f(c_{j(l_1)}, \dots, c_{j(l_{n^*})}) = f(c_{j'(l_1)}, \dots, c_{j'(l_{n^*})})$. Since $\mathcal{U}$ is $(\mathcal{M, A})$-canonically Ramsey, since $f$ is coded in $\mathcal{A}$ and since there is a bijection between the universe and the ordinals coded in $\mathcal{A}$, there is $H \in \mathcal{U}$ and $S \subseteq \{1, \dots, n^*\}$, such that for any $\alpha_1 < \dots < \alpha_{n^*}$ and $\beta_1 < \dots < \beta_{n^*}$ in $H$, 
$$f(\alpha_1, \dots, \alpha_{n^*}) = f(\beta_1, \dots, \beta_{n^*}) \leftrightarrow \forall m \in S . \alpha_m = \beta_m.$$
Since $a \not\in \mathcal{G}(\mathbb{K})$, there is $1 \leq n^\circ \leq n^*$, such that $l_{n^\circ} \in \mathbb{L} \setminus \mathbb{K}$ and consequently $j(l_{n^\circ}) \neq j'(l_{n^\circ})$. So since $f(c_{j(l_1)}, \dots, c_{j(l_{n^*})}) = f(c_{j'(l_1)}, \dots, c_{j'(l_{n^*})})$, we have $n^\circ \not\in S$.
Note that $f' : (\mathrm{Ord}^\mathcal{M})^{n^*-1} \rightarrow \mathcal{M} $, defined by 
\begin{align*}
& f'(\alpha_1, \dots, \alpha_{n^\circ - 1}, \alpha_{n^\circ +1}, \dots, \alpha_{n^*}) = \\
& f(\alpha_1, \dots, \alpha_{n^\circ-1}, \alpha_{n^\circ-1}+1, \alpha_{n^\circ + 1}, \dots, \alpha_{n^*}),
\end{align*}
is coded in $\mathcal{A}$; and by $n^\circ \not\in S$, 
$$f(\alpha_1, \dots, \alpha_{n^\circ-1}, \alpha_{n^\circ-1}+1, \alpha_{n^\circ + 1}, \dots, \alpha_{n^*}) = f(\alpha_1, \dots, \alpha_{n^*})$$ 
for all $\alpha_1< \dots< \alpha_{n^*} \in H$. Since $H \in \mathcal{U}$, it follows that 
$$\mathcal{G}(\mathbb{L}) \models a = f'(c_{l_1}, \dots, c_{l_{n^\circ -1}}, c_{l_{n^\circ +1}}, \dots, c_{l_n}),$$ 
contradicting minimality of $n^*$. 

(\ref{contr}) ($\Leftarrow$) is obvious. For ($\Rightarrow$), let $\alpha$ be an arbitrary ordinal in $\mathcal{G}(\mathbb{K})$ not in $\mathcal{M}$. Let $k^*$ be the least element of $\mathbb{K}$, such that $\mathcal{G}(\mathbb{K}) \models \alpha =  f(c_{k_1}, \dots, c_{k_n})$, where $f \in \mathcal{A}$, $n \in \mathbb{N}$ and $k_1 < \dots < k_n = k^* \in \mathbb{K}$. Note that $i'(\alpha)$ is in $\mathcal{G}(\mathbb{L}_{\leq i'(k^*)})$ and that $i(\alpha)$ is in $\mathcal{G}(\mathbb{L}_{\leq i(k^*)})$. On the other hand, since $i(k^*) < i'(k^*)$, we have by minimality of $k^*$ that $i'(\alpha)$ is not in $\mathcal{G}(\mathbb{L}_{\leq i(k^*)})$. So by (\ref{elem}) and (\ref{init}), $\mathcal{G}(\mathbb{L}) \models i(\alpha) < i'(\alpha)$.
\end{proof}

As seen in the proofs of the Corollaries below, this theorem is quite powerful when applied to the set of rational numbers $\mathbb{Q}$, with the usual ordering $<_\mathbb{Q}$. For any structure $\mathcal{K}$, and $S \subseteq \mathcal{K}$, we define $\mathrm{End}_S(\mathcal{K})$ as the monoid of endomorphisms of $\mathcal{K}$ that fix $S$ pointwise, and we define $\mathrm{Aut}_S(\mathcal{K})$ as the group of automorphisms of $\mathcal{K}$ that fix $S$ pointwise.

\begin{cor}\label{end extend model of weakly compact to model with autos}
	If $\mathcal{M} \models \mathrm{ZFC}$ expands to a countable model $(\mathcal{M}, \mathcal{A})$ of $\mathrm{GBC} + $ ``$\mathrm{Ord}$ is weakly compact'', then there is $\mathcal{M} \prec^\textnormal{rank-cut} \mathcal{N}$, such that $\mathrm{SSy}_\mathcal{M}(\mathcal{N}) = (\mathcal{M}, \mathcal{A})$, and such that for any countable linear order $\mathbb{L}$, there is an embedding of $\mathrm{End}(\mathbb{L})$ into $\mathrm{End}_\mathcal{M}(\mathcal{N})$. Moreover, this embedding sends every automorphism of $\mathbb{L}$ to an automorphism of $\mathcal{N}$, and sends every contractive self-embedding of $\mathbb{L}$ to a self-embedding of $\mathcal{N}$ that is contractive on $\mathrm{Ord}^\mathcal{N} \setminus \mathcal{M}$ and whose fixed-point set is $\mathcal{M}$.
\end{cor}
\begin{proof}
	Since $(\mathcal{M}, \mathcal{A})$ is countable, Lemma \ref{generic filter existence} tells us that there is an $(\mathcal{M}, \mathcal{A})$-generic ultrafilter $\mathcal{U}$. Let $\mathcal{N} = \mathcal{G}_{\mathcal{U}, (\mathcal{M}, \mathcal{A})}(\mathbb{Q})$. By Theorem \ref{Gaifman thm} (\ref{elem}), (\ref{cons}) and (\ref{init}), $\mathcal{M} \prec^\textnormal{rank-cut} \mathcal{N}$ and $\mathrm{SSy}_\mathcal{M}(\mathcal{N}) = (\mathcal{M}, \mathcal{A})$. By Theorem \ref{Gaifman thm} (\ref{func}) and (\ref{equal}), there is an embedding of $\mathrm{End}(\mathbb{Q})$ into $\mathrm{End}_\mathcal{M}(\mathcal{N})$. Moreover, it is well-known that for any countable linear order $\mathbb{L}$, there is an embedding of $\mathrm{End}(\mathbb{L})$ into $\mathrm{End}(\mathbb{Q})$. Composing these two embeddings gives the result. The last sentence in the statement follows from Theorem \ref{Gaifman thm} (\ref{iso}), (\ref{equal}) and (\ref{contr}).
\end{proof}

\begin{lemma}\label{contractive proper selfembedding of Q}
For any $q \in \mathbb{Q}$, there is an initial topless contractive self-embedding of the usual linear order on  $\mathbb{Q}$ that is strictly bounded by $q$.
\end{lemma}
\begin{proof}
It suffices to show that there is an initial topless contractive self-embedding of $\mathbb{Q}$, because by toplessness that would be bounded by some $q' \in \mathbb{Q}$ and we can compose it with the self-embedding $(x \mapsto x - |q' - q|)$ to obtain an initial topless contractive self-embedding bounded by $q$. Thus, we proceed to show that the usual linear order on $\mathbb{Q}$ can be expanded to a model of the following theory $T$, in the language of a binary relation $<$ and a unary function $f$:
\begin{align*}
&\text{``$<$ is a dense linear order without endpoints''}; \\
&\forall x, y . ( x < y \leftrightarrow f(x) < f(y)); \\
&\forall x, y . (x < f(y) \rightarrow \exists z . f(z) = x); \\
&\exists y . \forall x . f(x) < y; \\
&\forall y . \big((\forall x . f(x) < y) \rightarrow \exists y' . (y' < y \wedge \forall x . f(x) < y') \big); \\
&\forall x . f(x) < x.
\end{align*}
Let $\mathcal{R}$ be the expansion of the order of the punctured reals $\mathbb{R} \setminus \{0\}$ inherited from the usual order of $\mathbb{R}$, interpreting $f$ by the function $f^\mathcal{R} : \mathbb{R} \setminus \{0\} \rightarrow \mathbb{R} \setminus \{0\}$, defined by $f^\mathcal{R}(x) = -2^{-x}$, for all $x \in \mathbb{R} \setminus \{0\}$. Note that $\mathcal{R} \models T$. Now by the Downward L\"owenheim-Skolem Theorem, there is a countable model $\mathcal{Q}$ of $T$. Since every countable dense linear order without endpoints is isomorphic to $(\mathbb{Q}, <_\mathbb{Q})$, it follows that $f^\mathcal{Q}$ induces an initial topless contractive self-embedding of $\mathbb{Q}$.
\end{proof}

\begin{cor}\label{end extend model of weakly compact to model with endos}
	Suppose that $\mathcal{M} \models \mathrm{ZFC}$ expands to a countable model $(\mathcal{M}, \mathcal{A})$ of $\mathrm{GBC} + $``$\mathrm{Ord}$ is weakly compact''. Then there is a model $\mathcal{M} \prec^\textnormal{rank-cut} \mathcal{N}$, with $\mathrm{SSy}_\mathcal{M}(\mathcal{N}) = (\mathcal{M}, \mathcal{A})$, such that for any $\nu \in \mathrm{Ord}^\mathcal{N} \setminus \mathcal{M}$, there is a rank-initial topless elementary self-embedding $j$ of $\mathcal{N}$, which is contractive on $\mathrm{Ord}^\mathcal{N} \setminus \mathcal{M}$, bounded by $\nu$, and satisfies $\mathcal{M} = \mathrm{Fix}(j)$.
\end{cor}
\begin{proof}
	Let $\mathcal{U}$ be an $(\mathcal{M}, \mathcal{A})$-generic ultrafilter, and let $\mathcal{N}$ be the model $\mathcal{G}_{\mathcal{U}, (\mathcal{M}, \mathcal{A})}(\mathbb{Q})$. As in Corollary \ref{end extend model of weakly compact to model with autos}, $\mathcal{M} \prec^\textnormal{rank-cut} \mathcal{N}$ and $\mathrm{SSy}_\mathcal{M}(\mathcal{N}) = (\mathcal{M}, \mathcal{A})$. Let $\nu \in \mathrm{Ord}^\mathcal{N} \setminus \mathcal{M}$. By Theorem \ref{Gaifman thm} (\ref{downcof}), there is $q \in \mathbb{Q}$, such that $\mathcal{N} \models c_q < \nu$. Using Lemma \ref{contractive proper selfembedding of Q}, let $\hat j$ be an initial topless contractive self-embedding of $\mathbb{Q}$ that is strictly bounded by $q$. Let $j = \mathcal{G}(\hat j)$. The result now follows from Theorem \ref{Gaifman thm}: By (\ref{elem}), $j$ is an elementary embedding; by (\ref{init}), $j$ is rank-initial; by (\ref{bnd}), $j$ is bounded by $\nu$; by toplessness of $\hat j$ and by (\ref{downcof}) applied to $\mathcal{N} \setminus \mathcal{G}(\hat j(\mathbb{Q}))$, $j$ is topless; by (\ref{contr}), $j$ is contractive on $\mathcal{N} \setminus \mathcal{M}$; and by (\ref{equal}) $\mathcal{M} = \mathrm{Fix}(j)$.
\end{proof}

\section{Geometric results}\label{Characterizations}

\begin{thm}[Kirby-Paris-style]\label{Kirby Paris thm}
	Let $\mathcal{M} \models \mathrm{KP}^\mathcal{P} + \textnormal{Choice}$ be countable and let $\mathcal{S} \leq^\textnormal{rank-cut} \mathcal{M}$. The following are equivalent:
	\begin{enumerate}[{\normalfont (a)}]
		\item $\mathcal{S}$ is a strong rank-cut in $\mathcal{M}$ and $\omega^\mathcal{M} \in \mathcal{S}$.
		\item $\mathrm{SSy}_\mathcal{S}(\mathcal{M}) \models \mathrm{GBC} + \text{``$\mathrm{Ord}$ is weakly compact''}$.
	\end{enumerate}
\end{thm}
\begin{proof}
	The two directions are proved as Lemmata \ref{Kirby Paris forward} and \ref{Kirby Paris backward} below.
\end{proof}

\begin{lemma}\label{Kirby Paris lemma}
	Let $\mathcal{M} \models \mathrm{KP}^\mathcal{P} + \textnormal{Choice}$, let $\mathcal{S}$ be a strongly topless rank-initial substructure of $\mathcal{M}$ and let us write $\mathrm{SSy}_\mathcal{S}(\mathcal{M})$ as $(\mathcal{S}, \mathcal{A})$. For any $\phi(\vec{x}, \vec{Y}) \in \mathcal{L}^1$ and for any $\vec{A} \in \mathcal{A}$, there is a formula $\theta_\phi(\vec{x}, \vec{y}) \in \Delta_0^\mathcal{P} \subseteq \mathcal{L}^0$ and parameters $\vec{p} \in \mathcal{M}$, such that for all $\vec{s} \in \mathcal{S}$,
	\[
	(\mathcal{S}, \mathcal{A}) \models \phi(\vec{s}, \vec{A}) \Leftrightarrow \mathcal{M} \models \theta_\phi(\vec{s}, \vec{p}).
	\]
\end{lemma}
\begin{proof}
	We construct $\theta_\phi$ recursively on the structure of $\phi$. Let $\vec{A} \in \mathcal{A}$ be arbitrary and let $\vec{a}$ be a tuple of codes in $\mathcal{M}$ for $\vec{A}$. In the base cases, given a coordinate $A$ in $\vec{A}$ and its code $a$ in $\vec{a}$, put:
	\begin{align*}
	\theta_{x=y} &\equiv x = y, \\
	\theta_{x \in y} &\equiv x \in y, \\
	\theta_{x \in A} &\equiv x \in a. 
	\end{align*}
	It is clear that the result holds in the first two cases, and also in the third case since $a$ codes $A$. 
	
	Assume inductively that the result holds for $\phi(\vec{x}, \vec{Y}), \psi(\vec{x}, \vec{Y}) \in \mathcal{L}^1$ and $\theta_\phi(\vec{x}, \vec{y}), \theta_\psi(\vec{x}, \vec{y}) \in \Delta_0^\mathcal{P}$, and put:
	\begin{align*}
	\theta_{\neg \phi} &\equiv \neg \theta_\phi, \\
	\theta_{\phi \vee \psi} &\equiv \theta_\phi \vee \theta_\psi, \\
	\theta_{\exists x . \phi} &\equiv \exists x \in b . \theta_\phi, 
	\end{align*}
	where $b \in \mathcal{M}$ is next to be constructed. But before doing so, note that the result holds for the first two cases, simply because the connectives commute with $\models$.
	
	Let us write $\vec{x}$ as $x, \vec{x}'$, and let $k = \mathrm{arity}(\vec{x}')$. Since $\mathcal{S}$ is bounded in $\mathcal{M} \models \textnormal{Infinity}$, there is a limit $\mu \in \mathrm{Ord}^\mathcal{M} \setminus \mathcal{S}$. Therefore, by letting $d = V_\mu^\mathcal{M}$, we obtain $d \in \mathcal{M}$, $\mathcal{S} \subseteq d_\mathcal{M}$, and $\mathcal{M} \models d^k \subseteq d$. Working in $\mathcal{M}$, by Choice and $\Delta_0^\mathcal{P}$-Separation, there is a function $f : d \rightarrow d$, such that for all $\vec{t} \in d^k$,
	\[
	\exists x \in d . \theta_\phi(x, \vec{t}, \vec{a}) \Leftrightarrow f(\vec{t}) \in \{ u \in d \mid \theta_\phi(u, \vec{t}, \vec{a}) \} \Leftrightarrow \theta_\phi(f(\vec{t}), \vec{t}, \vec{a}). \tag{$\dagger$}
	\]
	By strong toplessness and rank-initiality, there is $\beta \in \mathrm{Ord}^\mathcal{M} \setminus \mathcal{S}$ such that for all $s \in \mathcal{S}$,
	\[
	f(s) \in \mathcal{S} \Leftrightarrow \rnk(f(s)) < \beta.  \tag{$\ddagger$}
	\]
	Put $b = V_\beta^\mathcal{M}$. Note that by toplessness, $\mathcal{S}$ is closed under ordered pair. Putting ($\dagger$) and ($\ddagger$) together, we have that, for all $\vec{s} \in \mathcal{S}$,
	\[
	\mathcal{M} \models \exists x \in b . \theta_\phi(x, \vec{s}, \vec{a}) \Leftrightarrow \exists s_0 \in \mathcal{S} . \mathcal{M} \models \theta_\phi(s_0, \vec{s}, \vec{a}),
	\]
	and by induction hypothesis,
	\begin{align*}
	\exists s_0 \in \mathcal{S} . \mathcal{M} \models \theta_\phi(s_0, \vec{s}, \vec{a}) & \Leftrightarrow \exists s_0 \in \mathcal{S} . (\mathcal{S}, \mathcal{A}) \models \phi(s_0, \vec{s}, \vec{A}) \\
	& \Leftrightarrow  (\mathcal{S}, \mathcal{A}) \models \exists x . \phi(x, \vec{s}, \vec{A}).
	\end{align*}
	Putting these equivalences together yields the desired result for $\exists x . \phi$. The parameters $\vec{p}$ appearing in $\theta_{\exists x . \phi}$ are $b$ together with $\vec{a}$.
\end{proof}

\begin{lemma}\label{Kirby Paris forward}
	Let $\mathcal{M} \models \mathrm{KP}^\mathcal{P} + \textnormal{Choice}$. If $\mathcal{S}$ is strong rank-cut in $\mathcal{M}$ and $\omega^\mathcal{M} \in \mathcal{S}$, then $\mathrm{SSy}_\mathcal{S}(\mathcal{M}) \models \mathrm{GBC} + \text{``$\mathrm{Ord}$ is weakly compact''}$.
\end{lemma}
\begin{proof}
	Let us write $\mathrm{SSy}_\mathcal{S}(\mathcal{M})$ as $(\mathcal{S}, \mathcal{A})$. By toplessness and rank-initiality, there is $d \in \mathcal{M}$ such that $d_\mathcal{M} \supset \mathcal{S}$. 
	
	$(\mathcal{S}, \mathcal{A}) \models \textnormal{Class Extensionality, Pair, Union, Powerset, Infinity}$ are inherited from $\mathcal{M}$, because $\omega^\mathcal{M} \in \mathcal{S}$ and for any $\alpha \in \mathrm{Ord}^\mathcal{S}$, $\alpha +^\mathcal{M} 2 \in \mathcal{S}$ by toplessness, and $\mathcal{M}_{\alpha +^\mathcal{M} 2} \subseteq \mathcal{S}$ by rank-initiality.
	
	$(\mathcal{S}, \mathcal{A}) \models \textnormal{Class Foundation}$: Let $A \in \mathcal{A}$ such that $(\mathcal{S}, \mathcal{A}) \models A \neq \varnothing$. Let $a$ be a code for $A$ in $\mathcal{M}$. By $\Pi_1^\mathcal{P}$-Foundation, there is an $\in^\mathcal{M}$-minimal element $m \in^\mathcal{M} a$. Since $A$ is non-empty, we have by rank-initiality that $m \in \mathcal{S}$. If there were $s \in \mathcal{S}$ such that $(\mathcal{S}, \mathcal{A}) \models s \in m \cap A$, then we would have $\mathcal{M} \models s \in m \cap a$, contradicting $\in^\mathcal{M}$-minimality of $m$. Hence, $m$ is an $\in$-minimal element of $A$ in $(\mathcal{S}, \mathcal{A})$.
		
	$(\mathcal{S}, \mathcal{A}) \models \textnormal{Global Choice}$: By Choice in $\mathcal{M}$, there is a choice function $f$ on $(d \setminus \{\varnothing\})^\mathcal{M}$. Note that $f$ codes a global choice function $F \in \mathcal{A}$ on $(V \setminus \{\varnothing\})^{(\mathcal{M}, \mathcal{A})} \in \mathcal{A}$.
	
	$(\mathcal{S}, \mathcal{A}) \models \textnormal{Class Comprehension}$: Let $\phi(x, \vec{Y}) \in \mathcal{L}^1$, in which all variables of sort $\mathsf{Class}$ are free, and let $\vec{A} \in \mathcal{A}$. By Lemma \ref{Kirby Paris lemma}, there are $\theta_\phi \in \Delta_0^\mathcal{P}$ and $\vec{a} \in \mathcal{M}$, such that for all $s \in \mathcal{S}$,
	\[
	(\mathcal{S}, \mathcal{A}) \models \phi(s, \vec{A}) \Leftrightarrow \mathcal{M} \models \theta_\phi(s, \vec{a}).
	\]
	Working in $\mathcal{M}$, let $c = \{ t \in d \mid \theta_\phi(t, \vec{a}) \}$. Let $C \in \mathcal{A}$ be the class coded by $c$. It follows that 
	\[
	\mathcal{S} \models \forall x . (x \in c \leftrightarrow \phi(x, \vec{A})).
	\]
	
	$(\mathcal{S}, \mathcal{A}) \models \textnormal{Extended Separation}$: Simply observe that if $s \in \mathcal{S}$, $\vec{A} \in \mathcal{A}$ and $\phi(x, \vec{Y}) \in \mathcal{L}^1$, in which all variables of sort $\mathsf{Class}$ are free, then by Class Comprehension, the class $\{ x \mid \phi(x, \vec{A}) \}^{(\mathcal{S}, \mathcal{A})}$ exists in $\mathcal{A}$ and is coded in $\mathcal{M}$ by $c$, say. So by rank-initiality, $c \cap s \in \mathcal{S}$, and $c \cap s$ clearly witnesses the considered instance of $\textnormal{Extended Separation}$.
	
	$(\mathcal{S}, \mathcal{A}) \models \textnormal{Class Replacement}$: Let $F \in \mathcal{A}$ be a class function such that $\dom(F) \in \mathcal{S}$, and let $f$ be a code in $\mathcal{M}$ for $F$. In $\mathcal{M}$, using $\dom(F)$ and $f$ as parameters, we can construct a function $f'$ such that
	\begin{align*}
	\dom(f'_\mathcal{M}) = \phantom{.}& d \supseteq \mathcal{S}, \\
	\mathcal{M} \models \phantom{.}& \forall x \in \dom(F) . \big( (x \in \dom(F) \rightarrow f'(x) = f(x)) \wedge \\
	& (x \not\in \dom(F) \rightarrow f'(x) = 0).
	\end{align*}
	Note that $f'_\mathcal{M}\restriction_{\mathcal{S}} \subseteq \mathcal{S}$. Suppose that 
	$$(\mathcal{S}, \mathcal{A}) \models \forall \xi \in \mathrm{Ord} . \exists x \in \dom(F) . \rnk(F(x)) > \xi.$$ 
	Then we have for all $\xi \in \mathrm{Ord}^\mathcal{S}$ that 
	$$\mathcal{M} \models \exists x \in \dom(F) . \rnk(d) > \rnk(f'(x)) > \xi.$$ 
	So by Overspill, there is $\mu \in \mathrm{Ord}^\mathcal{M} \setminus \mathcal{S}$ such that 
	$$\mathcal{M} \models \exists x \in \dom(F) . \rnk(d) > \rnk(f'(x)) > \mu.$$
	But this contradicts that $f'_\mathcal{M}\restriction_{\mathcal{S}} \subseteq \mathcal{S}$. Therefore, 
	$$(\mathcal{S}, \mathcal{A}) \models \exists \xi \in \mathrm{Ord} . \forall x \in \dom(F) . \rnk(F(x)) < \xi.$$
	Now it follows by $\textnormal{Extended Separation}$ that $\mathrm{image}^{(\mathcal{S}, \mathcal{A})}(F) \in \mathcal{S}$.
	
	$(\mathcal{S}, \mathcal{A}) \models$ ``$\mathrm{Ord}$ is weakly compact'': Let $\mathcal{T}$ be a binary tree of height $\mathrm{Ord}$ in $\mathcal{S}$, coded in $\mathcal{M}$ by $\tau \in \mathcal{M} \setminus \mathcal{S}$. Note that for all $\zeta \in \mathrm{Ord}^\mathcal{S}$, $\tau_\zeta$ is a binary tree of height $\zeta$. So by $\Delta_0$-Overspill, there is $\mu \in \mathrm{Ord}^\mathcal{M} \setminus \mathcal{S}$, such that $\tau_\mu$ is a binary tree of height $\mu$. Let $f \in^\mathcal{M} \tau_\mu$ such that $\dom^\mathcal{M}(f) \in \mathrm{Ord}^\mathcal{M} \setminus \mathcal{S}$. Let $F \in \mathcal{A}$ be the class coded by $f$. Since $\mathcal{S}$ is rank-initial in $\mathcal{M}$, we have for each $\zeta \in \mathrm{Ord}^\mathcal{S}$, that $F_{(\mathcal{S}, \mathcal{A})}(\zeta) = f_\mathcal{M}\restriction_\zeta$. It follows that $F$ is a branch in $\mathcal{T}$.  
\end{proof}

\begin{lemma}\label{embed in Gaifman model}
	Let $\mathcal{M} \models \mathrm{KP}^\mathcal{P}  + \Sigma_1^\mathcal{P}\textnormal{-Separation}$ be countable and non-standard, and let $\mathcal{S}$ be a $\Sigma_1^\mathcal{P}$-elementary rank-cut of $\mathcal{M}$, such that 
	\[
	\mathrm{SSy}_\mathcal{S}(\mathcal{M}) \models \mathrm{GBC} + \textnormal{``$\mathrm{Ord}$ is weakly compact''}.
	\]
	Then there is a structure $\mathcal{N} \succ \mathcal{S}$, and a self-embedding $i : \mathcal{N} \rightarrow \mathcal{N}$, such that $\mathrm{Fix}(i) = \mathcal{S}$, $i$ is contractive on $\mathcal{N} \setminus \mathcal{S}$, and
	\[
	\mathcal{S} <^{\textnormal{rank-cut}} i(\mathcal{M}) <^{\textnormal{rank-cut}} i(\mathcal{N})  <^{\textnormal{rank-cut}} \mathcal{M} <^{\textnormal{rank-cut}} \mathcal{N}.
	\] 
\end{lemma}
\begin{proof}
	Let us write $\mathrm{SSy}_\mathcal{S}(\mathcal{M})$ as $(\mathcal{S}, \mathcal{A})$. Since
	\[
	(\mathcal{S}, \mathcal{A}) \models \mathrm{GBC} + \textnormal{``$\mathrm{Ord}$ is weakly compact''},
	\]
	we can apply Corollary \ref{end extend model of weakly compact to model with endos} to obtain a model $\mathcal{N}$, such that $\mathcal{S} \preceq^{\rnk} \mathcal{N}$ and for each $\nu \in \mathrm{Ord}^\mathcal{N} \setminus \mathcal{S}$, there is a topless rank-initial self-embedding $i_\nu$ of $\mathcal{N}$, which is contractive on $\mathcal{N} \setminus \mathcal{S}$ and which satisfies $\mathrm{image}(i_\nu) \subseteq \mathcal{N}_\nu$ and $\mathrm{Fix}(i_\nu) = \mathcal{S}$.
	
	Since $\mathcal{S} \preceq \mathcal{M}$ and $\mathcal{S} \preceq \mathcal{N}$, we have $\mathrm{Th}_{\Sigma_1^\mathcal{P}, \mathcal{S}}(\mathcal{M}) = \mathrm{Th}_{\Sigma_1^\mathcal{P}, \mathcal{S}}(\mathcal{N})$. So by Corollary \ref{Friedman cor}, there is a topless rank-initial embedding $j : \mathcal{M} \rightarrow \mathcal{N}$ which fixes $\mathcal{S}$ pointwise. Identify $\mathcal{M}$, pointwise, with the image of this embedding, and pick $\mu \in \mathrm{Ord}^\mathcal{M} \setminus \mathcal{S}$. Let $i = i_\mu$. It now follows from Proposition \ref{comp emb} that $i(\mathcal{M}) <^{\textnormal{rank-cut}} i(\mathcal{N})$, so	
	\[
	\mathcal{S} <^{\textnormal{rank-cut}} i(\mathcal{M}) <^{\textnormal{rank-cut}} i(\mathcal{N})  <^{\textnormal{rank-cut}} \mathcal{M} <^{\textnormal{rank-cut}} \mathcal{N},
	\] 
	as desired.
\end{proof}

\begin{thm}[Bahrami-Enayat-style]\label{characterize strongly topless substructure}
	Let $\mathcal{M} \models \mathrm{KP}^\mathcal{P}  + \Sigma_1^\mathcal{P}\textnormal{-Separation} + \textnormal{Choice}$ be countable and non-standard, and let $\mathcal{S}$ be a proper rank-initial substructure of $\mathcal{M}$. The following are equivalent:
	\begin{enumerate}[{\normalfont (a)}]
		\item\label{characterize strongly topless substructure emb} $\mathcal{S} = \mathrm{Fix}(i)$, for some self-embedding $i : \mathcal{M} \rightarrow \mathcal{M}$. 
		\item[{\normalfont (\ref{characterize strongly topless substructure emb}')}] $\mathcal{S} = \mathrm{Fix}(i)$, for some topless rank-initial self-embedding $i : \mathcal{M} \rightarrow \mathcal{M}$, which is contractive on $\mathcal{M} \setminus \mathcal{S}$.
		\item\label{characterize strongly topless substructure substr} $\mathcal{S}$ is a $\Sigma_1^\mathcal{P}$-elementary strong rank-cut of $\mathcal{M}$.
	\end{enumerate}
\end{thm}
\begin{proof}
	(\ref{characterize strongly topless substructure emb}) $\Rightarrow$ (\ref{characterize strongly topless substructure substr}): 
	It follows from Lemma \ref{rank-initial Fix is Sigma_1} that $\mathcal{S} \preceq_{\Sigma_1^\mathcal{P}} \mathcal{M}$. Next we observe that $\mathcal{S}$ is topless: It is assumed to be a proper substructure. If there were a least $\lambda \in \mathrm{Ord}^\mathcal{M} \setminus \mathcal{S}$, then by initiality of $i$ and $\mathcal{S} \subseteq \mathrm{Fix}(i)$, we would have $i(\lambda) = \lambda$, contradicting $\mathcal{S} \supseteq \mathrm{Fix}(i)$ and $\lambda  \not\in \mathcal{S}$. 
	
	Let $\alpha, \beta \in \mathrm{Ord}^\mathcal{M}$, with $\mathrm{Ord}^\mathcal{M} \cap \mathcal{S} \subseteq \alpha_\mathcal{M}$, and let $f \in \mathcal{M}$ code a function from $\alpha$ to  $\beta$ in $\mathcal{M}$. Note that $\mathrm{Ord}^\mathcal{M} \cap \mathcal{S} \subseteq i(\alpha)_\mathcal{M}$ and that $i(f)$ codes a function from $i(\alpha)$ to $i(\beta)$ in $\mathcal{M}$. By $\mathcal{S} = \mathrm{Fix}(i)$, for all $\zeta \in \mathrm{Ord}^\mathcal{M} \cap \mathcal{S}$ we have
	\[
	f(\zeta) \not\in \mathcal{S} \Leftrightarrow f(\zeta) \neq i(f)(\zeta). \tag{$\dagger$}
	\]
	We define a formula, with $f$ and $i(f)$ as parameters:
	\[
	\phi(\xi) \equiv \mathrm{Ord}(\xi) \wedge \forall \zeta < \xi . \big( f(\zeta) \neq i(f)(\zeta) \rightarrow f(\zeta) > \xi \big)
	\]
	Note that $\phi$ is $\Pi_1$ and that $\mathcal{M} \models \phi(\zeta)$, for all $\zeta \in \mathrm{Ord}^\mathcal{M} \cap \mathcal{S}$. So by $\Pi_1^\mathcal{P} \text{-Overspill}$ and toplessness of $\mathcal{S}$, there is $\mu \in \mathrm{Ord}^\mathcal{M} \setminus \mathcal{S}$ such that $\mathcal{M} \models \phi(\mu)$. Combining $\mathcal{M} \models \phi(\mu)$ with ($\dagger$), we have for all $\zeta \in \mathrm{Ord}^\mathcal{M} \cap \mathcal{S}$ that
	\[
	f(\zeta) \not\in \mathcal{S} \Rightarrow f(\zeta) > \mu.
	\]
	On the other hand, by $\mu \in \mathrm{Ord}^\mathcal{M} \setminus \mathcal{S}$, the converse is obvious. Hence, $\mathcal{S}$ is strongly topless.
	
	(\ref{characterize strongly topless substructure substr}) $\Rightarrow$ (\ref{characterize strongly topless substructure emb}'): By Lemma \ref{Kirby Paris forward}, we can apply Lemma \ref{embed in Gaifman model}. The restriction $i\restriction_\mathcal{M}$ of $i : \mathcal{N} \rightarrow \mathcal{N}$ (from Lemma \ref{embed in Gaifman model}) to $\mathcal{M}$, is a topless rank-initial self-embedding of $\mathcal{M}$ with fixed-point set $\mathcal{S}$, which is contractive on $\mathcal{M} \setminus \mathcal{S}$. 
\end{proof}

\begin{lemma}\label{Kirby Paris backward}
	Let $\mathcal{M} \models \mathrm{KP}^\mathcal{P} + \textnormal{Choice}$ be countable and let $\mathcal{S}$ be a rank-cut of $\mathcal{M}$. If $\mathrm{SSy}_\mathcal{S}(\mathcal{M}) \models \mathrm{GBC} + \text{``$\mathrm{Ord}$ is weakly compact''}$, then $\mathcal{S}$ is a strong rank-cut of $\mathcal{M}$.
\end{lemma}
\begin{proof}
	Write $\mathrm{SSy}_\mathcal{S}(\mathcal{M})$ as $(\mathcal{S}, \mathcal{A})$. Let $\mathcal{S} \preceq \mathcal{N}$ be the elementary extension obtained from Corollary \ref{end extend model of weakly compact to model with endos}, with a topless rank-initial self-embedding $i$ that is contractive on $\mathcal{N} \setminus \mathcal{S}$. It follows from Theorem \ref{characterize strongly topless substructure} that $\mathcal{S}$ is a strong rank-cut of $\mathcal{N}$. 
	
	Since $\mathcal{S}$ is rank-initial in $\mathcal{M}$, we have by Proposition \ref{emb pres}(\ref{emb pres Delta_0^P}) that $\mathrm{Th}_{\Sigma_1^\mathcal{P}, \mathcal{S}}(S) \subseteq \mathrm{Th}_{\Sigma_1^\mathcal{P}, \mathcal{S}}(M)$. So since $\mathcal{S} \preceq \mathcal{N}$, we have $\mathrm{Th}_{\Sigma_1^\mathcal{P}, \mathcal{S}}(N) \subseteq \mathrm{Th}_{\Sigma_1^\mathcal{P}, \mathcal{S}}(M)$, whence by Corollary \ref{Friedman cor}, there is a topless rank-initial embedding $j : \mathcal{N} \rightarrow \mathcal{M}$ which fixes $\mathcal{S}$ pointwise. Now it follows from Proposition \ref{comp emb} that $\mathcal{S}$ is also a strong rank-cut of $\mathcal{M}$.
\end{proof}

\begin{lemma}\label{rec sat elementary strongly topless}
	Let $\mathcal{M}$ be a countable recursively saturated model of $\mathrm{ZFC}$. If $\mathcal{S}$ is a strongly topless rank-initial elementary substructure of $M$, then $S \cong \mathcal{M}$, and a safe satisfaction relation on $\mathcal{S}$ is coded in $\mathrm{SSy}_\mathcal{S}(\mathcal{M})$.
\end{lemma}
\begin{proof}
We start by showing that a safe satisfaction relation on $\mathcal{S}$ is coded in $\mathrm{SSy}_\mathcal{S}(\mathcal{M})$. By the forward direction of Theorem \ref{rec sat char}, $\mathcal{M}$ is $\omega$-non-standard and admits a safe satisfaction relation $\mathrm{Sat}^\mathcal{M}$. Put $\mathrm{Sat}^\mathcal{S} = \mathrm{Sat}^\mathcal{M} \cap \mathcal{S}$. Note that $\mathrm{Sat}^\mathcal{S}$ is coded in $\mathcal{M}$, so the relation $\mathrm{Sat}^\mathcal{S}$ is coded as a class in $\mathrm{SSy}_\mathcal{S}(\mathcal{M})$. Since $\mathcal{S} \prec \mathcal{M}$, we have $\omega^\mathcal{S} = \omega^\mathcal{M}$. So since $\mathcal{S}$ is rank-initial in $\mathcal{M}$ and $\mathcal{M}$ is $\omega$-non-standard, $\mathcal{S}$ is $\omega$-non-standard. 

Since $\mathcal{S}$ is a strongly topless rank-initial elementary substructure of $M$, we have by Lemma \ref{Kirby Paris forward} that $\mathrm{SSy}_\mathcal{S}(\mathcal{M}) \models \mathrm{GBC}$. Therefore we have $(\mathcal{S}, \mathrm{Sat}^\mathcal{S}) \models \mathrm{ZF}(\mathcal{L}^0_\mathrm{Sat})$. To establishes that $\mathrm{Sat}^\mathcal{S}$ is a safe satisfaction relation on $\mathcal{S}$, it remains only to check that $(\mathcal{S}, \mathrm{Sat}^\mathcal{S}) \models \forall \sigma \in \bar\Sigma_n[x] . (\mathrm{Sat}(\sigma, x) \leftrightarrow \mathrm{Sat}_{\Sigma_n}(\sigma, x)$ and $(\mathcal{S}, \mathrm{Sat}^\mathcal{S}) \models \forall \pi \in \bar\Pi_n[x] . \mathrm{Sat}(\pi, x) \leftrightarrow \mathrm{Sat}_{\Pi_n}(\pi, x))$, for each standard $n \in \mathbb{N}$. But this follows from that $\mathrm{Sat}_{\Sigma_n}^\mathcal{S} = \mathrm{Sat}_{\Sigma_n}^\mathcal{M} \cap \mathcal{S}^2$ and $\mathrm{Sat}_{\Pi_n}^\mathcal{S} = \mathrm{Sat}_{\Pi_n}^\mathcal{M} \cap \mathcal{S}^2$, for each standard $n \in \mathbb{N}$, which in turn follows from that $\mathcal{S} \prec \mathcal{M}$. 

By the backward direction of Theorem \ref{rec sat char}, it now follows that $\mathcal{S}$ is recursively saturated. Since recursively saturated models are $\omega$-non-standard, we have by Lemma \ref{omega-topless existence} that $\mathrm{WFP}(\mathcal{M})$ is $\omega$-topless in $\mathcal{M}$ and in $\mathcal{S}$. So by Theorem \ref{rec sat iso thm}, $\mathcal{S} \cong \mathcal{M}$.
\end{proof}

\begin{thm}[Kaye-Kossak-Kotlarski-style]\label{characterize strongly topless substructure of rec sat}
	Let $\mathcal{M} \models \mathrm{ZFC} + V = \mathrm{HOD}$ be countable and recursively saturated, and let $\mathcal{S}$ be a proper rank-initial substructure of $\mathcal{M}$. The following are equivalent:
	\begin{enumerate}[{\normalfont (a)}]
		\item\label{characterize strongly topless substructure of rec sat emb} $\mathcal{S} = \mathrm{Fix}(i)$, for some automorphism $i : \mathcal{M} \rightarrow \mathcal{M}$. 
		\item\label{characterize strongly topless substructure of rec sat substr} $\mathcal{S}$ is a strongly topless elementary substructure of $\mathcal{M}$.
		\item[{\normalfont (\ref{characterize strongly topless substructure of rec sat substr}')}] $\mathcal{S}$ is a strongly topless elementary substructure of $\mathcal{M}$ isomorphic to $\mathcal{M}$.
	\end{enumerate}
\end{thm}
\begin{proof}
(\ref{characterize strongly topless substructure of rec sat emb}) $\Rightarrow$ (\ref{characterize strongly topless substructure of rec sat substr}'): Since $\mathcal{M} \models V = \mathrm{HOD}$, it has definable Skolem functions, whence Lemma \ref{rank-initial Fix is elementary} may be applied to the effect that $\mathcal{S} \prec \mathcal{M}$. Strong toplessness of $\mathcal{S}$ follows from the forward direction of Theorem \ref{characterize strongly topless substructure}. By Lemma \ref{rec sat elementary strongly topless}, we also have that $\mathcal{S} \cong \mathcal{M}$.

(\ref{characterize strongly topless substructure of rec sat substr}) $\Rightarrow$ (\ref{characterize strongly topless substructure of rec sat emb}): Let $(\mathcal{S}, \mathcal{A}) = \mathrm{SSy}_\mathcal{S}(\mathcal{M})$. Since $\mathcal{S} \prec \mathcal{M}$, we have $\omega^\mathcal{M} \in \mathcal{S}$. Now, by Lemma \ref{Kirby Paris forward}, $(\mathcal{S}, \mathcal{A}) \models \mathrm{GBC} + \text{``$\mathrm{Ord}$ is weakly compact''}$. Thus, we may apply Theorem \ref{Gaifman thm} (say with $\mathbb{L} = \mathbb{Q}$) to obtain a countable model $\mathcal{S} \prec \mathcal{N}$ with an automorphism $j : \mathcal{N} \rightarrow \mathcal{N}$ such that $\mathrm{Fix}(j) = \mathcal{S}$. By Lemma \ref{Kirby Paris backward}, $\mathcal{S}$ is strongly topless in $\mathcal{N}$.

Moreover, we have by Lemma \ref{rec sat elementary strongly topless} that $\mathcal{S}$ is recursively saturated with a safe satisfaction relation $\mathrm{Sat}^\mathcal{S}$ coded in $\mathcal{A}$. By part (\ref{elem}) of Theorem \ref{Gaifman thm}, $\mathrm{Sat}^\mathcal{S}$ corresponds to a safe satisfaction class $\mathrm{Sat}^\mathcal{N}$ on $\mathcal{N}$. So by Theorem \ref{rec sat char}, $\mathcal{N}$ is recursively saturated. Since $\mathcal{S}$ is strongly topless in both $\mathcal{M}$ and $\mathcal{N}$, it now follows from Theorem \ref{rec sat iso thm} that there is an isomorphism $k \in \llbracket \mathcal{M} \cong_\mathcal{S} \mathcal{N} \rrbracket$. The desired automorphism of $\mathcal{M}$ is now obtained as $i = k^{-1} \circ j \circ k$.
\end{proof}

\begin{thm}\label{strongly topless self-embedding iff GBC weakly compact}
	Suppose that $\mathcal{M} \models \mathrm{KP}^\mathcal{P} + \Sigma_1^\mathcal{P}\textnormal{-Separation} + \textnormal{Choice}$ is countable and non-standard. The following are equivalent:
	\begin{enumerate}[{\normalfont (a)}]
		\item\label{strongly topless self-embedding iff GBC weakly compact strongly topless} There is a strongly topless rank-initial self-embedding $i$ of $\mathcal{M}$.
		\item\label{strongly topless self-embedding iff GBC weakly compact ZFC} $\mathcal{M}$ expands to a model $(\mathcal{M}, \mathcal{A})$ of $\mathrm{GBC} + \text{``$\mathrm{Ord}$ is weakly compact''}$.
	\end{enumerate}
\end{thm}
\begin{proof}
	(\ref{strongly topless self-embedding iff GBC weakly compact strongly topless}) $\Rightarrow$ (\ref{strongly topless self-embedding iff GBC weakly compact ZFC}): If $i$ is strongly topless and rank-initial, then by the Lemma \ref{Kirby Paris forward}, we have $\mathrm{SSy}_{i(\mathcal{M})}(\mathcal{M}) \models \mathrm{GBC} + \textnormal{``$\mathrm{Ord}$ is weakly compact''}$. So since $\mathcal{M} \cong i(\mathcal{M})$, we have that $\mathcal{M}$ expands to a model of $\mathrm{GBC} + $``$\mathrm{Ord}$ is weakly compact''.

	(\ref{strongly topless self-embedding iff GBC weakly compact ZFC}) $\Rightarrow$ (\ref{strongly topless self-embedding iff GBC weakly compact strongly topless}): Expand $\mathcal{M}$ to a countable model $(\mathcal{M}, \mathcal{A})$ of $\mathrm{GBC} + $ ``$\mathrm{Ord}$ is weakly compact''. Let $\mathcal{N} \succ^\textnormal{rank-cut} \mathcal{M}$ be a model obtained from Theorem \ref{Gaifman thm} by putting $\mathbb{L}$ to be a countable linear order without a least element, e.g. $\mathbb{Q}$. By Theorem \ref{characterize strongly topless substructure}, $\mathcal{M}$ is strongly topless in $\mathcal{N}$. Note that $\mathrm{Th}(\mathcal{N}) = \mathrm{Th}(\mathcal{M})$ and $\mathrm{SSy}(\mathcal{N}) = \mathrm{SSy}(\mathcal{M})$. So by Theorem \ref{Friedman thm}, there is a rank-initial embedding $i : \mathcal{N} \rightarrow \mathcal{M}$. By Proposition \ref{comp emb}, it now follows that $i(\mathcal{M})$ is strongly topless in $\mathcal{M}$. 
\end{proof}


\begin{thebibliography}{9}


\bibitem[Bahrami, Enayat, 2018]{BE18}
S. Bahrami and A. Enayat.
{\em Fixed points of self-embeddings of models of arithmetic.}
Annals of Pure and Applied Logic, Volume 169, Issue 6 (June 2018), Pages 487-513.

\bibitem[Barwise, 1975]{Bar75}
J. Barwise.
Admissible Sets and Structures.
Springer-Verlag (1975).

\bibitem[Chang, Keisler, 1990]{CK90}
C. C. Chang and H. J. Keisler.
Model Theory. 
Elsevier Science Publishers (1990).


\bibitem[Ehrenfeucht, Mostowski, 1956]{EM56}
A. Ehrenfeucht and A. Mostowski.
{\em Models of axiomatic theories admitting automorphisms}.
Fundamenta Mathematicae, Vol. 43 (1956), pp. 50-68. 

\bibitem[Enayat, 2001]{Ena01}
A. Enayat.
{\em Power like models of set theory.}
The Journal of Symbolic Logic, vol. 66 , no. 2 (2001), pp. 1766-1782.

\bibitem[Enayat, 2004]{Ena04}
A. Enayat.
{\em Automorphisms, Mahlo Cardinals, and NFU}.
In Nonstandard Models of Arithmetic and Set Theory, Contemporary Mathematics, vol. 361, American Mathematical Society (2004), pp. 37-59.

\bibitem[Enayat, 2007]{Ena07}
A. Enayat.
{\em Automorphisms of models of arithmetic: a unified view.}
Annals of Pure and Applied Logic, vol. 145 (2007), pp. 16-36. 

\bibitem[Enayat, Hamkins, 2017]{EH17}
A. Enayat, J. D. Hamkins.
{\em ZFC proves that the class of ordinals is not weakly compact for definable classes}.
arXiv:1610.02729 [math.LO].
Forthcoming in The Journal of Symbolic Logic.

\bibitem[Enayat, Kaufmann, McKenzie, 2017]{EKM17}
A. Enayat, M. Kaufmann, Z. McKenzie.
{\em Iterated ultrapowers for the masses}.
Archive for Mathematical Logic, Published online (October, 2017).

\bibitem[Enayat, Kaufmann, McKenzie, 2018]{EKM18}
A. Enayat, M. Kaufmann, Z. McKenzie.
{\em Largest initial segments pointwise ﬁxed by automorphisms of models of set theory}.
Archive for Mathematical Logic, v. 57, i. 1-2 (February, 2018).

\bibitem[Friedman, 1973]{Fri73}
H. Friedman.
{\em Countable models of set theories.}
In A.R.D. Mathias and H. Rogers, eds., Cambridge Summer School in Mathematical Logic.
Springer-Verlag (1973).

\bibitem[Gaifman, 1976]{Gai76}
H. Gaifman. 
{\em Models and types of arithmetic}. 
Annals of Mathematical Logic, vol. 9 (1976), pp. 223-306.

\bibitem[Gorbow, 2018]{Gor18}
P. K. Gorbow.
Self-similarity in the foundations.
Ph.D. Thesis. Acta Philosophica Gothoburgensia 32. University of Gothenburg (2018).
On-line: \url{https://arxiv.org/abs/1806.11310}

\bibitem[Hamkins, 2013]{Ham13}
J. D. Hamkins. {\em Every countable model of set theory embeds into its own constructible universe.} 
Journal of Mathematical Logic, Vol. 13 (2013).

\bibitem[Jech, 2002]{Jeh02}
T. Jech.
Set Theory.
Springer (2002).

\bibitem[Jensen, 1969]{Jen69}
R. B. Jensen. {\em On the consistency of a slight (?) modification of Quine's NF}. Synthese, Vol. 19 (1969), pp. 250-263.

\bibitem[Kaye, 1991]{Kay91}
R. Kaye.
Models of Peano Arithmetic.
Clarendon Press (1991).

\bibitem[Kaye, Kossak, Kotlarski, 1991]{KKK91}
R. Kaye, R. Kossak, and H. Kotlarski. 
{\em Automorphisms of recursively
saturated models of arithmetic}.
Annals of Pure and Applied Logic 55 (1991), pp. 67-99.

\bibitem[Kirby, Paris, 1977]{KP77}
L. Kirby and J. Paris. 
{\em Initial segments of models of Peano's axioms}, in Lecture Notes in Mathematics, Vol. 619, Springer-Verlag (1977), pp. 211-226.

\bibitem[Kossak, Kotlarski, 1988]{KK88}
R. Kossak and H. Kotlarski.
{\em Results on automorphisms of recursively saturated
models of PA}.
Fund. Math. 129 (1988), pp. 9-15.

\bibitem[Kossak, Schmerl, 2006]{KS06}
R. Kossak and J. Schmerl.
The Structure of Models of Peano Arithmetic.
Oxford Science Publications (2006).

\bibitem[Kunen, 2013]{Kuh13}
K. Kunen.
Set Theory.
College Publications (2013).

\bibitem[L\'evy, 1965]{Lev65}
A. L\'evy.
A hierarchy of formulas in set theory.
American Mathematical Society (1965).

\bibitem[Mathias, 2001]{Mat01}
A. R. D. Mathias.
{\em The strength of Mac Lane set theory.}
Annals of Pure and Applied Logic, vol. 110 (2001), pp. 107-234.

\bibitem[Ressayre, 1987a]{Res87a}
J. -P. Ressayre.
{\em Mod\`eles non Standard et Sous-Syst\`emes Remarquables de ZF.}
In Mod\`eles non Standard en arithm\'etique et th\'eorie des ensambles.
Publications math\'ematiques de l'Universit\'e Paris VII, no. 22, U.E.R. de Math\'ematiques, Paris (1987), pp. 47-147.

\bibitem[Ressayre, 1987b]{Res87b}
J. -P. Ressayre.
{\em Non standard universes with strong embeddings, and their finite approximations.}
Logic and Combinatorics. Contemp. Math. 65. Amer. Math. Soc. Providence, RI (1987), pp. 333-358.

\bibitem[Schlipf, 1978]{Sch78}
J. S. Schlipf. 
{\em Toward Model Theory Through Recursive Saturation}.
The Journal of Symbolic Logic, vol. 43, no. 2 (Jun., 1978), pp. 183-206.

\bibitem[Takahashi, 1972]{Tak72}
M. Takahashi.
{\em {$\tilde{\Delta}_1$}-definability in set theory.} 
In W. Hodges, ed., Conference in mathematical logic \textemdash \text{} London '70, Springer Lecture Notes in Mathematics 255 (1972) pp. 281-304.

\bibitem[Wilkie, 1973]{Wil73}
A. J. Wilkie.
Models of Number Theory. 
Doctoral dissertation, University of London (1973).

\bibitem[Wilkie, 1977]{Wil77}
A. J. Wilkie.
{\em On the theories of end-extensions of models of arithmetic.}
In A. Lachlan, M. Srebrny and A. Zarach, eds., Set Theory and Model Theory V, Springer-Verlag, Heidelberg (1977).




\end{thebibliography}
\end{document}